\documentclass[cmp]{svjour}
\usepackage{amsmath}
\usepackage{latexsym}
\usepackage{mathrsfs}
\usepackage{amssymb,exscale}
\usepackage{hyperref}
\usepackage{verbatim} % for commenting out big chunks of text
\usepackage{tikz}\usetikzlibrary{calc,decorations.pathmorphing,decorations.markings,decorations.pathreplacing,patterns}

\def\ourtrt{\mathscr{T}_t}
\def\ourtr0{\mathscr{T}_0}

\def\sn{{\rm sn}}

\def\TL{{\mbox{TL}}}

\def\ra{\rightarrow}

\def\PP{{\mathcal P}}
\def\Q{{\mathcal Q}}

\let\SectMark\S

\def\S{\mathbb{S}}

\def\e{{\epsilon}}

\def\<{\langle}
\def\>{\rangle}

\newcommand\Tr{{\operatorname{Tr}}}
\newcommand\tr{{\operatorname{tr}}}
\newcommand\mytau{{\overline{\operatorname{tr}}}}

\newcommand\mnote[1]{} %off
\newcommand\be{\begin{equation*}}

\newcommand\ee{\end{equation*}}

\newcommand\ben{\begin{equation}}
\newcommand\een{\end{equation}}
\newcommand\bes{\begin{eqnarray*}}
\newcommand\ees{\end{eqnarray*}}

\newcommand{\sm}{{\raise0.3ex\hbox{$\scriptstyle \setminus$}}}

\renewcommand{\phi}{\varphi}

\def\CHI{\mathchoice%
{\raise2pt\hbox{$\chi$}}%
{\raise2pt\hbox{$\chi$}}%
{\raise1.3pt\hbox{$\scriptstyle\chi$}}%
{\raise0.8pt\hbox{$\scriptscriptstyle\chi$}}}
\def\smalloplus{\raise1pt\hbox{$\,\scriptstyle \oplus\;$}}
\def\circled#1{
	\begin{tikzpicture}
		[inner sep=0.3mm, circ/.style={circle,draw,minimum size=1mm}]
		\node at (0,0) [circ] {$\scriptscriptstyle #1$}; 
	\end{tikzpicture}
}
\def\copyrightsign{\hbox{{\circled{c}}}}
\def\vertexvw#1(#2){%vertical white
\begin{scope}
\node[draw,rectangle,rounded corners,minimum size=0.8cm] (#1) at (#2) {};
\clip[rounded corners] (#1.south west) rectangle (#1.north east); % nodes can't be clipped which is annoying
\draw[fill=gray] (#1.south west) .. controls (#1) .. (#1.north west);
\draw[fill=gray] (#1.south east) .. controls (#1) .. (#1.north east);
\end{scope}
}
\def\vertexvb#1(#2){%vertical black
\begin{scope}
\node[draw,rectangle,rounded corners,minimum size=0.8cm] (#1) at (#2) {};
\clip[rounded corners] (#1.south west) rectangle (#1.north east); % nodes can't be clipped which is annoying
\draw[fill=gray] (#1.south west) .. controls (#1) .. (#1.north west) -- (#1.north east) .. controls (#1) .. (#1.south east);
\end{scope}
}
\def\vertexhw#1(#2){%horiz white
\begin{scope}
\node[draw,rectangle,rounded corners,minimum size=0.8cm] (#1) at (#2) {};
\clip[rounded corners] (#1.south west) rectangle (#1.north east); % nodes can't be clipped which is annoying
\draw[fill=gray] (#1.south west) .. controls (#1) .. (#1.south east);
\draw[fill=gray] (#1.north west) .. controls (#1) .. (#1.north east);
\end{scope}
}
\def\vertexhb#1(#2){%horiz black
\begin{scope}
\node[draw,rectangle,rounded corners,minimum size=0.8cm] (#1) at (#2) {};
\clip[rounded corners] (#1.south west) rectangle (#1.north east); % nodes can't be clipped which is annoying
\draw[fill=gray] (#1.south west) .. controls (#1) .. (#1.south east) -- (#1.north east) .. controls (#1) .. (#1.north west);
\end{scope}
}
\def\cupb#1(#2){% single black cup
\node[rectangle,minimum size=0.8cm] (#1) at (#2) {};
\draw[rounded corners] ([xshift=-0.2cm] #1.south west) rectangle ([xshift=0.2cm] #1.north east);
\draw[fill=gray] (#1.north west) .. controls ++(0,-0.6) and ++(0,-0.6) .. (#1.north east);
}
\def\cupcupb#1(#2){% double black cup
\node[rectangle,minimum size=0.8cm] (#1a) at ([xshift=-0.6cm]#2) {};
\node[rectangle,minimum size=0.8cm] (#1b) at ([xshift=0.6cm]#2) {};
\draw[rounded corners] ([xshift=-0.2cm] #1a.south west) rectangle ([xshift=0.2cm] #1b.north east);
\draw[fill=gray] (#1a.north west) .. controls ++(0,-0.6) and ++(0,-0.6) .. (#1a.north east);
\draw[fill=gray] (#1b.north west) .. controls ++(0,-0.6) and ++(0,-0.6) .. (#1b.north east);
}
\def\halfvertexwl#1(#2){%
\node[rectangle,minimum height=0.8cm,minimum width=0.4cm] (#1) at (#2) {};
\begin{scope}
\clip[draw] decorate[decoration={zigzag,segment length=0.12cm,amplitude=1pt}] { ([yshift=-0.04cm]#1.east) --  (#1.south east) } [rounded corners] -- (#1.south west) -- (#1.north west) [sharp corners] -- (#1.north east) decorate[decoration={zigzag,segment length=0.12cm,amplitude=1pt}] { -- ([yshift=0.04cm]#1.east)};
\draw[fill=gray] (#1.north west) .. controls (#1.east) .. (#1.south west);
\end{scope}
}
\def\halfvertexwr#1(#2){%
\node[rectangle,minimum height=0.8cm,minimum width=0.4cm] (#1) at (#2) {};
\begin{scope}
\clip[draw] decorate[decoration={zigzag,segment length=0.12cm,amplitude=1pt}] { ([yshift=-0.04cm]#1.west) --  (#1.south west) } [rounded corners] -- (#1.south east) -- (#1.north east) [sharp corners] -- (#1.north west) decorate[decoration={zigzag,segment length=0.12cm,amplitude=1pt}] { -- ([yshift=0.04cm]#1.west)};
\draw[fill=gray] (#1.north east) .. controls (#1.west) .. (#1.south east);
\end{scope}
}
\def\halfvertexbl#1(#2){%
\node[rectangle,minimum height=0.8cm,minimum width=0.4cm] (#1) at (#2) {};
\begin{scope}
\clip[draw] decorate[decoration={zigzag,segment length=0.12cm,amplitude=1pt}] { ([yshift=-0.04cm]#1.east) --  (#1.south east) } [rounded corners] -- (#1.south west) -- (#1.north west) [sharp corners] -- (#1.north east) decorate[decoration={zigzag,segment length=0.12cm,amplitude=1pt}] { -- ([yshift=0.04cm]#1.east)};
\draw[fill=gray] (#1.north west) .. controls (#1.east) .. (#1.south west) -- ([xshift=2pt]#1.south east) -- ([xshift=2pt]#1.north east);
\end{scope}
}
\def\halfvertexbr#1(#2){%
\node[rectangle,minimum height=0.8cm,minimum width=0.4cm] (#1) at (#2) {};
\begin{scope}
\clip[draw] decorate[decoration={zigzag,segment length=0.12cm,amplitude=1pt}] { ([yshift=-0.04cm]#1.west) --  (#1.south west) } [rounded corners] -- (#1.south east) -- (#1.north east) [sharp corners] -- (#1.north west) decorate[decoration={zigzag,segment length=0.12cm,amplitude=1pt}] { -- ([yshift=0.04cm]#1.west)};
\draw[fill=gray] (#1.north east) .. controls (#1.west) .. (#1.south east) -- ([xshift=-2pt]#1.south west) -- ([xshift=-2pt]#1.north west);
\end{scope}
}
\def\fixbb{%fix bb to current plus margins
\useasboundingbox ([xshift=-10pt,yshift=-2pt] current bounding box.south west) ([xshift=10pt,yshift=2pt] current bounding box.north east);
}
\tikzset{connect/.style={fill=gray,fill opacity=0.2,draw,densely dashed,very thin}}
\tikzset{braid/.style={double,double distance=0.1cm}}
\begin{document}

\author{A. Guionnet\inst{1}\thanks{A.G. and P.Z.-J. are supported by the ANR project ANR-08-BLAN-0311-01.}
\and V.~F.~R. Jones\inst{2}\thanks{V.J.'s research is supported by NSF grant DMS-0856316} \and D. Shlyakhtenko\inst{3}\thanks{D.S.'s research is supported by NSF grant DMS-0900776} \and P. Zinn-Justin\inst{4}}
\institute{UMPA, CNRS  UMR 5669, ENS Lyon, 46 all\'ee d'Italie,
69007 Lyon, France. \email{aguionne@ens-lyon.fr}
\and
Department of Mathematics, UC Berkeley, Berkeley, CA 94720. \email{vfr@math.berkeley.edu}
\and  
Department of Mathematics, UCLA, Los Angeles, CA 90095. \email{shlyakht@math.ucla.edu.}
\and
UPMC Univ Paris 6, CNRS UMR 7589, LPTHE,
75252 Paris Cedex, France. \email{pzinn@lpthe.jussieu.fr}
}

\title{Loop models, random matrices and  planar algebras}
\date{June 1, 2012}
\maketitle
\begin{abstract}We define matrix models that converge to the generating 
functions of a wide variety of  loop models with fugacity taken in 
sets with an accumulation point. The latter can also be seen as moments of
a non-commutative law on a subfactor planar algebra.  We apply this construction to compute the generating functions of the Potts model on a random planar map.
\end{abstract}

\section{Introduction}
Loop models  naturally appear in a variety of statistical models  where the loops represent the configuration of boundaries of some random regions. In this article, we restrict ourselves to loop models on random planar maps.  Perhaps  the most famous of these is  the  
so-called  $O(n)$ loop model  which can be described as follows. 

Consider the two-sphere $S^2$, and fix $r$ disjoint disks $\mathscr{D}_j: 1\leq j\leq r$ inside of $S^2$.  Each disk is given an even number of boundary points, one of which is  marked.  By a {\em tangle} we mean (the isotopy class of) any possible collection of non-intersecting strings in $S^2\setminus(\mathscr{D}_1\cup \dots\cup\mathscr{D}_r)$ joining the boundary points. 
$$
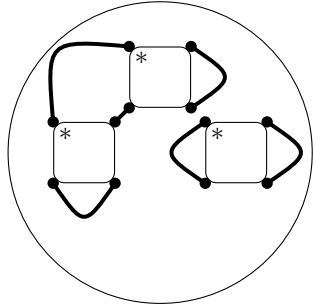
\tikzset{vertex/.style={circle,fill=black,inner sep=1.5pt}}
\tikz[baseline=0,label distance=-2mm]{
\begin{scope}
\node[draw,circle, minimum size=4cm](c) at (0,0){};
\node[draw,rectangle,rounded corners,minimum size=0.8cm] (a) at (-1,0) {};
\node[draw,rectangle,rounded corners,minimum size=0.8cm] (b) at (1,0) {};
\node[draw,rectangle,rounded corners,minimum size=0.8cm] (d) at (0,1) {};
\draw[line width=1.5pt](a.south west) node[vertex] {} .. controls (-1,-1)  .. (a.south east) node[vertex] {};
\draw[line width=1.5pt](a.north east) node[vertex] {}.. controls (-0.5,0.5) .. (d.south west) node[vertex] {};
\draw[line width=1.5pt](a.north west) node[vertex,label={below right:$*$}] {}.. controls (-1.5,1.5) .. (d.north west) node[vertex,label={below right:$*$}] {};
\draw[line width=1.5pt](b.north west) node[vertex,label={below right:$*$}] (bnw) {} .. controls (0,0) .. (b.south west) node[vertex] {};
\draw[line width=1.5pt](b.north east) node[vertex] {} .. controls (2,0) .. (b.south east) node[vertex] {};
\draw[line width=1.5pt](d.north east) node[vertex] {} .. controls (1,1) .. (d.south east) node[vertex] {};
\end{scope}}
 $$

By a {\em vertex} (also called a Temperley--Lieb diagram) \label{firstdefof:TL} we mean an (isotopy class of) arbitrary non-intersecting collection of strings drawn inside a disk with a given even number of boundary points (one of which is marked).  

By a {\em configuration $P$ built on vertices $D_1,\dots,D_r$} we mean a tangle $Q$ into which the vertices $D_1,\dots,D_r$ have been inserted (so that $D_j$ is inserted into $\mathscr{D}_j$ in a way that boundary points match, and marked boundary points match).  In the $O(n)$ model, the vertices are all assumed to be copies of the same vertex $D$:
Let $D$ be given by the following picture:
$$^*
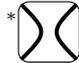
\tikz[baseline=0]{
\begin{scope}
\node[draw,rectangle,rounded corners,minimum size=0.8cm] (a) at (0,0) {};
\clip[rounded corners] (a.south west) rectangle (a.north east); % nodes can't be clipped which is annoying
\draw[line width=1.5pt](a.south west) .. controls (a) .. (a.north west);
\draw[line width=1.5pt](a.south east) .. controls (a) .. (a.north east);
\end{scope}}
$$  The outer boundary is represented by a thin line.  The boundary contains $4$ {\em boundary points}, which are joined inside of the disk by non-intersecting strings (represented by thick lines).  One of the boundary points is marked by a $*$ to distinguish it from the others.
%\begin{figure}[htbp]
%\begin{center}
%\resizebox{1cm}{!}{\input{twocup.pstex_t}}
%\end{center}
%\end{figure}
A possible configuration is drawn below:
$$
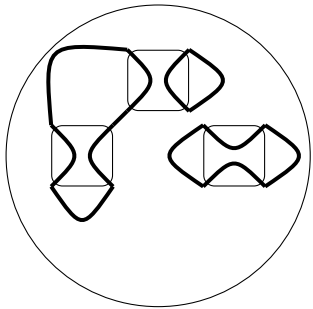
\tikz[baseline=0]{
\begin{scope}
\node[draw,circle, minimum size=4cm](c) at (0,0){};
\node[draw,rectangle,rounded corners,minimum size=0.8cm] (a) at (-1,0) {};
%\clip[rounded corners] (a.south west) rectangle (a.north east); % nodes can't be clipped which is annoying
\draw[line width=1.5pt](a.south west) .. controls (a) .. (a.north west);
\draw[line width=1.5pt](a.south east) .. controls (a) .. (a.north east);
\node[draw,rectangle,rounded corners,minimum size=0.8cm] (b) at (1,0) {};
%\clip[rounded corners] (a.south west) rectangle (a.north east); % nodes can't be clipped which is annoying
\draw[line width=1.5pt](b.north west) .. controls (b) .. (b.north east);
\draw[line width=1.5pt](b.south west) .. controls (b) .. (b.south east);
\node[draw,rectangle,rounded corners,minimum size=0.8cm] (d) at (0,1) {};
%\clip[rounded corners] (d.south west) rectangle (d.north east); % nodes can't be clipped which is annoying
\draw[line width=1.5pt](d.south west) .. controls (d) .. (d.north west);
\draw[line width=1.5pt](d.south east) .. controls (d) .. (d.north east);
\draw[line width=1.5pt](a.south west) .. controls (-1,-1)  .. (a.south east);
\draw[line width=1.5pt](a.north east) .. controls (-0.5,0.5) .. (d.south west);
\draw[line width=1.5pt](a.north west) .. controls (-1.5,1.5) .. (d.north west);
\draw[line width=1.5pt](b.north west) .. controls (0,0) .. (b.south west);
\draw[line width=1.5pt](b.north east) .. controls (2,0) .. (b.south east);
\draw[line width=1.5pt](d.north east) .. controls (1,1) .. (d.south east);
\end{scope}}
$$

Fix a number $n$ (weight loop, sometimes also denoted $\delta$ and called {\em fugacity}) and associate to a configuration $P$ the {\em value of $P$}, computed as follows.  For each $\mathscr{D}_j$, remove its outer boundary.  We are left with a collection of closed loops (formed by strings outside and inside $\mathscr{D}_1,\dots,\mathscr{D}_r$).  The value of $P$ is then given by $$\textrm{value}(P)=n^{\#\textrm{ loops in P}}.$$  

In studying this loop model, one is interested in understanding the partition function
\begin{equation}\label{ftrbeta}
f_n(t) = \sum_{r=0}^\infty
\frac{t^{r}}{r!} \sum_{P\in
P(r)}\textrm{Value}(P)
\end{equation}
where the sum is taken over the set $P(r)$ of all configurations  built on  $r$ copies of  $D$ (note here that the vertices and boundary points are labeled, so configurations corresponding to different matchings of the labeled boundaries of vertices and the tangle give rise to different terms in the summation). 

One can consider a 
  slightly more complicated situation, where exactly one of the vertices $T$ (sometimes called the external face) is different from $D$.  For example, $T$ can be taken to be 
 $$^*
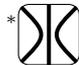
\tikz[baseline=0]{
\begin{scope}
\node[draw,rectangle,rounded corners,minimum size=0.8cm] (a) at (0,0) {};
\clip[rounded corners] (a.south west) rectangle (a.north east); % nodes can't be clipped which is annoying
\draw[line width=1.5pt](a.south west) .. controls (a) .. (a.north west);
\draw[line width=1.5pt](a.north)..controls(a)..(a.south);
\draw[line width=1.5pt](a.south east) .. controls (a) .. (a.north east);
\end{scope}}
$$

We can  then consider the observable
\begin{equation}\label{trbeta}
\mathscr{T}^n_t( T): = \sum_{r=0}^\infty
\frac{t^{r}}{r!} \sum_{P\in
P(r,T)}
\textrm{Value}(P)
\end{equation}
where $P(r,T)$ is the configuration based on $r$ %A
 copies of $D$ and one copy of $T$.

The loop model we described has been widely studied, in part because of its connection with the critical Potts model.
In the case that $n$ takes an integer value, the observable $\mathscr{T}^n_t$ can be related to a random matrix model as follows.
 Let $A_i, 1\le i\le n,$ be $n$ $M\times M$ independent GUE matrices whose entries have variance $M^{-1}$. For a vertex $T$, enumerate its boundary points starting with the marked one clockwise.  Then consider the pair partition $$\sim_T:\qquad i\sim_T j \quad\textrm{iff points $(i,j)$ of $T$ are linked by a string inside $T$.}$$  
 For a Temperley--Lieb element $T$ with $2p$ boundary points, define the polynomial $P_T$   in $n$ non-commutative indeterminates  $(X_1,\ldots, X_n)$ by
\begin{equation}\label{defpt}P_T(X_1,\ldots,X_n)=\sum_{\substack{1\le i_1,\ldots,i_{2p}\le n \\ i_l=i_k\textrm{ if } l\sim_T k} } X_{i_1}X_{i_2}\cdots X_{i_{2p}}\,. 
\end{equation}
Then,  Wick calculus (see also its large $M$ limit given by Voiculescu's asymptotic freeness \cite{dvv:random,dvv:book}) shows that
$$\lim_{M\ra\infty} \mathbb E\left[\frac{1}{M}\Tr\left( P_T(A_1,\cdots,A_n)\right)\right]
=\mathscr{T}_{t=0}^{n}(T)$$
where $\Tr$ is the standard unnormalized trace, $\Tr(A)=\sum A_{ii}$. %A
Moreover, one can generalize this relation to negative real numbers
$t$  which are sufficiently close to zero by instead considering random matrices chosen according to a certain Gibbs measure (rather than the Gaussian measure).
For each $M=1,2\dots$, consider the random choice of $n$ matrices $A_1,\dots,A_n$ of size $M\times M$ according to the probability measure
$$\label{eq:GibbsMeasureRM}
d\mu_{t}^{n,M} (A_1,\ldots,A_n)= \frac{1}{Z_N} \exp\left\{-M  \operatorname{Tr}\left(
\sum_{j=1}^n A_j^2 - t \big(\sum_{i=1}^n A_i^2 \big) ^2\right)\right\} dA_1\dots dA_n
$$ 
where $dA_1\dots dA_n$ denotes the Lebesgue measure on the space of $n$-tuples of Hermitian matrices.  This is the so-called $O(n)$ matrix model.  The specific choice of the potential $- t \big(\sum_{i=1}^n A_i^2 \big) ^2 $ corresponds precisely to our way of representing by a polynomial  the diagram $D = {}^*\tikz[baseline=0]{
\begin{scope}
\node[draw,rectangle,rounded corners,minimum size=0.8cm] (a) at (0,0) {};
\clip[rounded corners] (a.south west) rectangle (a.north east); % nodes can't be clipped which is annoying
\draw[line width=1.5pt](a.south west) .. controls (a) .. (a.north west);
\draw[line width=1.5pt](a.south east) .. controls (a) .. (a.north east);
\end{scope}}
$ appearing in the definition of the loop model observable $\mathscr{T}$.  

The random matrix model then computes the observable: according to \cite{brezin-itzykson-parisi-zuber:planarDiagrams} (see also \cite{guionnet-edouard:combRM} for the derivation of the asymptotics of multi-matrix models) if $t<0$ is small enough
for any polynomial $P$ in the set of  polynomial of  $n$ non-commutative variables
$$\lim_{M\ra\infty} \int \left[\frac{1}{M}\Tr\big( P(A_1,\cdots,A_n)\big)\right]d\mu_{t}^{n,M} (A_1,\ldots,A_n)=\mytau_t^n(P)$$
where,
if $T$ is a Temperley--Lieb element, and $P_T$ is given by \eqref{defpt}
$$\mytau_t^n(P_T)=
{\mathscr{T}}^{n}_{t}(T)\,.$$ %A

This representation of the $O(n)$ model as the asymptotics of a matrix model is useful in many respects. It was used to actually compute
the partition function $f_n(t)$, see \cite{DK88,Kos89,KS92,EK95,BE09}.

Finally, we mention that in the context of Voiculescu's free probability theory, \label{freeProbabilityConnection} the observables $\mytau^n_t$ and  $\mathscr{T}$ have an interpretation as a {\em non-commutative law}. Indeed, $\mytau^n_t$ is a tracial linear functional on the set of polynomials in $n$ non-commutative self-adjoint variables equipped with its natural product. For $\mathscr{T}$,
 the observation is that polynomials of the form $P_T$ form a subalgebra of the $*$-algebra of non-commutative polynomials in $n$ self-adjoint variables.   Multiplication of polynomials $P_T$, $P_S$ corresponds to combining Temperley--Lieb diagrams in a disjoint fashion, $T\wedge S$, and taking adjoints amounts to flipping the diagram along a vertical axis.  The observable function $\mathscr{T}$ then gives us a tracial linear functional on this algebra.

  Remarkably, for small  real values of $t$, this linear functional is a state: $\mytau_t^n (P P^*)\geq 0$ for any $P$, a property which therefore also holds for $\mathscr{T}_t^n$.  This means that for ``self-adjoint'' diagrams $T$, the values of the observable $\{\mathscr{T}^n_t (T^k)\}_{k\geq 0}$ are moments of a probability measure, which turns out to be a useful observation. %A 

The non-commutative law $\tau = \mytau_t^n$ can be called a {\em free Gibbs state}, the free-probabi\-lity analog of a Gibbs measure.  The combinatorics of this law can be compactly encoded by the so-called Schwinger--Dyson (or loop) equation, which is a free analog of integration by parts, and for which the linear functional  $\tau=\mytau_t^n$ is the unique (suitably bounded) solution: 
\begin{equation}
\tau (X_j P) = \tau \otimes \tau (\partial_j P) - t \ \tau (P \mathscr{D}_j P_D), \qquad\forall P \label{eq:SchwingerDysonIntro}
\end{equation}%A
where $\partial_j$ and $\mathscr{D}_j$ are certain differential operators occurring in free probability theory \cite{dvv:entropy5,dvv:entropysurvey,voiculescu:conjectureAboutPotentials,guionnet-edouard:combRM,alice:StFlour}.

To summarize our discussion of the $O(n)$ model, we observe that we have (well-known) equality of the following values, for any $\TL$ diagram $T$:
\begin{enumerate}
\item The value of the observable $\mathscr{T}_t^n(T)$ obtained by summing over all possible configurations certain values associated to each configuration; it is defined for small values of   $t$ and all (possibly non-integer) $n$;
\item The value of the free Gibbs law $\tau(P_T)$; here $\tau$ was defined only for integer $n$ as the unique trace on the algebra non-commutative polynomials in $n$ variables which solves the Schwinger--Dyson equation \eqref{eq:SchwingerDysonIntro};
\item The limit  $\mytau_t^n(P_T)$  as $M\to\infty$ of expected values $\mathbb{E}_{\mu_t^{n,M}}( M^{-1}\Tr(P_T))$; it is  defined only for integer $n$ averaging with respect to certain Gibbs measures on $n$-tuples of $M\times M$ self-adjoint matrices.
\end{enumerate}

The goal of this paper is to generalize and extend these three objects by making use of the notion of a {\em Jones' subfactor planar algebra}.  The planar algebra is an abstract object that simultaneously encodes the value of the fugacity $n$ and the possible choices of the vertices $T$, $D$ in the loop model.  

Thus given a planar algebra $\mathcal{P}$ and having chosen elements $S,S_1,\dots,S_n$ of $\mathcal{P}$, we construct:
\begin{itemize}
\item A observable $\mathscr{T}_t$ for the associated loop model;
\item A family of Gibbs measures $\mu_{t}^{M}$ defined on certain spaces of matrices with size going to infinity with $M$;
\item An algebra $A=\{P_T : T\in \mathcal{P}\}$ with the  linear functional
 $\mathscr{T}_{t} : A\to \mathbb C$  satisfying a certain Schwinger--Dyson equation and a map $A\ni T\mapsto P_T$ into our space of matrices.
 \end{itemize}
Moreover, we prove that these three objects are related by equalities: $$\mathscr{T}_t(S) = \lim_{M\to\infty}\frac{1}{M}\int {\rm Tr}(P_S) d\mu_t^{M}.$$

In other words, we introduce generalizations of loop models whose vertices are from an arbitrary planar algebra $\mathcal P$; find analogs of the Schwinger--Dyson equation and of free Gibbs states on graded algebras associated to $\mathcal P$; and construct random matrix models that give the appropriate free Gibbs state in the large-$M$ limit. 

In particular, by specializing to the Temperley--Lieb planar algebra $\mathcal{P}=\TL$, our work allows one to represent the observable $\mathscr{T}_t^n$ as random matrix integrals for non-integer values of $n\in \{2\cos (\pi/p): p\geq 3\}\cup[2,+\infty)$. %A
  This is important because $\mathscr{T}^n_t$ is analytic in $n$ on a certain domain (see Lemma \ref{lemana} in the appendix) and thus knowing it on a set with an accumulation point determines the function. Our work thus put on a firm mathematical ground the (often assumed) analytic extension of $\mathscr{T}^n_t$ as a function of $n$, and this in a remarkably general setting.

The passage to general planar algebras also naturally leads us to consider shaded tangles and vertices.  This in particular allows us to consider a kind of $U(n)$ model (closely related to the $O(n)$ model) in which configurations are shaded.  This model corresponds to the Potts model, rather than just the critical Potts model described by the $O(n)$ model.  Using our random matrix model, we are able to compute a observable of the $U(n)$ model, giving a refinement of the computation in the $O(n)$ case.

In the next three sections, we give a few details about our construction.

\subsection{Loop models associated to arbitrary Jones planar algebras}
\label{sect:ArtitraryLoops}
We begin by noting that it is possible to define an analog of the loop model if one permits the vertices $T$ and $D$ to lie in an arbitrary {\em Jones subfactor planar algebra}.  We will review the definition of a planar algebra later in the paper; to give an informal idea of the connection we again consider a planar tangle, but this time also requiring that there be  %A
 a checkerboard shading of the regions between the strings.  The main axiom of a subfactor planar algebra then specifies that given such a tangle and $k$ arbitrary elements of the planar algebra, one obtains a number, which is the value of the ``configuration obtained by gluing in the vertices into the tangle''.   Thus if we choose as our vertices an element $T$ and $n$ copies of an element $D$ of the planar algebra, then one can make sense of the observable $\mathscr{T}_t(T,D)$ %A
 just as before, by forming the summation over all tangles of the values of the resulting configurations.  

A particular example of a planar algebra is the (unshaded) Temperley--Lieb planar algebra $TL(n)$, whose elements are exactly the Temperley--Lieb diagrams (the fugacity $n$ is part of the planar algebra structure).  Considering this planar algebra leads to the observable $\mathscr{T}_t^n$ we described earlier in the introduction.

Another example is the shaded Temperley--Lieb model, in which we require that the vertices be Temperley--Lieb diagrams with a fixed checkerboard shading of the regions between strings.

Requiring tangles to be shaded comes naturally from the theory of  subfactors: 
the planar algebra of a subfactor $N\subseteq M$ is necessarily shaded, the shaded and unshaded 
regions corresponding to $N$ and $M$ respectively \cite{jones:planar}. But planar algebras
do not need to be shaded. In fact the original
Temperley--Lieb equivalence \cite{TL-correspondence} was a mathematical equivalence between a shaded
model (the Potts model) and an unshaded one (Lieb's ice-type model). This equivalence
is clearly described in chapter 12 of  \cite{baxter-solvable} where we see that
 the spins of the Potts 
model live in the shaded
regions of a shaded 4-valent planar graph,  while for the ice-type model there are only two "spin" states which live on
the \emph{edges} of that graph.
 A more general form of TL equivalence  is that all representations
of the Temperley--Lieb planar algebra will yield the same answer for the partition function
of a model whose Boltzmann weights are defined using elements of the TL algebra. (The interactions
may involve more than two spins, even in the shaded case.)
 
In this paper we give matrix models for shaded planar algebras, for which  't Hooft diagrams
in the perturbative expansion \emph{are only allowed if they can be shaded}. 
In the particular case of the planar algebra underlying the Potts model,   
a quartic potential that uses the shading in an essential way is shown to yield the generating function
for the Tutte polynomials of planar graphs.
It is conceivable that one may obtain the same generating function from an unshaded 
planar algebra by some kind of TL equivalence. But there are
shaded planar algebras which are not equivalent to any unshaded 
planar algebra, and matrix models whose potentials are elements of these planar algebras would
be inaccessible to unshaded models.

 \subsection{Random matrix models}
In the case of our specific loop model and for integer values of fugacity, the partition function of the model could be expressed as a limit of random matrix integrals.  For example, if $D = \tikz[baseline=0]{
\begin{scope}
\node[draw,rectangle,rounded corners,minimum size=0.8cm] (a) at (0,0) {};
\clip[rounded corners] (a.south west) rectangle (a.north east); % nodes can't be clipped which is annoying
\draw[line width=1.5pt](a.south west) .. controls (a) .. (a.north west);
\draw[line width=1.5pt](a.south east) .. controls (a) .. (a.north east);
\end{scope}}
$, one is led to the random matrix model associated to the Gibbs measure on matrices given by \eqref{eq:GibbsMeasureRM}.  Considering a shaded analog of our loop model (corresponding to the observable given by 
\begin{equation}\label{defdbcup}
\mathscr{T}^\delta_{t_1,t_2}(T) = \sum_{r_1,r_2=0}^\infty
\frac{t_1^{r_1}t_2^{r_2}}{(r_1!)(r_2!)} \sum_{P\in
P(r_1,r_2,T)}
\delta^{\#\mbox{ loops in } P}\end{equation}
where the sum is over all possible 
sum over all possible configurations involving a vertex $T$, $r_1$ copies of  \tikz[baseline=0]{\vertexvw{a}(0,0)} and $r_2$ copies of  \tikz[baseline=0]{\vertexvb{a}(0,0)}) leads to a kind of $U(n)$ matrix model which corresponds to integration over $M\times M$ non-Hermitian matrices $A_1,\dots,A_n$ with respect to the Gibbs measure
\begin{multline} 
d\mu_{t_1,t_2}^{M} (A_1,\ldots,A_n)= \frac{1}{Z_{t_1,t_2}^M} \exp\Bigg\{-M \operatorname{Tr}\Big(
\sum_{j=1}^n A_j A_j^* \label{udelta} \\ - \sum_{i,j=1}^n \left(t_1  A_i A_j^* A_j A_i^* + t_2 A_i A_i^* A_j A_j^*\right)  \Big) \Bigg\} dA_1\dots dA_n.
\end{multline} 
It is not hard to generalize the definition of $P_T$ given in \eqref{defpt} by replacing some of the $A_i$ by $A_i^*$ to prove that
$$\lim_{M\ra\infty} \mu_{t_1,t_2}^{M}\big( \frac{1}{M}\Tr(P_T(A_1,\ldots, A_n))\big)= \mathscr{T}^n_{t_1,t_2}(T)\,.$$
It is natural to wonder if there are any random matrix models that can be defined for non-integer values of $n$.  To be consistent with subfactor notation we will from now on use the letter $\delta$ to denote the fugacity (or loop weight) instead of $n$.

In this paper, we show that this is the case for arbitrary loop models  coming from arbitrary planar algebras (satisfying a technical restriction).  In particular, this allows us to construct random matrix models that correspond to loop models with $\delta$ in the set $\{2\cos(\pi/p):p\geq 3\}\cup[2,+\infty)$.   

A key step in this construction uses the fact that arbitrary subfactor planar algebras can be viewed as subspaces of the linear spaces generated by loops on certain bipartite graphs (the so-called principal graphs).   Thus in our matrix model the random matrices are indexed by the edges of a bipartite graph $G=(V,E)$ with vertices $V$ and edges $E$, whose adjacency matrix $\Gamma$
has Perron--Frobenius eigenvalue $\delta$, which restricts ourselves to the above possible values.
As $G$ is bipartite,  $V=V^+\cup V^-$. We will see that the configurations can be indexed by the vertices and the (oriented) edges of the graph; shaded (resp. unshaded) regions will
be labelled by vertices from $V_-$ (resp. $V_+$) and the strings between a region indexed by $u$ and another by $v$ will be indexed by edges between $u$ and $v$. At the level of  the matrix model,
if  $A_{e}$ denotes the matrix
indexed by the edge $e$ from vertex $s$ to vertex $t$ in the bipartite graph, then $A_e$  will
 be a $[\mu(s)M]\times [\mu(t)M]$ matrix with $M$ a parameter going to infinity
and $\mu(v), v\in V$ the  Perron--Frobenius eigenvector of $\Gamma$ for the eigenvalue $\delta$ (here $[\cdot]$ denotes the integer part of a number).  We will require that $A_{e^o} = A_e^*$ where $e^o$ denotes the edge $e$ with opposite orientation.
To simplify, let us consider the planar algebra generated by shaded Temperley--Lieb
elements. In that case the graph $G$ can be chosen to be the Dynkin graph $\mathbb{A}_m$ (with $m$ a function of $\delta$).  
Let $B$ be  a shaded Temperley--Lieb
element. The generalization of the polynomial
$P_B$ of \eqref{defpt}  is then given by
$$
P_B^v(X_e, e\in E)=\sum_{e_j=e_p^o \textrm{ if }
j\stackrel{B}{\sim} p} \sigma_B(e_1\cdots e_{2k})
 X_{e_1} \cdots X_{e_{2k}}$$
where we sum over all loops $w=e_1\cdots e_{2k}$  starting (and finishing) at
$v$ in $G$ 
so that if $j\sim_B p$, $e_p$ is the opposite edge $e_j^o$ of $e_j$. Here
$\sigma_B(e_1\cdots e_{2k})$ is an appropriate constant (see equation 
\eqref{eq:defOfSigmaE}) which only depends on the Perron--Frobenius vector $\mu$ and the loop $e_1,\cdots,e_{2k}$. Here $v\in V_+$ iff the marked point of $B$ is in an unshaded region.

 To give a more precise idea of our matrix model let us first consider 
the case without interaction (that is when the parameter $t$ vanishes)
and denote $\mathscr{T}_0^\delta=\mathscr{T}_{0,0}$ with
$\mathscr{T}_{0,0}$ defined in \eqref{defdbcup}. In this case,
the entries  of $A_e$ will be independent Gaussian variables with
 covariance  given by $\sqrt{\mu(s)\mu(t) M^2}^{-1}$. We denote $\mu^M$
the law of the  matrices $A_e, e\in E$ which are independent 
modulo the symmetry constraint $ A_{e^o}=A^*_e$.
 If we have a shaded
Temperley--Lieb element $S$ with marked point in an unshaded region,
it follows from Wick calculus  (see also \cite{GJS07}) that for any  $v\in V_+$
$$\lim_{M\ra\infty}\mathbb E[ \frac{1}{M\mu(v)}\Tr(P_S^v(A_e,e\in E))]
=\mathscr{T}_0^\delta(S)\,.$$
Hence $\mathscr{T}_0(S)$ can be constructed as a limit of matrix models.
This fact was used in \cite{GJS07} to construct tracial states on planar algebras. In the case with interaction, $t\neq 0$, we need to add an interacting potential as we did in the case where $\delta=n$ is an integer number. This potential is defined in the same spirit as  the polynomials $P_B^v$: namely, if we want to add   in the interaction a shaded Temperley--Lieb element $S$ with
marked point in (say) an unshaded region
 we may  consider the potential
$$W_S(X_e, e\in E)=\sum_{v\in V^+}\mu(v) P^v_S(X_e, e\in E).$$
Let us now define, for $S_1,\dots,S_r$ shaded Temperley--Lieb elements the Gibbs measure
$$
d\mu_{t_1,\dots,t_r}^M (A_e, e\in E)
= \frac{1}{Z^{M}_{t_1,\dots,t_r}} \exp \left(M\operatorname{Tr} \Big\{\sum_{k} t_k W_{S_k} (A_e,e\in E)\Big\} \right)
 d\mu^M (A_e,e\in E)
$$
(for this to be well defined, we may need to  add a large enough cut off, see section \ref{pottsmatrixmodel}; we will also require that the graph $G$ be finite (otherwise we can show in this case that
the measure can still be defined due to the exponentially small correlations of matrices with ``far enough'' edges, see section \ref{Ainfty})). Note that even though  $W_S(A_e,e\in E)$ is a sum of square matrices with possibly different dimensions, its trace is well defined by linearity.

For instance,  
the matrix model associated with the $U(\delta)$ model \eqref{udelta} we already 
discussed is given, see \eqref{defdbcup},  by 
$$%\begin{multline*} 
d\mu_{t_1,t_2}^M(A)=\frac{1 }{Z^{M}_{t_1,t_2}}
\exp\Bigg\{ M \sum_{v\in V} (t_2 1_{v\in V_-}+t_1 1_{v\in V_+})\mu(v)  \Tr
\left(\sum_{e:s(e)=v}\sigma(e) A_{e}A_{e}^*\right)^2
\Bigg\}d\mu^M(A)
$$%\end{multline*}
This Gibbs measure is well defined provided $t_1,t_2\le 0$.

In the $U(n)$ model for $n$ integer, Temperley--Lieb diagrams can be represented as loops on the graph consisting of two vertices and $n$ edges between them: to a diagram $T$ with $2r$ boundary points we associate the sum of loops
\begin{equation}\label{defsigmaT}
\sigma_T = \sum_{e_{i_j}=e_{i_k}^o \mbox{ if } j\sim_T k} \sigma(e_1\cdots e_{i_{2r}}) e_{i_1} e_{i_2} \dots e_{i_{2r-1}} e_{i_{2r}}
\end{equation}
where $e_k$ denotes the $k$-th (oriented) edge from the first vertex to the second one, and $e_k^o$ denotes the same edge with opposite orientation.  
Thus the random matrix model involves non-self-adjoint $M\times M$ matrices $A_k$ corresponding to the edge $e_k$ (the adjoint $A_k^*$ corresponds to $e_k^o$).  The diagrams 
${}^*\tikz[baseline=0]{\vertexvw{a}(0,0)}$ and   ${}^*\tikz[baseline=0]{\vertexvb{a}(0,0)}$ then correspond to the terms $\sum_{ij} A_j^* A_i A_i^* A_j$ and $\sum_{ij} A_j  A_i^* A_i A_j^*$ which (after a cyclic permutation not changing the value of the trace) can be recognized as the two quartic terms in the potential of the $U(n)$ model.  (Note that in the case where the fugacity takes integer values, because of the special structure of the graph, the Perron--Frobenius eigenvector is identically equal to $1$ and so $\sigma(e_1\cdots e_{i_{2r}}) = 1$.)

%\end{itemize}
One of the main theorems of this paper is the following  generalization of the representation of  the observables associated to loop models by the asymptotics of random matrix models:
\begin{theorem}\label{cocomain} 
Let $\delta\in\{2\cos(\frac{\pi}{p})\}_{ p\ge 3}\cup [2,\infty[$
and let  $S,S_1,\ldots,S_k$ be $k$ fixed shaded elements of the Temperley--Lieb planar algebra with fugacity $\delta$.   
For $t_i, 1\le i\le k$, small enough real numbers, consider the observable
$$\mathscr{T}_t(S) = \sum_{r_1,\dots,r_k} \prod_{i=1}^k \frac{(t_i)^{r_i}}{r_i !} \sum_{P\in P(r_1,\dots,r_k)} \textrm{Value}(P)  
$$ where $P(r_1,\dots,r_k)$ is the set of configurations  build on $S$, $r_1$ copies of $S_1$, $r_2$ copies of $S_2$, etc.  and $$\textrm{Value}(P)=\delta^{\#\textrm{closed loops}}.$$
Then, for all $v\in V_+$ (resp. $V_-$) if the marked boundary segment is in an unshaded (resp. shaded) region,

 $$\mathscr{T}_t(S) = \lim_{M\ra\infty}\mathbb E_{\mu_t^{M}}\left[ \frac{1}{M\mu(v)}\Tr(P_S^v(A_e,e\in E))\right]
.$$
\end{theorem} 
The same theorem holds for more general planar algebras (rather than the Temperley--Lieb one). This theorem is proved in section \ref{proofcocomain}. %A
 
\subsection{Graded algebras and free Gibbs states}
We saw on page \pageref{freeProbabilityConnection} that the observable $\mathscr{T}$ has an interpretation as a linear functional on the set of polynomials of the form $P_T$, where $T$ is a Temperley--Lieb diagram.  It is not hard to see that such polynomials form a subalgebra of the algebra of all non-commutative polynomials.  In fact, $P_T \cdot P_S = P_{T\wedge S}$ if we define $T\wedge S$ as the disjoint union of $T$ and $S$ drawn so that their marked boundary points are at the top of their disks, and then inscribed into a larger disk.  It turns out that the operation $\wedge$ still makes sense whenever $T,S$ are elements of an arbitrary Jones planar algebra.  The resulting associative algebra is denoted by $Gr_0(\mathcal{P})$. \label{firstdefof:Gr0}  In consequence, the observable $\mathscr{T}$ can be viewed as a trace on  $Gr_0(\mathcal{P})$.

Remarkably, we show that this trace is also positive-definite (i.e., it is a tracial {\em state}): $\mathscr{T}_t(T\wedge T^*)\geq 0$ for all $T$.    Furthermore, the combinatorics of this trace is also governed by a kind of Schwinger--Dyson equation, which can be represented diagrammatically (see Lemma~\ref{picture-SD}).  The differential operators appearing in this equation are analogs of the free difference quotient and the cyclic derivative from Voiculescu's free probability theory \cite{dvv:cyclomorphy,dvv:entropysurvey,voiculescu:conjectureAboutPotentials}, and turn out to be important for subfactor theory (see e.g. \cite{curran-jones-shlyakht:enveloping}).

\subsection{Computations for certain models}
In the last part of the paper, we turn to two classical loop models,
and show that we can indeed use the matrix models
we have constructed to compute  partition functions
of the loop models.  Since in our random matrix models the fugacity $\delta$
can take its value in the set  $\{2\cos(\pi/p) : p\ge 3\}\cup[2,+\infty)$, this
allows us to determine these partition functions
for any fugacity by analyticity (see Lemma~\ref{lemana}).

The first model we
consider  is constructed with  Temperley--Lieb elements
 with  non nested strings 
and black inside (that is, depending
only on a cup shaded black inside
with the notations of \cite{GJS07})

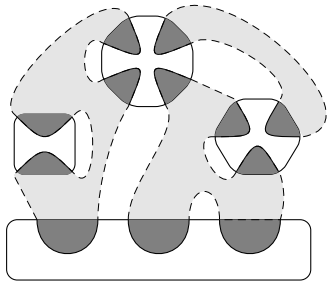
\begin{center}
%\resizebox{5cm}{!}{\includegraphics{cup}}
\begin{tikzpicture}
\draw[rounded corners] (0,-0.8) rectangle (4,0);
\draw[fill=gray] (0.4,0) .. controls (0.4,-0.6) and (1.2,-0.6) .. (1.2,0);
\draw[fill=gray] (1.6,0) .. controls (1.6,-0.6) and (2.4,-0.6) .. (2.4,0);
\draw[fill=gray] (2.8,0) .. controls (2.8,-0.6) and (3.6,-0.6) .. (3.6,0);
\vertexhw{a}(0.5,1);
%draw by hand non-standard tangles
\begin{scope}
\clip[draw,rounded corners] (3,0.6) -- ++(0.6,0) -- ++(0.3,0.5) -- ++(-0.3,0.5) -- ++(-0.6,0) -- ++(-0.3,-0.5) -- cycle;
\coordinate (b) at (3.3,1.1);
\draw[fill=gray] (3,0.6) .. controls (b) .. ++(0.6,0) ++(0.3,0.5) .. controls (b) .. ++(-0.3,0.5) ++(-0.6,0) .. controls (b) .. ++(-0.3,-0.5);
\end{scope}
\begin{scope}
\clip[draw,rounded corners] (1.6,1.5) -- ++(0.5,0) -- ++(0.35,0.35) -- ++(0,0.5) -- ++(-0.35,0.35) -- ++(-0.5,0) -- ++(-0.35,-0.35) -- ++(0,-0.5) -- cycle;
\coordinate (c) at (1.85,2.1);
\draw[fill=gray] (2.1,1.5) .. controls (c) .. ++(0.35,0.35) ++(0,0.5) .. controls (c) .. ++(-0.35,0.35) ++(-0.5,0) .. controls (c) .. ++(-0.35,-0.35) ++(0,-0.5) .. controls (c) .. ++(0.35,-0.35); 
\end{scope}
\path[connect] (0.4,0) -- (a.south west) .. controls (a) .. (a.south east) .. controls ++(0.3,-0.3) and ++(0.3,0.3) .. (a.north east) .. controls (a) .. (a.north west) .. controls ++(-0.3,0.3) and ++(-0.3,0.5) .. (1.6,2.7) .. controls (c) .. (1.25,2.35) .. controls ++(-0.3,0.2) and ++(-0.3,-0.2) .. (1.25,1.85) .. controls (c) .. (1.6,1.5) .. controls ++(-0.2,-0.3) and ++(0,0.5) .. (1.2,0);
\path[connect] (1.6,0) .. controls ++(0,0.5) and ++(0.2,-0.3) .. (2.1,1.5) .. controls (c) .. (2.45,1.85) -- (3,1.6) .. controls (b) .. (2.7,1.1) .. controls ++(-0.3,0) and ++(-0.2,-0.3) .. (3,0.6) .. controls (b) .. (3.6,0.6) to[bend left=15] (3.6,0) -- (2.8,0) .. controls ++(0,0.5) and ++(0,0.5) .. (2.4,0);
\path[connect] (3.9,1.1) .. controls ++(1.1,0.2) and ++(0.7,0.7) .. (2.1,2.7) .. controls (c) .. (2.45,2.35) .. controls ++(0.3,0.2) and ++(0.2,0.3) .. (3.6,1.6);
\end{tikzpicture}
\end{center}
In section  \ref{cupmodsec}, see Lemma \ref{cupmodlem},
we identify the law of cup (the element made with only one string 
and black inside)
 under  $\ourtrt$
as a probability measure minimizing a certain
entropy, whose Cauchy--Stieltjes functional
can be computed.

The second model we consider is the one mentioned earlier on, 
built on two Temperley--Lieb elements with two strings and opposite
shading $S_1=$\tikz[baseline=0]{\vertexvw{a}(0,0)} and 
$S_2=$\tikz[baseline=0]{\vertexvb{a}(0,0)}. We study the law of
cup  under $\ourtrt$:
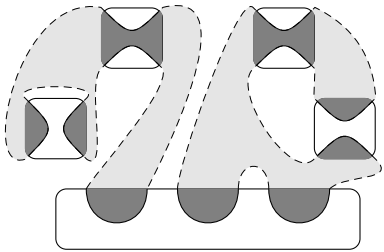
\begin{center}
%\resizebox{5cm}{!}{\includegraphics{dbcup}}
\begin{tikzpicture}%[baseline=0]
\draw[rounded corners] (0,-0.8) rectangle (4,0);
\draw[fill=gray] (0.4,0) .. controls (0.4,-0.6) and (1.2,-0.6) .. (1.2,0);
\draw[fill=gray] (1.6,0) .. controls (1.6,-0.6) and (2.4,-0.6) .. (2.4,0);
\draw[fill=gray] (2.8,0) .. controls (2.8,-0.6) and (3.6,-0.6) .. (3.6,0);
\vertexhw{a}(3.8,0.8)
\vertexhb{b}(3,2)
\vertexhb{c}(1,2)
\vertexvw{d}(0,0.8)
\path[connect] (0.4,0) .. controls ++(0.1,0.2) and ++(0.1,-0.3) .. (c.south east) .. controls (c) .. (c.south west) .. controls ++(-0.2,-0.2) and ++(0.3,-0.2) .. (d.south east) .. controls (d) .. (d.north east) .. controls ++(0.2,0.2) and ++(-0.2,0.2) .. (d.north west) .. controls (d) .. (d.south west) .. controls ++(-0.6,-0.2) and ++(-1.1,0.2) .. (c.north west) .. controls (c) .. (c.north east) .. controls ++(1.2,0.3) and ++(0.1,0.2) .. (1.2,0);
\path[connect] (3.6,0) .. controls ++(0.2,0.2) and ++(0.3,-0.3) .. (a.south east) .. controls (a) .. (a.south west) .. controls ++(-0.3,-0.3) and ++(-0.3,-0.3) .. (b.south west) .. controls (b) .. (b.south east) to (a.north west) .. controls (a) .. (a.north east) to[bend right] (b.north east) .. controls (b) .. (b.north west) .. controls ++(-0.2,0.2) and ++(0,0.3) .. (1.6,0) (2.4,0) .. controls ++(0,0.4) and ++(0,0.4) .. (2.8,0);
\end{tikzpicture}
\end{center}
By our construction, we can see this loop model as a limit of 
matrix models for any $\delta\in \{2\cos(\pi/p), p\ge 3\}\cup [2,+\infty)$.
We show in section \ref{doublecupsec} that this model 
is related to an auxiliary model which
shows up thanks to the Hubbard--Stratonovitch transformation, see Proposition \ref{prop:conv}. This auxiliary model also allows to compute 
the generating function of cup under $\ourtrt$.
 We then solve this auxiliary model based on the remark
that it depends only on the eigenvalues of the matrices involved. Therefore,
 standard large deviations techniques \cite{BAG97} can be used  and the
asymptotics of this model are described by a variational formula,
see Proposition \ref{propldp}. We finally study the optimizer of
this variational formula and show that its Stieltjes transform  can, up to a reparametrization,
be expressed as a ratio of theta functions, 
see Proposition \ref{ratiotheta}. We summarize below our main results 
on the Potts model.

\begin{theorem} Let $\ourtrt^\delta$ be the tracial state built on the
two \TL{} elements $S_1$ and $S_2$ with two strings and opposite
shading, $S_1$ having one unshaded region, see \eqref{defdbcup}. Assume that  $t=(t_1,t_2)$ are non negative and small enough.
\begin{itemize}
\item  There exists two unique  auxiliary probability measures $(\nu_-^t, \nu_+^t)$
on the real line  which minimize 
$$
\sum_{\e=\pm}\left(\frac{1}{2}\int x^2d\nu_\e(x)-\Sigma(\nu_\e)\right)
+\delta \iint \log|1+\sqrt{8t_1} x+\sqrt{8t_2}
y|d\nu_+(x)d\nu_-(y)$$
with $\Sigma$ the free entropy $\Sigma(\mu)=\iint \log|x-y|d\mu(x)d\mu(y)$.
\item 
Let $B_n$ be the shaded Temperley--Lieb element with $n$ non nested strings
and black shading inside and put  
$$C(\gamma,t)=\sum_{n\ge 0}\gamma^n\ourtrt(B_n).$$
The above series converges absolutely for $\gamma$ small enough. 

Let
 $M^t(z)=\int \sum_{n\ge 0} z^n x^n d\nu_+^t(x)$. Then,
  $\gamma^t(z)=\frac{\sqrt{8 t_1} z}{1-z^2 M^t(z)}$
is invertible on a neighborhood of
the origin, with inverse $z^t(\gamma)$, and we have on a neighborhood of the origin
$$C(\gamma,t)=\frac{ \alpha z^t(\gamma)}{\gamma} M^t( z^t(\gamma)).$$
\item
 The probability measures $(\nu_+^t,\nu_-^t)$ can also be described as follows. There exists numbers $a_1<a_2<b_1<b_2$ so that $\nu_+^t$ (resp. $\nu_-^t$)
is supported on $[-a_2/\sqrt{8t_1},-a_1/\sqrt{8t_1}]$ (resp. $[(b_1-1)/\sqrt{8t_2},(b_2-1)/\sqrt{8t_2}]$). Moreover  set
$$u(z)
:=\frac{i}{2}\sqrt{(b_1-a_1)(b_2-a_2)} 
\int_{b_2}^z \frac{dz'}{\sqrt{(z'-a_1)(z'-a_2)(z'-b_1)(z'-b_2)}}\,,$$
 let $z(u)$ be  its inverse, and set
$$\omega_+(u)=\int \frac{1}{z(u)+\sqrt{8t_1} x}d\nu_+^t(x)=\frac{1}{\sqrt{8t_1}}
M^t(\frac{z(u)}{\sqrt{8t_1}})\,.$$
Then, if $q$ is such that $\delta=q+q^{-1}$, $q=e^{i\pi\nu}$, we have
$$\omega_+(u)=\frac{1}{q-q^{-1}}[(\phi_+(u)-\phi_+(-u)) + R(z(u))]$$
with, if $\Theta$ is the theta function given by \eqref{theta}, 
$$\phi_+(u)=c_+ \frac{\Theta(u-u_\infty-2\nu K)}{\Theta(u-u_\infty)}
+c_- \frac{\Theta(u+u_\infty-2\nu K)}{\Theta(u+u_\infty)}$$
and
$$R(z)=\frac{2q}{1-q^2}\frac{z}{\sqrt{8t_1}} +\frac{q^2+1}{1-q^2} \frac{(z-1)}{\sqrt{8t_2}
}.$$
The constants $a_1,a_2,b_1,b_2,u_\infty,C_+,c_-,K$ are defined by implicit equations. If $\delta=2\cos(\frac{\pi}{p})$ for some $p\in \mathbb N$, $\nu$ is an integer
and $\omega_+$ satisfies an algebraic equation.
\end{itemize} %A: added relation \alpha beta with t, change order of statements, added t_i\le 0
\end{theorem}

Whereas the matrix model for the first model is well-known,
the matrix model for the second (a kind of $U(n)$ model) has only been solved in some
special cases in the literature, such as  
the $O(n)$ model \cite{DK88,Kos89,KS92,EK95,BE09} which
corresponds to the case where the shading is neglected, that is $t_1=t_2$. 
Moreover, matrix models are usually provided for $\delta$ integer, 
whereas our approach allows a general construction
of the matrix model for all the values of $\delta$ above,
a set which accumulates at $2$. For this specific model, one could a priori 
define a matrix model for all $\delta$ by using the auxiliary model which appears thanks to the Hubbard--Stratanovitch transformation of section \ref{doublecupsec}; this is however a rather indirect and specific approach whereas our matrix models apply in a much greater generality, in particular in situations where the model has not yet been computed.

Our
construction is closely related to Pasquier's \cite{Pas87}, though the latter
is in a slightly different context, namely that of
 statistical models on the square
lattice (whereas there is no underlying lattice in our construction;
it is in some sense random). See also \cite{Kos-ADE,Kos-ADEb,Kos-ADEc} for
another application of Pasquier's construction in the context of matrix models.
Moreover, the enumeration problem 
corresponding to our second matrix model was
recently considered in \cite{BBM09} in an equivalent formulation,
namely the Potts model; we shall comment on the exact relation
to our work in section 3.

\subsection{List of notations}
\begin{description}
\item[$\delta$]  Loop weight (fugacity), or the loop parameter (part of the data of a planar algebra).  To be consistent with established notation, we use $\delta$ and $n$ (see below) interchangeably (\SectMark\ref{firstdefof:delta}, p. \pageref{firstdefof:delta}). 
\item[$f$] Partition function of a loop model.
\item[$\Gamma=(V,E)$] A bipartite graph whose set of vertices is $V$ and set of edges is $E$.  The set $V$ is the disjoint union of the sets $V_+$ and $V_-$ of even and odd vertices (\SectMark\ref{firstdefof:bipartiteGraph}, p. \pageref{firstdefof:bipartiteGraph}).
\item[$Gr_0(\mathcal{P})$] The graded algebra associated to a planar algebra $\mathcal P$. In the Temperley--Lieb case, the elements of $Gr_0$ are Temperley--Lieb diagrams (drawn with all of their marked boundary points on the top, starting with the marked boundary point).  The multiplication is then obtained by drawing diagrams next to each other and embedding the resulting diagram into a disk.
(\SectMark\ref{firstdefof:Gr0}, p. \pageref{firstdefof:Gr0})
\item[$L^\pm$] Loops (closed paths) on a bi-partite graph; $+$ and $-$ indicates whether the loop starts at an even or an odd vertex (\SectMark\ref{firstdefof:bipartiteGraph}, p. \pageref{firstdefof:bipartiteGraph}).
\item[$\mu(v)$] The value of the Perron--Frobenius eigenvector at a vertex $v$. (\SectMark\ref{firstdefof:bipartiteGraph}, p. \pageref{firstdefof:bipartiteGraph})
\item[$\mu_t^M$, $\mu_t^{M,K}$] $\mu_t^M$ is a Gibbs measure from which our random matrix model (of size proportional to $M$) is sampled. (\SectMark\ref{firstdefof:gibbsMeasure}, p. \pageref{firstdefof:gibbsMeasure})
\item[$M$] Matrices in our random matrix model have blocks with sizes proportional to $M$; we are concerned with the large $M$ limit of the random matrix model.
\item[$n$] The loop weight (fugacity) of a loop model when it is an integer number; %A
 the number of matrices or non-commutative variables. See $\delta$.
\item[$\mathcal P$] Planar algebra.  A reader unfamiliar with planar algebra may wish to concentrate on the Temperley--Lieb planar algebra $\TL$, which can be thought of as the linear span of shaded Temperley--Lieb diagrams. 
(\SectMark\ref{firstdefof:PA}, p. \pageref{firstdefof:PA})
\item[$P_k, P_{k,\pm}$]  The $k$-box space ($k$-th graded component) of a planar algebra. The choice of sign $\pm$ indicates shading.  For the case of the Temperley--Lieb alagebra, $P_k$ denotes all (shaded) Temperley--Lieb diagrams with $2k$ boundary points, and $P_{k,+}$ denotes those diagrams for which the marked boundary point is clockwise from an unshaded region. (\SectMark\ref{firstdefof:PA}, p. \pageref{firstdefof:PA})
\item[$\mathcal{P}^\Gamma$] The planar algebra associated to a bipartite graph $\Gamma$. (\SectMark\ref{firstdefof:bipartiteGraph}, p. \pageref{firstdefof:bipartiteGraph})
\item[$\sigma(e)$] The factor $\sigma(e)=\sqrt{\mu(t(e))/\mu(s(e))}$ associated to an edge $e$ of a bipartite graph.
\item[$\sigma_B$] For a planar algebra element $B\in\mathcal{P}$ and an embedding $i$ of $\mathcal{P}$ into a graph planar algebra $\mathcal{P}^\Gamma$, $\sigma_B(w)$ denotes the value of the image of $i(B)$ on the loop $w\in \mathcal{P}^\Gamma$.  In the case of a Temperley--Lieb diagram $B\in TL$, $\sigma_B(w)$ is given by \eqref{eq:defOfSigmaE} on p. \pageref{eq:defOfSigmaE}.
\item[$S,T$] Elements of a planar algebra; particular examples of these are shaded Temperley--Lieb diagrams.
\item[$s(e),t(e)$] For a oriented edge $e$, $s(e)$ is the starting point and $t(e)$ is the end-point. (\SectMark\ref{firstdefof:bipartiteGraph}, p. \pageref{firstdefof:bipartiteGraph}).  We sometimes write $s(w)$ for a loop $w\in L$ to indicate its starting vertex.
\item[$t$] One or more parameters (e.g., $t_1,\dots,t_k$).
\item[$\Tr(A)$] The unnormalized trace of an $L\times L$ matrix $A$ given by $\sum_{i=1}^L A_{ii}$. % In our random matrix model associated to a bipartite graph, matrix entries are indexed by a pair of vertices $v$, $w$ and integers $i,j$; thus a matrix $A$ can be viewed as a block matrix whose (rectangular) blocks are indexed by pairs $v,w$ of vertices in the graph.  
\item[$\Tr_0(S)(v)$] For an arbitrary planar algebra $\mathcal{P}$, the map $\mathscr{T}(S)$ given by the pairing $\langle S, \sum_{T\in\TL} T\rangle$ between $S$ and the sum of Temperley--Lieb diagrams gives rise to an element in $P_0^\pm$.  In the case of a graph planar algebra, $\mathscr{T}(S)$ is a function on the set of vertices, and $\Tr_0(S)(v)$ denotes the value of this function at a vertex $v$. (\SectMark\ref{firstdefof:Tr0}, p. \pageref{firstdefof:Tr0})
%%\item[$\Tr_V(A)$] For our matrix model, $\Tr_V(A)=\sum_{v\in V} \mu(v) \Tr(A)(v).$%% We don't need it any more?
\item[$\tr(A)$] The normalized trace of an $L\times L$ matrix $A$ given by $\frac{1}{L}\Tr(A)$.
\item[$\mytau$] Linear functional obtained as limit of averages of traces with respect to a family Gibbs measures on matrices.
\item[$\tau$] A non-commutative law, i.e., a tracial linear state on the algebra of non-commutative polynomials.
\item[$\mathscr{T}$] An observable associated to a loop model; also a tracial linear functional on the algebra $Gr_0(\mathcal{P})$ associated to a planar algebra $\mathcal P$ (\SectMark\ref{sect:observables}, p. \pageref{sect:observables}).
\item[$\TL$] The Temperley--Lieb planar algebra; the same symbol also sometimes refers to the set of  all Temperley--Lieb diagrams. (\SectMark\ref{firstdefof:TL}, p. \pageref{firstdefof:TL}; \SectMark\ref{firstdefof:TLalg}, p. \pageref{firstdefof:TLalg})
\item[$V_\pm$] Vertices of a bipartite graph, $+$ indicates even and $-$ indicates odd vertices. (\SectMark\ref{firstdefof:bipartiteGraph}, p. \pageref{firstdefof:bipartiteGraph})
\item[$W$] The interaction term in our models; we assume that $W=\sum t_i W_i$ with $W_i$ elements of a planar algebra $\mathcal{P}$.

\item[$\Vert\cdot\Vert_\infty$] Operator norm.
\end{description}

\subsection*{Acknowledgements}
The authors are grateful to the referees of the paper for their helpful suggestions for improving the readability of the paper.

\section{From planar algebras to matrix models}\label{proofcocomain}
In this section, we introduce our  matrix
models and prove Theorem \ref{cocomain}.  The matrix models depend on
 a bipartite graph which shows up
in Jones's subfactor theory (see
\cite{JP10} and references therein).
We first recall the definition of a subfactor 
planar algebra and the construction of the
planar algebra from a bipartite graph. The example 
that the reader can keep in mind is the Temperley--Lieb algebra.

We then define a family of random matrices associated to a bipartite graph
whose adjacency matrix has eigenvalue 
$\delta$ corresponding to some Perron--Frobenius vector $\mu$.

We next consider the case of a finite graph
and introduce
the Gibbs measure associated to $\ourtrt $
 and prove Theorem \ref{cocomain}
for planar subalgebras of graph planar algebras.\
 Finally, we extend our construction
to certain infinite graphs. 

\subsection{Planar algebras} \label{firstdefof:PA}
%In this paper, we crucially use the planar algebra structure on the 
%linear span $\TL$ of Temperley--Lieb diagrams.  
We  first recall some generalities on planar algebras.  The reader is referred to \cite{jones:planar}
or \cite{GJS07} for a more extensive introduction.

Recall  \cite{jones:planar} that a {\em planar algebra}\/ is a collection of vector spaces $\mathcal{P}=\{P_k^\pm\}$ endowed with an action of {\em planar tangles}.

A planar tangle is a drawing consisting of an {\em output disk}\/ $D_0$ and some number of {\em input disks}\/ $D_1,\dots,D_k$ in the interior of $D$ ($k\geq 0$).  Each disk has an even number of marked boundary points.  On each disk, one of the boundary segments is marked and called the {\em initial segment}.  The boundary points are joined by strings drawn in the interior of $D_0$ and outside all $D_1,\dots, D_k$; in addition there may be some number of closed strings not connected to any of the $D_i$'s.  All of  the strings are non-crossing.  Lastly, some of the regions between the strings are supposed to be 
shaded, so that each string lies between a shaded and  an unshaded region.
Planar tangles can be composed by gluing the output disk of one tangle into an input disk of another tangle so as to match up the initial segments.  In doing so, one must ensure that the numbers of boundary points and the shadings match.

The main axiom of a planar algebra is the existence, for each tangle $T$  with disks $D_0,\dots, D_k$ as above, so that $D_j$ has $2b_j$ boundary points,
of a multilinear map $M_T: P^{\sigma_1}_{b_1}\times \dots \times P^{\sigma_n}_{b_k} \to P^{\sigma_0}_{b_0}$, where $\sigma_j=+$ if the initial segment of $D_j$ is adjacent to a white region, and $\sigma_j=-$ otherwise.  The maps $M_T$ are supposed to be compatible with the operation of composition of tangles and  invariant under isotopy. Moreover, the vector spaces $\{P_k^\pm\}$ are equipped with
an involution compatible with $M_T$ in the sense that $M_T(f^*)=M_{\phi(T)}(\phi\circ f)^*$ for any
orientation reversing diffeomorphism
$\phi$.

A {\em subfactor}\/ planar algebra is a planar algebra so that ${\rm dim}( P^{\sigma}_{0})=1$. As a consequence
we can define a sesquilinear form 
on each $P_{n}^{\pm}$
by
$$\langle A,B\rangle = 
\begin{array}{c}
\begin{tikzpicture}[ inner sep=2mm]
\draw[rounded corners] (0,0) rectangle (7,1.25);
\draw[line width=1.5pt,rounded corners] node (A) at (2.5,0.625) [rectangle,draw] {$\ A\  $};
\draw[line width=1.5pt,rounded corners] node (B) at (4.5,0.625) [rectangle,draw] {$B^*$};
\draw (node cs:name=A, angle=30) ..controls +(30:3mm) and +(150:3mm)  ..  (node cs:name=B, angle=150);
\draw (node cs:name=A, angle=-30) ..controls +(-30:3mm) and +(-150:3mm)  ..  (node cs:name=B, angle=210);
\draw (node cs:name=A, angle=10) ..controls +(10:3mm) and +(170:3mm)  ..  (node cs:name=B, angle=170);
\node at (3.5,0.625){$\vdots$};
\node at (2.2,0.2){$\ast$};
\node at (4.2,0.2){$\ast$};
\end{tikzpicture}
\end{array}
$$
where the outside region is shaded according to $\pm$.  We also require 
that $\langle \,,\,\rangle$ is positive definite and $M_{T_1}=M_{T_2}$
where $T_{1}$ and $T_{2}$ are the following
two 0-tangles:
$$
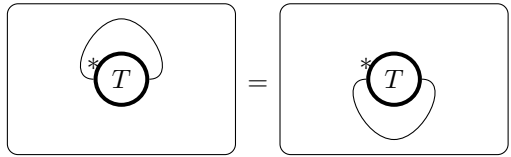
\begin{array}{c}
\begin{tikzpicture}
\draw[rounded corners](0,0) rectangle (3,2);
\draw[line width=1.5pt, rounded corners] node (A) at (1.5,1) [circle,draw] {$T$};
\draw(node cs:name=A, angle=180) .. controls +(180:4mm) and +(180:4mm) .. 
(1.5,1.8)  .. controls +(0:4mm) and +(0:4mm) .. (node cs:name=A,angle=0);
\node [above] at (A.west) {$\ast$};
\end{tikzpicture}
\end{array}
 = 
\begin{array}{c}
\begin{tikzpicture}
\draw[rounded corners](0,0) rectangle (3,2);
\draw[line width=1.5pt, rounded corners] node (A) at (1.5,1) [circle,draw] {$T$};
\draw(node cs:name=A, angle=180) .. controls +(180:4mm) and +(180:4mm) .. 
(1.5,0.2)  .. controls +(0:4mm) and +(0:4mm) .. (node cs:name=A,angle=0);
\node [above] at (A.west) {$\ast$};
\end{tikzpicture} 
\end{array}
$$

Once $P_{0,\pm}$ have been identified with the scalars
there is a canonical scalar $\delta$ \label{firstdefof:delta} associated with a subfactor
planar algebra with the property that the multilinear map associated
to any tangle containing a closed string is equal to $\delta$ times
the multilinear map of the same tangle with the closed string removed.
By positivity of the scalar product, $\delta$ has to be positive
 and in fact it is well-known that the possible
values of $\delta$ form the set $\{2\cos (\pi/p):p=3,4,5,\ldots\}\cup[2,\infty)$
\cite{jones:index}.

\subsubsection{Example 1: Temperley--Lieb algebra $\TL$} \label{firstdefof:TLalg}

It is not hard to see that the Temperley--Lieb algebra
 is a planar algebra.  Indeed, given a planar tangle $T$ and
 some elements $B_1,\dots,B_k\in \TL$ one can glue these elements 
into the input disks of $T$.  Next, one can remove all closed strings by replacing each closed string by the factor $\delta$. 
 What results is another $\TL$ tangle, which is the result of applying the map $M_T$ to $B_1,\dots,B_k$.
Clearly, $M_T$ is invariant by isotopy and $P_0^\pm$ has dimension one.  Finally, the canonical scalar product
is positive definite according to \cite{jones:planar}. One way to prove it is by verifying 
that the  map of \TL{}  into a graph planar algebra is a planar algebra map. It thus takes the canonical bilinear form on \TL{}  to the canonical bilinear form on a graph
planar algebra, where non-negativity can be verified directly.   We will write $\TL(\delta)$ when we wish to emphasize the loop parameter (fugacity) $\delta$.

\subsubsection{Example 2: The planar algebra of two stitched Temperley--Lieb algebras}
\label{subsec:coupledTL}
Later in the paper, we will use (in addition to the Temperley--Lieb algebra) the stitched planar algebra $\mathcal{P}=\TL(\delta_1)\copyrightsign\TL(\delta_2)$.
This will be needed 
in order to realize the so-called $O(n,m)$ model in physics.

The $n$-th graded component of $P_n$ of $\mathcal{P}$ is given by$$P_n^{\pm} = \bigoplus_{\pi} (TL\copyrightsign TL)_\pi^{\pm}$$
where the sum is over all partitions $\pi$ of $\{1,2,\dots,2n\}$ into two subsets of even sizes $2(p_\pi)$ and $2(q_\pi)$
and $(\TL\copyrightsign \TL)_\pi^\pm = \TL_{p_\pi}^\pm \otimes \TL_{q_\pi}^\pm$.  Graphically, one can view $P_n$ as the span
of the collection of isotopy classes of tangles
obtained by superimposing an arbitrary $TL$ tangle colored red with
an arbitrary $TL$ tangle colored black, in such a way that the strings
of the red and the black $TL$ tangles are allowed to only intersect
transversally, and so that the resulting tangle has a total of $2n$
boundary points (counted regardless of color). The partition $\pi$ then corresponds to the colorings of the boundary points of the resulting diagram.
We assume that the checkerboard coloring of the two 
superimposed $TL$ diagrams are retained and are independently superimposed.

 The isotopies need
not preserve the intersections of the red and black strings, but must preserve the partition $\pi$.
We also assume that one of the boundary regions is marked {}``first''
(it could be of either color). Two different isotopy classes of diagrams,
$T$ and $S$, are presented below (black strings are indicated by solid lines and 
red by dotted lines, and the four possible shadings of regions are indicated by 
$
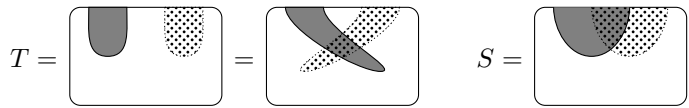
\begin{tikzpicture}\draw(0,0) rectangle (0.6,0.3); \end{tikzpicture}, 
\begin{tikzpicture}\draw [fill=gray] (0,0) rectangle (0.6,0.3); \end{tikzpicture}, 
\begin{tikzpicture}\draw [pattern=crosshatch dots] (0,0) rectangle (0.6,0.3); \end{tikzpicture}, 
\begin{tikzpicture}\draw [fill=gray] (0,0) rectangle (0.6,0.3); \draw [pattern= crosshatch dots] (0,0) rectangle (0.6,0.3); 
\end{tikzpicture} 
$):

$$
T = 
\begin{array}{c}
\begin{tikzpicture}
\draw[rounded corners](0,0) rectangle (2,1.3);
\draw[fill=gray] (0.25,1.3) .. controls +(-90:3mm) and +(180:3mm) .. (0.5,0.65) 
	.. controls +(0:3mm) and +(-90:3mm) .. (0.75,1.3) -- cycle;
\draw[densely dotted,pattern= crosshatch dots] (1.25,1.3) .. controls +(-90:3mm) and +(180:3mm) .. (1.5,0.65) 
	.. controls +(0:3mm) and +(-90:3mm) .. (1.75,1.3) -- cycle;
\end{tikzpicture} 
\end{array}
=
\begin{array}{c}
\begin{tikzpicture}
\draw[rounded corners](0,0) rectangle (2,1.3);
\draw [fill=gray] (0.25,1.3) .. controls +(-90:3mm) and +(180:3mm) .. (1.5,0.45) 
	.. controls +(0:3mm) and +(-90:3mm) .. (0.75,1.3) -- cycle;
\draw[densely dotted,pattern= crosshatch dots] (1.25,1.3) .. controls +(-90:3mm) and +(180:3mm) .. (0.5,0.45) 
	.. controls +(0:3mm) and +(-90:3mm) .. (1.75,1.3) -- cycle;
\end{tikzpicture} 
\end{array}
\qquad 
S=
\begin{array}{c}
\begin{tikzpicture}
\draw[rounded corners](0,0) rectangle (2,1.3);
\draw [fill=gray] (0.25,1.3) .. controls +(-90:3mm) and +(180:3mm) .. (0.75,0.65) 
	.. controls +(0:3mm) and +(-90:3mm) .. (1.25,1.3);
\draw[pattern= crosshatch dots, densely dotted] (0.75,1.3) .. controls +(-90:3mm) and +(180:3mm) .. (1.25,0.65) 
	.. controls +(0:3mm) and +(-90:3mm) .. (1.75,1.3);
\end{tikzpicture} 
\end{array}
$$

The planar algebra structure of $P_n$ is defined as follows.  Given a planar tangle $T$ and diagrams $A_1,\dots,A_k$ in $\mathcal{P}$, the result $M_T(A_1,\dots,A_k)$ is obtained by gluing the diagrams $A_1,\dots,A_k$ into the input disks of $T$ and summing over all possible ways of extending the colorings and shadings of the $A_i$'s to the resulting tangle.  The construction of $\mathcal{P}$ is a particular case of a more general construction $\mathcal{P}_1\copyrightsign \mathcal{P}_2$, which is possible for any pair of planar algebras $\mathcal{P}_1, \mathcal{P}_2$.  This construction is presented in Appendix~\ref{sec:couplings}.

\subsection{On the planar algebra of a graph}
Following \cite{GJS07},
we shall use the construction of planar algebras 
from bipartite graphs, as introduced in \cite{jones:graphPlanarAlg}.
The key ingredient here is the fact that every subfactor planar algebra (in particular, $\TL$) embeds
into a graph planar algebra.  This makes it possible to ``coordinatize'' planar algebras.

We first fix notations.

Let $\Gamma=(V,E)$ be a bipartite \label{firstdefof:bipartiteGraph} graph with vertices $V=V_-\cup V_+$
so that any edge is either from $V_+$ to $V_-$ or $V_-$ to $V_+$.
We denote by $E_+$ (resp. by $E_-$) the set of (oriented) edges starting in $V_+$
(resp. $V_-$). Thus $E=E_+\cup E_-$.
We let  $\mu$
 be a fixed Perron--Frobenius eigenvector with eigenvalue
$\delta$ for the adjacency matrix of $\Gamma$.  The vector $\mu$ has positive
entries. If $e\in E$, we denote by $e^o$ the edge with opposite orientation. We  denote by $L$ the set of  loops
on $\Gamma$, $L^+$ (resp. $L^-$) the set of loops
starting in $V_+$ (resp. $V_-$) (so $L=L^+\cup L^-$). We denote by $L(v)$
the set of loops starting at $v\in V$. We finally let 
 $\mathcal{P}^{\Gamma}=\cup_{n,\pm}{P}_{n,\pm}^{\Gamma}$
where $P_{n,\pm}^{\Gamma}$ is  the vector space of bounded functions
on loops on $\Gamma$ of length $2n$ starting and ending in $V_{+}$
for the plus sign and $V_{-}$ for the minus sign. In the following,
$s(e)$ (resp. $t(e)$) is the starting (resp. target) vertex of an edge $e$ (note that several
edges can have the same starting and ending  points). 

\begin{example} Consider the graph with one vertex in $V_+$, one vertex in $V_-$ and $n$ edges between them.  In this case, $\delta=n$ is an integer, and the 
eigenvector $\mu$ is identically equal to $1$.   The $\TL$ algebra embeds into the planar algebra of this graph for integer fugacity $\delta$.
\end{example}

We next describe the action of planar tangles on $\mathcal{P}^\Gamma$.  Let $T$ be a planar tangle with $k$ input disks
and let $L_1,\dots,L_k$ be loops on $\Gamma$.  To define the planar algebra structure, we must exhibit 
 $M_T(L_1, \dots,L_k)$ as an element of the planar algebra, i.e., we must prescribe its value on a loop $L$.  The value 
$$M_T(L_1,\dots,L_k) (L) $$ is computed as follows.  First, label the marked points on the input disks of $T$ with the edges comprising the loops $L_1,\dots,L_k$, clockwise starting from the marked vertex (and the beginning of $L_j$). 
Next, label the marked points on the output disk of $T$ with the edges of $L$, clockwise starting from the marked vertex.  As a result, we obtain a {\em labeled tangle}, and we'll set $M_T(L_1,\dots,L_k) (L) $ equal to the value of this labeled tangle which we compute as follows.
First, we remove all closed loops in the tangle $T$ (let us say, $p$ loops total) and multiply $M_T(L_1,\dots,L_k)(L)$ by $\delta^p$. We'll again denote the tangle with removed interior loops by $T$.

Next, we isotope the tangle in such a way that each input disk becomes a rectangle whose top is horizontal, and so that all strings emanate from the top (in this way, the marked initial segment comprises the sides and bottom of the rectangle).
 Put $$\sigma(e)=\sqrt\frac{\mu(t(e))}{\mu(s(e))}, e\in E.$$  Then the value $M_T(L_1,\dots,L_k) (L) $ is zero unless each string connects points which are labeled by opposite edges.  Otherwise, 
\begin{equation} \label{eq:simplificationRules}
M_T(L_1,\dots,L_k) (L)= \prod_{\textrm{strings $s$}} \sigma(e_s)^{-\theta_s/\pi}
\end{equation}
where $e_s$ is the start of $s$ and $\theta_s = \int_s d\theta $ is the total winding angle of the string $s$.  Here $d\theta$ stands for the $1$-form $y dx - x dy$ on the coordinate plane.    (Note that the choice of which edge is 
selected as the start of a string $s$ is irrelevant: if $s'$ is the string $s$ traversed backwards, then we get a non-zero value for the tangle iff $e_{s'}=e_s^o$; note that $\sigma(e^o)=\sigma(e)^{-1}$ and $\theta_{s'} = -\theta_s$.)  

We note the identities (thick lines indicate an arbitrary number of parallel strings, $e$ and $f$
are two arbitrary edges, and $x,y,z,w$ are arbitrary paths on graphs so that $xey$ and $zfw$ are loops, and all planar algebra elements are arranged so that the marked initial segment is at top-left):
\begin{gather*}
\begin{array}{c}
\begin{tikzpicture} [ xxx/.style={
	rectangle,
		rounded corners,
		minimum size=12mm,
		line width = 1.5pt},
		manyStrings/.style={
			line width=5pt}]
\node (A) at (0,0)  [xxx,draw] {$$};
\draw [manyStrings] (node cs:name=A, angle=120) -- +(0,0.5);
\draw [manyStrings] (node cs:name=A, angle=60) -- +(0,0.5);
\draw (node cs:name=A, angle=100) .. controls (up:11mm)  and (up:11mm) .. (node cs:name=A, angle=80);
%\draw (A) edge [in=80, out=100, loop];
\node [below] at (node cs:name=A, angle=110) {$\scriptstyle e$};
\node [below] at (node cs:name=A, angle=70) {\raisebox{-4.5pt}[1pt][1pt]{$\scriptstyle f$}};
\end{tikzpicture}
\end{array}
= \delta_{e=f^o} \sigma(e)
\begin{array}{c}
\begin{tikzpicture} [ xxx/.style={
	rectangle,
		rounded corners,
		minimum size=12mm,
		line width = 1.5pt},
		manyStrings/.style={
			line width=5pt}]
%\node at (-0.7,0.9) {$\cdots$};
%\node at (0.7, 0.9) {$\cdots$};
\node (A) at (0,0)  [xxx,draw] {};
\draw [manyStrings] (node cs:name=A, angle=120) -- +(0,0.5);
\draw [manyStrings] (node cs:name=A, angle=60) -- +(0,0.5);
\end{tikzpicture}
\end{array} \\ %%%%%%%% second equation
\begin{array}{c}
\begin{tikzpicture} [ xxx/.style={
	rectangle,
		rounded corners,
		minimum size=12mm, 
		line width = 1.5pt},
		manyStrings/.style={
			line width=5pt}]
%\node at (-0.7,0.9) {$\cdots$};
%\node at (0.7, 0.9) {$\cdots$};
\node (A) at (0,0)  [xxx,draw] {};
\draw [manyStrings] (node cs:name=A, angle=120) -- +(0,0.5);
\node [below] at (A.120){$\scriptstyle x$};
\node [below] at (A.60) {$\scriptstyle y$};
%\draw (A.north) -- +(0,0.5);
\node (B) at (3,0) [xxx,draw]{};
%\node at (2.3,0.9) {$\cdots$};
%\node at (3.7, 0.9) {$\cdots$};
%\draw (B.north) -- +(0,0.5);
\draw [manyStrings] (node cs:name=B, angle=60) -- +(0,0.5);
\node [below] at (B.120){$\scriptstyle z$};
\node [below] at (B.60){$\scriptstyle w$};
%\draw (A.north) .. controls +(up:7mm) and +(right:3mm) .. (1.35,1.4) ..
%controls +(right:3mm) and +(up:7mm) .. 
%(B.north);
\draw (A.north) .. controls +(up:7mm) and +(up:7mm) .. (B.north);
\draw [manyStrings, rounded corners] (node cs:name=A, angle=60) -- ++(0,2mm)
	-- ++(7mm,0) -- ++(0,-18mm) -- ++(-20mm,0) -- ++(0, +21mm) ;
\draw [manyStrings, rounded corners] (node cs:name=B, angle=120) -- ++(0,2mm)
	-- ++(-7mm,0) -- ++(0,-18mm) -- ++(20mm,0) -- ++(0,+21mm);
\node [below] at (A.north) {$\scriptstyle e$};
\node [below] at (B.north) {$\scriptstyle f$};
\end{tikzpicture}
\end{array}
= \delta_{e=f^o}\sigma(e)\sigma(z)^{-2}  %A2: there is 360degre angle, to be consistent with p. 23
 \begin{array}{c}
\begin{tikzpicture} [ xxx/.style={
	rectangle,
		rounded corners,
		minimum size=12mm,
		line width = 1.5pt},
		manyStrings/.style={
			line width=5pt}]
\node(A) at (0,0) [xxx,draw]{\hbox to24mm{\hfill}};
\draw [manyStrings] (A.140) -- +(0,5mm);
\node [below] at (A.140) {\raisebox{-0.0em}{${\scriptstyle x}$}};
\draw [manyStrings] (A.120) -- +(0,5mm);
\node [below] at (A.120) {\raisebox{-0.0em}{${\scriptstyle w}$}};
\draw[manyStrings](A.60) --+(0,5mm);
\node [below] at (A.60) {\raisebox{-0.00em}{${\scriptstyle z}$}};
\draw[manyStrings, rounded corners](A.40) -- ++(0,3mm) -- ++(0.8,0) -- ++(0,-1.8) -- ++(-3.1,0) -- ++(0,2.0);
\node [below] at (A.40) {\raisebox{-0.0em}{${\scriptstyle y}$}};
%\node at (-0.1,1.5) {$\cdots$};
%\node at (1.3, 1.5) {$\cdots$};
%\node at (2.9,1.5) {$\cdots$};
%\node at (4.3, 1.5) {$\cdots$};
\end{tikzpicture}
\end{array}
\end{gather*} where $\sigma(z) = \prod \sigma(z_j)$ if $z=z_1\dots z_l$.
(The last operation can be though of as replacing the string connecting $e$ and $f$ by a ``tunnel'' joining the two planar algebra elements: {$$
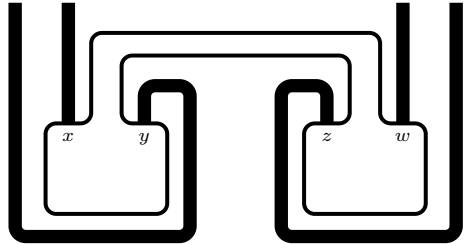
\begin{array}{c}
\begin{tikzpicture} [ xxx/.style={
	rectangle,
		rounded corners,
		minimum size=12mm,
		line width = 1.5pt},
		manyStrings/.style={
			line width=5pt}]
\draw [xxx] (0,0) -- ++(0, 1.2) -- ++(0.6,0) -- ++(0, 1.2) -- ++ (3.8,0) 
-- ++(0,-1.2) -- ++(0.6,0) -- ++(0,-1.2) -- ++(-1.6,0) -- ++(0, 1.2) -- ++(0.6, 0) -- ++(0,0.9) -- ++(-3.0,0)
-- ++(0,-0.9) -- ++(0.6,0) -- ++(0,-1.2) -- cycle;
%\draw [xxx] (0,0) -- (0,1.2) -- (1.2,1.2) -- (1.5,0.8) 
%-- (2.7,0.8) -- (3,1.2) -- (4.2,1.2) -- (4.2,0) -- (3,0) -- (2.7,0.4) -- (1.5,0.4) -- (1.2,0) -- cycle;
\draw [manyStrings] (0.3,1.2) -- +(0,1.6);
\node [below] at (0.3,1.2) {$\scriptstyle x$};
\draw  [manyStrings, rounded corners](1.3,1.2) -- ++(0,0.5) -- ++(0.6,0) -- ++(0, -2.0) -- ++(-2.3,0) -- ++(0,3.1);
\node [below] at (1.3,1.2) {$\scriptstyle y$};
\draw  [manyStrings, rounded corners](3.7,1.2) -- ++(0,0.5) -- ++(-0.6,0) -- ++(0,-2,0) -- ++(2.3,0) -- ++(0,3.1);
\node [below] at (3.7,1.2) {$\scriptstyle z$};
\draw  [manyStrings](4.7,1.2) -- +(0,1.6);
\node [below] at (4.7,1.2) {$\scriptstyle w$};
%\node at (-0.1,1.5) {$\cdots$};
%\node at (1.3, 1.5) {$\cdots$};
%\node at (2.9,1.5) {$\cdots$};
%\node at (4.3, 1.5) {$\cdots$};
\end{tikzpicture}
\end{array} 
$$} followed by an isotopy ``straightening'' the resulting picture).

%%
%%%
%\raisebox{-0.95in}[0.95in][0.95in]{
%\includegraphics[scale=1]{evaluating-graph-algebra.eps}}
Using these operations, any labeled tangle can be simplified leaving only strings connecting the outer disk to inner disks.  Each of these remaining strings can be removed by contributing a certain multiplicative factor according to \eqref{eq:simplificationRules}.

%%DS
As mentioned in the beginning of the section, we extensively make use of the following fact (see e.g. \cite{jones:graphPlanarAlg}):
\begin{proposition}
Let $\mathcal P$ be a subfactor planar algebra.  Then there exists a bipartite graph $\Gamma$ and a planar algebra embedding $i: \mathcal{P}\to\mathcal{P}^\Gamma$.  
\end{proposition}

Note that if $\mathcal{P}$ is embedded into $\mathcal{P}^\Gamma$, then any $B\in \mathcal{P}$ can be regarded as a function on the space of loops of $\Gamma$.  We denote by $\sigma_B(w)$ the value of this function on a loop $w\in L$.  Naturally, $\sigma_B(w)$ depends on the specific embedding of $\mathcal{P}$ (indeed, a fixed planar algebra $\mathcal{P}$ may admit embeddings into graph planar algebras $\mathcal{P}^\Gamma$ for many different $\Gamma$).

\subsubsection{Example 1 continued:  Temperley--Lieb algebra inside
the planar algebra of a bipartite graph}
A particular case of the planar algebra axiom is the natural embedding from $\TL(k,\pm)$ into
a linear span of loops, following \cite{GJS07}.   Indeed, Temperley--Lieb tangles are tangles with no input disks, and thus produce elements in 
any planar algebra.

Suppose that we are given a box $B$
with $2k$ boundary points (arranged so that all boundary points are
at the top and $*$ is at position $0$ from the top-left). Assume
also that there are $k$ non-crossing curves inside $B$ which connect
pairs of boundary points together. Let $\pi$ be the associated non-crossing
pairing. We let $L_B$ be the set of loops
in $\Gamma$ so that $w\in L_B$ iff $w=e_1\cdots e_{2k}$ 
with 

\begin{itemize}
\item
$e_n=e_\ell^o$ if $\{n,\ell\}$ is a block  of the partition 
$\pi$ (which is denoted $n\stackrel{\pi}{\sim}\ell$)
\item $s(e_1)\in V_+$ (resp. in $V_-$) if $B\in \TL(k,+)$ (resp.
$B\in \TL(k,-)$).
\end{itemize}
For $e\in E$, $\sigma(e):=\sqrt{\frac{\mu(t(e))}{\mu(s(e))}}$
and for
  $w\in L$, we define the weight
\begin{eqnarray}\label{eq:defOfSigmaE}
\sigma_{B}(w)&=&
\sigma(e_{i_1})\cdots\sigma(e_{i_n}) \textrm{ if }
e_{i_k}=e_{j_k}^{o}\textrm{ whenever }w\in L_B, \qquad i_k\stackrel{\pi}{\sim} j_k\textrm{ and }
\ i_k< j_k,\\ \nonumber
\sigma_B(w)&=&0\textrm{ if } w\notin L_B.
\end{eqnarray}

We then associate to the Temperley--Lieb tangle $B$ 
the element %A2: erase of 
 $w_B\in \mathcal{P}^\Gamma$  given by
 \begin{equation}\label{defwT}w_B(L)= 
\sigma_B(L)
.\end{equation} 
Informally, $w_B = \sum_{w\in L} \sigma_B(w) w$, where we identify a loop $w\in L$ with the delta function supported on that loop (viewed as an element in $\mathcal{P}^\Gamma$).  Of course, this is only correct when the graph is finite, so that the expression for $w_B$ is a finite sum.
\subsubsection{Example 2 continued: The planar algebra of two stitched Temperley--Lieb algebras
realized inside a graph planar algebra}\label{CoupledTLInside}
Let $\mathcal{P} = \TL(\delta_1) \copyrightsign \TL(\delta_2)$ as in \SectMark\ref{subsec:coupledTL}.  We'll realize $\mathcal{P}$
inside a graph planar algebra.  Let $\Gamma_r, \Gamma_b$ be two graphs so that the associated planar algebras contain $\TL(\delta_r)$ and $\TL(\delta_b)$, respectively.  We can thus choose $\Gamma_x$ to be $\mathbb{A}_\infty$ if
$\delta_x \geq 2$ and otherwise $\mathbb{A}_n$ for some $n$ (related to $\delta_x$). By appendix 
\ref{sec:couplings}, $\mathcal{P}$ embeds into the graph planar algebra of $\Gamma$.

Let $\Gamma=\Gamma_{r}\times\Gamma_{b}$.
More precisely, the vertices of $\Gamma$ are pairs $(v_{r},v_{b})$
with $v_{x}\in\Gamma_{x}$, $x\in\{b,r\}$. The pair $(v_r, v_b)$ is {\em even}\/ iff either both $v_r,v_b$ are even or both 
are odd; the pair $(v_r,v_b)$ is odd otherwise.
The edges of $\Gamma$
are of two kinds: the \emph{red} edges, consisting of pairs $(e,v)$
with $e$ an edge in $\Gamma_{r}$ and $v$ a vertex in $\Gamma_{b}$;
this is an edge from $(s(e),v)$ to $(t(e),v)$; and \emph{black}
edges, consisting of pairs $(f,w)$ with $f$ an edge in $\Gamma_{b}$
and $w$ a vertex in $\Gamma_{r}$; this edge goes from $(w,s(f))$
to $(w,t(f))$. For an edge in $\Gamma$, $e=(f,w)$, put $u(e)=f$
(which is in $\Gamma_{r}$ or $\Gamma_{b}$ according to whether $e$
is red or black).  
Note that $\Gamma$ is a bi-partite graph, since each edge in $\Gamma$ changes the parity 
of one of the components of a vertex $(v_r,v_b)$.  By Appendix \ref{sec:couplings}, $\Gamma$ is the principal graph of $\mathcal{P}$.

Let $\mu$ be the Perron--Frobenius eigenvector for $\Gamma$, given
at a vertex $(v,w)$ by the product of the eigenvectors of $\Gamma_{r}$
and $\Gamma_{b}$. For $e$ and edge of $\Gamma$, put $\sigma(e)=(\mu(t(e))/\mu(s(e)))^{1/2}$.
Let $c(x)$ be the color of the $x$-th boundary point of $T$.

Let $T\in{P}_{n}$ be a diagram, and let $$R_T=\{(x,y):\textrm{boundary points }x\textrm{ and }y\textrm{ are connected in }T\}.$$  The embedding of $\mathcal{P}$ into 
the graph planar algebra of $\Gamma$ is given by sending $T$ to the function $f_T\in\mathcal{P}^{\Gamma}$ given by:\[
f_{T}(e_{1}\cdots e_{2n})=\prod_{\substack{(x,y)\in R_T\\ x<y}}\sigma(e)\delta_{e\ \textrm{has same color as }c(x)}\ \delta_{u(e_{x})=u(e_{y})^{o}},\]
for any loop $e_{1}\cdots e_{2n}$ in $\Gamma$. 

\subsection{The observables $\mathscr{T}_t$} \label{sect:observables} 
From now on, we shall fix a subfactor planar algebra $\mathcal{P}$ as well as an embedding of $\mathcal{P}$ into the planar algebra $\mathcal{P}^\Gamma$ associated to a bipartite graph $\Gamma$.  Moreover, we shall fix elements $B_1,\dots,B_k\in \mathcal{P}$.  

For $t_i, 1\le i\le k$, small enough real numbers and any $S\in \mathcal{P}$, consider the observable (see the discussion in \SectMark\ref{sect:ArtitraryLoops})
$$\mathscr{T}_t(S) = \sum_{r_1,\dots,r_k} \prod_{i=1}^k \frac{(t_i)^{r_i}}{r_i !} \sum_{T\in T(r_1,\dots,r_k)} M_T(S,\underbrace{B_1,\dots,B_1}_{r_1},\dots,
\underbrace{B_k,\dots,B_k}_{r_k})  
$$ where $T(r_1,\dots,r_k)$ is the set of all possible planar tangles (configurations) with $1+r_1+\cdots +r_k$ input disks and 
$M_T(S,B_1,\dots,B_1,\dots,B_k,\dots,B_k)$ denotes the value (in $\mathcal{P}_0$) of the tangle $T$ applied to $S$ and $r_1$ copies of $B_1$, $r_2$ copies of $B_2$, etc.  

Since we have assumed that $\mathcal{P}$ is a {\em subfactor} planar algebra, $\mathcal{P}_0\cong \mathbb{C}$ and so the value of $\mathscr{T}_t(S)$ is a complex number.  However, the expression for $\mathscr{T}_t(S)$ makes sense for any $S\in \mathcal{P}^\Gamma$.  In this case, $M_T(S,B_1,\dots,B_1,\dots,B_k,\dots,B_k)$ is valued in $\mathcal{P}^\Gamma_0$, i.e., it is a function on the set of vertices of $\Gamma$ (this function is constant if $S$ actually happens to be in $\mathcal{P}^\Gamma$).  

\subsection{Random matrices associated with a graph}
In the sequel, we fix a graph $\Gamma$
with an eigenvalue $\delta$ and a Perron--Frobenius
eigenvector $\mu$ as in the previous part.
For
$e\in E_+$,  $A_e$ is a
$[M_{s(e)}]\times [M_{t(e)}]$ matrix with i.i.d entries
with variance $(M_{s(e)}M_{t(e)})^{-\frac{1}{2}}$  for some integer
numbers $(M_v, v\in V)$ so that
$$\lim_{M\ra\infty}\frac{M_v}{M}=\mu(v).$$
We put $A_{e^o}=(A_e)^*$
and denote, for a word $w=e_1\cdots e_k$,
$A_w=A_{e_1}\cdots A_{e_k}$. $\mu^M$ is the law of $A=(A_e,e\in E)$ chosen independently except for the last constraint.  In order that $A_w$ is 
a square matrix, we shall assume that $w$ is a loop.  

In the particular case that $w$ is the trivial path (consisting of no edges) starting and ending at $v$, we will denote by $A_w$ the identity matrix of size $M_v\times M_v$. 

The center-valued trace $\Tr_{0}$ on $Gr_{0}\mathcal{P}^{\Gamma}$
is given by the equation, \label{firstdefof:Tr0} for $x=e_1\cdots e_{2k}$
 \[
\Tr_{0}(x)(v) =1_{s(e_1)=v}
\langle x,T_{k}\rangle
\]
with $T_{k}=\sum_{B\in \TL(k)}w_B$ the sum of all $TL$ diagrams
with $k$ strings and where for two loops  
$x,y$ on $\Gamma$, 
$\langle x,y\rangle = \delta_{x=y}$. Thus $\Tr_0 (x) $  is a complex-valued function on the set of vertices of the graph.

We have the following theorem from \cite[Proposition 2]{GJS07}
\begin{theorem}\label{beta0}
Let  $v\in V$ 
and  $w=e_1e_2\cdots e_k\in L(v)$. Then
\begin{equation}\label{nm}\Tr_0(w)(v)=
\lim_{M\ra\infty}
\int {\rm tr}(A_w) d\mu^M(A)
\end{equation}
where $\tr$ is the normalized trace on matrices, ${\rm tr}(R)=N^{-1}\sum_{i=1}^N R_{ii}$ if $R$ is a
$N\times N$ matrix.
In the case where we extend linearly $Tr_0$ and take
 $w=w_T$ for some subfactor planar algebra element 
$T$, $Tr_{0}(w_T)=C(T) 1$ is constant and  $C(T)=\ourtr0(T)$. 
\end{theorem}
We used above the notations of \eqref{defsigmaT} and \eqref{defwT}.
Note that in \cite{GJS07}, 
we had an additional dimension $N$ so that $X^{M,N}_e$ converges
to  free variables entries as $N$ goes to infinity.
 This is however not needed 
since $M$ goes to infinity.

We shall give a new proof of Theorem \ref{beta0}
based on the so-called Schwinger--Dyson equation and obtain its generalization to the case of non-trivial interaction ($t\neq 0$):

\begin{theorem}\label{beta}Assume that the subfactor planar algebra is the Temperley--Lieb algebra, or is a planar subalgebra of a graph planar algebra of a finite graph (this is the case, for example, if it is finite-depth).
For $t =(t_1,\cdots,t_k)\in{\mathbb R}^k$ small enough,
there exists a probability measure $\mu_t^M$ on the set of matrices  indexed by $e\in E$
with dimensions $M_{s(e)}\times M_{t(e)}$ 
so that for 
  $v\in V$ 
and  $w=e_1e_2\cdots e_k\in L(v)$, there exists a limit
\begin{equation}\label{nm1}
\mytau_t (X_w)(v)=\lim_{M\ra\infty}
\int {\rm tr}(A_w) d\mu_t^M (A).
\end{equation}
where $X_w$ denotes the monomial $X_w=X_{e_1}\cdots X_{e_k}$.
In the case where $w=w_T$ for some Temperley--Lieb diagram $T$ (or, more generally, $T$ comes from a subfactor planar algebra inside 
$\mathcal{P}^\Gamma$), with $P^v_T=\sum_{w\in L(v)} \sigma_T(w) X_w$, 
$\mytau_{t}(P^v_{T})(v)=C_t(T)$  is constant for all $v\in V_+$ (resp. $V_-$)
if the marked point of $T$ is in an unshaded (resp. shaded)  region. Moreover,
  $C_t(T)=\ourtrt(T)$ is
 given by \eqref{trbeta}.
\end{theorem}
In the statement above, $\mytau_t$ is extended  by linearity and we used the notations of \eqref{defsigmaT} and \eqref{defwT}.
In the next subsections we shall prove this theorem.
\subsection{The case of finite graphs}

For $\delta=2\cos(\frac{\pi}{p})$, $p\ge 3$ (or, more generally, if the planar algebra under consideration is finite-depth), 
the graph $\Gamma$ can be chosen to be
finite.  This is the case for $\TL(\delta)$ if $\delta=2\cos(\frac{\pi}{n+1})$; the graph is then $\mathbb{A}_{n}$, the linear chain of $n$ 
vertices and $n-1$ edges.  Another such example is the case of the graph with two vertices joined by $n$ edges (see
e.g.~\cite[Examples 4.1 and 4.4]{jones:graphPlanarAlg}).
We study the finite-graph case first. 

\subsubsection{Definition of the matrix models} \label{firstdefof:gibbsMeasure}
Let $\mathcal{P}\subset \mathcal{P}^\Gamma$ be a subfactor planar algebra realized inside a graph planar algebra.
We consider the law $\mu^M$ of  $|E_+|$ independent 
$M_{s(e)}\times M_{t(e)}$  matrices with
i.i.d Gaussian entries 
with  variance $(M_{s(e)} M_{t(e)})^{-\frac{1}{2}}$. We denote by $\|R\|_\infty$
the spectral norm of a matrix $R$ (that
is, the spectral radius of $\sqrt{RR^*}$), and we consider
 $A=(A_e)_{e\in E_+}$, the collection of these
matrices and $\|A\|_\infty:=\max_{e\in E_+}\|A_e\|_\infty$. 
%%% DS: we have to explain actually what we mean by \sigma_B in case of a
%% general planar subalgebra
Let us assume that $\mathcal{P}\subset\mathcal{P}^\Gamma$ and $B_1,\dots,B_k\in \mathcal{P}$ are fixed.  To fix notations, let us write each $B_j$ as a linear combination of the basis for $\mathcal{P}^\Gamma$, consisting of all loops in the graph $\Gamma$: 
$$
B_j =   \sum_{w\in L}\sigma_{B_j}(w) w.
$$
(Note that if $\mathcal{P}=\TL$, then $\sigma_{B_j}(w)$ is exactly given by the formula \eqref{eq:defOfSigmaE}).  

For given real
numbers $t_1,\ldots,t_k$ and some $K>2$, we set
$$\mu^{M,K}_t(dA):=\frac{{\bf 1}_{\|A\|_\infty\le K}
}{Z^{M,K}_t} \exp\left\{M\sum_{i=1}^k t_i 
 \sum_{w\in L}\mu(s(w))\sigma_{B_i}(w)\Tr( A_{w})\right\}
\mu^M(dA),$$
where we denote by $s(w)$ the starting vertex of a loop $w\in L$.

\begin{example}\label{extangle}Consider the $\TL$ algebra
for $\delta \in \{2\cos(\frac{\pi}{p}), p\ge 3\}$.  Then $\TL$ can be embedded into the 
graph planar algebra $\mathcal{P}^{\Gamma}$ where 
$\Gamma = \mathbb{A}_n$.  We can consider the following potentials:
\noindent
% \begin{itemize}
(i) Let $B$ be the element of $\TL$ given by the tangle with only one string:
$$
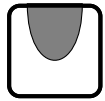
\begin{tikzpicture} [ xxx/.style={
	rectangle,
		rounded corners,
		minimum size=12mm,
		line width = 1.5pt}]
\node (A) at (0,0) [xxx,draw] {};
\draw [fill=gray] (A.120) .. controls +(down:10mm) and +(down:10mm) .. (A.60) -- cycle;
\node (AA) at (0,0) [xxx,draw] {};
\end{tikzpicture}
$$
We can either see it as a cup with black inside
and white outside or with the opposite shading,
both leading to the same  potential in the matrix model :
\begin{eqnarray*}
\sum_{v\in V^+} \mu(v)
 \sum_{e: s(e)=v}\sigma(e)\Tr( A_{e}A_{e}^*)&=& {\rm Tr}
\left(\sum_{\substack{v\in V_+ \\
s(e)=v}} \sqrt{\mu(v)\mu(t(e))} A_{e}A_{e}^*\right)\\
&=& \sum_{v\in V^-} \mu(v)
 \sum_{e: s(e)=v}\sigma(e)\Tr( A_{e}A_{e}^*)\\
\end{eqnarray*}

\noindent
(ii) Let $B$ be the element of $\TL$ given by a tangle with two strings,
 two white regions and one black: $$
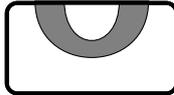
\begin{tikzpicture} [ xxx/.style={
	rectangle,
		rounded corners,
		minimum size=12mm,
		line width = 1.5pt}]
\node (A) at (0,0) [xxx,draw] {\hbox to2cm{\hfill}};
\draw [fill=gray] (A.140) .. controls +(down:10mm) and +(down:10mm) .. (A.40) -- cycle;
\draw [fill=white] (A.120) .. controls +(down:7mm) and +(down:7mm) .. (A.60) -- cycle;
\node (AA) at (0,0) [xxx,draw] {\hbox to2cm{\hfill}};
\end{tikzpicture}
$$
Then, it contributes to the potential by %according to \eqref{defwT}
%by $$ w_B=\sum_{e\in E_+} \sum_{f\in E_-}
%\sigma(e)\sigma(f) A_{e} A_f A_{f}^* A_{e}^*$$ with contribution to the potent%ial 
\begin{equation}\label{refdbcup}
\sum_{v\in V^+}\mu(v) \Tr\left[\sum_{\substack{e\\ s(e)=v}} \sum_{\substack{f\\ s(f)=t(e)}}
\sigma(e)\sigma(f) A_{e} A_f A_{f}^* A_{e}^*\right]=\sum_{v\in V^-}\mu(v)
\Tr\left[
\bigg(\sum_{\substack{e \\s(e)=v }}\sigma(e) A_eA_e^*\bigg)^2\right]\end{equation} 
where it is understood that 
products which make no sense
give no contribution. Inverting the shading amounts
to exchanging $V_+$ and $V_-$.
%\end{itemize}
\end{example}

\subsubsection{The main result}
To state our main result, as before, we fix a planar subalgebra $\mathcal{P}$ of a graph planar algebra $\mathcal{P}^\Gamma$ and, for some $B_i \in \mathcal{P}$, let $W=\sum t_i B_i$, $t_i \in \mathbb{R}$.  We let $\mathscr{T}_t$ be the observable functional considered in \SectMark\ref{sect:observables}. 
\begin{proposition}\label{propfinite}
(a) Let $K>2$. Then there exists $t(K)>0$ so that whenever
$\max_{1\le i\le k}
|t_i|\le t(K)$, for any $v\in V$ and for any loop $w\in L(v)$ there exists
a limit
$$ 
\mytau_t (X_w)(v) =\lim_{M\ra\infty} 
\mu^{M,K}_t(\tr(A_{w})).$$  
\\
(b) The Schwinger-Dyson equation (described below in \SectMark\ref{sssec:SD}) has a unique solution $\tau_t$. \\
(c) For any element $w\in \mathcal{P}^\Gamma$, $\tau_t(w)=\mytau_t(A_w) = \ourtrt(w)$.\\
(d) Let $\mathcal{P}\subset\mathcal{P}^\Gamma$ be a subfactor planar subalgebra and assume that $W\in P$.  Then for any $Q=\sum_w \alpha_w w \in \mathcal{P}$ and any $v\in \Gamma$, 
$\tau_t(Q)(v) = \mytau_t(\sum_w \alpha_w A_w )(v)= \ourtrt(Q)(v)$ are constant independent of $v$. 
\end{proposition}
The convergence of the 
matrix model is a small generalization
of  \cite[Theorem 3.5]{guionnet-edouard:combRM} (where
only Hermitian random matrices where considered),
whereas the identification of the limit is based on the analysis 
of the so-called Schwinger--Dyson (or loop) equations.
Note that in most papers
in the physics literature,
 the cutoff $K<\infty$ is not taken, leading sometimes to 
diverging integrals. The advantage of
adding this cutoff is that all integrals are well defined and 
moreover for small $t_i$'s, the Gibbs measure $\mu^{M,K}_t$
has a strictly log-concave density, providing many
interesting properties which allow to put on 
a firm mathematical ground the above convergence.
In fact, we can remove the cutoff in case 
the density is strictly log-concave.
$K$ has to be chosen strictly greater than $2$
(which is greater to the edge of the support of the 
 Pastur--Marchenko law) so that
the limit does not depend on
it for $t$ small enough. We start by recalling these
properties and sketching the proof of \cite[Theorem 3.5]{guionnet-edouard:combRM}.
\subsubsection{Convexity assumption and consequences}
For a loop $w\in L$, let us put $s(w)$ for the starting vertex of the loop.
As in \cite{guionnet-edouard:combRM}, we shall  restrict ourselves to the case 
where $\mu^{M,K}_t$ has a strictly log-concave density, which amounts to assume that 
the map from  the set of $|E_+|$ Hermitian matrices %
into $\mathbb R$ given by
\begin{equation}
\label{conve}(A_e)_{e\in E_+}\ra -\sum_{i=1}^k t_i 
 \sum_w \mu(s(w)) \sigma_{B_i}(w)\Tr( A_{w}) +\frac{1}{2}\sum_{e\in E}
\mu({s(e)})\sigma(e) \Tr(
A_eA_e^*)
\end{equation}
is strictly  convex (with Hessian bounded below by $ c$ for some $c>0$ 
independent of $M$) on $\{
\|A\|_\infty \le K\}$. This is always true for $t$ sufficiently small,
depending on $K$, with $c$ going to  $m=\min\{\mu(s(e))\sigma(e)=\sqrt{\mu(t(e))\mu(s(e))}, e\in E_+\}>0$ as
$t$ goes to zero.

As a consequence of strict convexity,
 we have concentration 
inequalities under $\mu^{M,K}_t$, see \cite[Section 6.3]{alice:StFlour}
namely for any $w\in L$,
 \begin{equation}\label{conc}
\mu^{M,K}_t\left(\left|\tr(A_w)-\mu^{M,K}_t
\left(\tr(A_w)\right)\right|^2\right)
\le \frac{C(w,K)}{M^2}.\end{equation}
The family $$\{\mu^{M,K}_t(\tr(A_w)), w\in L\}
$$ is tight. We
will denote by $\{\mytau^K_t(X_w):w\in L\}$ a limit point.
We also have Brascamp-Lieb inequalities, see \cite[Section 6.5]{alice:StFlour}, and so
by comparison to the Gaussian law for which we know that
the spectral radius is bounded with overwhelming probability, we  can
prove that there exists $\ell(c)$ (which only depends on $c$) so that
$$\mu^{M,K}_t\left(\max_{e\in E} \| A_e\|_\infty \ge \ell (c)
\right)\le e^{-\delta(c)M}$$
for some $\delta(c)>0$. We assume that  we have chosen $K>\ell(c)$ so that the limit
point $\mytau^K_t$ will not depend on $K$.  We denote in short 
$$\mytau_t(e_1e_2\cdots e_k)=
\mytau^K_t(X_{e_1}X_{e_2}\cdots X_{e_k})$$
for  any path $e_1e_2\cdots e_k$ in $G$. By definition the above vanishes 
if $s(e_1)\neq t(e_k)$. %A2: later on the notation between
%\tau(X_e and \tau(e where mixed; I preferred to choose one)
We next show any limit point satisfies the so-called 
Schwinger--Dyson equation and that this equation has
a unique solution when the $t_i$'s are small.

\subsubsection{Schwinger--Dyson (or loop) equations}\label{sssec:SD}
Let us fix $e\in E$ and $P=u$ with $u$ a path from $t(e)$ to $s(e)$.
By using integration by parts and concentration inequalities, we obtain (see
\cite[Section 8.1]{alice:StFlour}
or  \cite[Theorem 3.1]{guionnet-edouard:combRM}), that
\begin{multline*}
\lim_{M\ra\infty}
	\bigg[ \mu^{M,K}_t\left(\frac{\Tr}{M}\right) \left( 
		A_u \bigg(\mu({s(e)}) \sigma(e) A_e
		- \sum_{i=1}^k t_i 
		\sum_{w\in L}\mu(s(w)) \sigma_{B_i}(w) 
			( D_{X_{e^{o}}}A_{w})\bigg) \right)\\
	-\mu^{M,K}_t\left(\frac{\Tr}{M}\right)\otimes \mu^{M,K}_t\left(\frac{\Tr}{M}\right)
	\left(\partial_{X_{e^{o}}} (u)(A) \right)\bigg] %%DS:
		% I don't like the notation \partial(u)(A), but can't think of anything better
=0,\end{multline*}
where $\partial_{X_{e^{o}}}$ (resp.  $D_{X_{e^o}}$) is the non-commutative (resp. cyclic) derivative with respect to $X_{e^o}$. As before, we write $s(w)$ for the starting point of the loop $w$.

Recall that if $\omega$ is a path starting at a vertex $v$,  then $\mytau(\omega)$ is a limit point of $\mu(v)^{-1} \mu^M_t\frac{ \Tr}{M}(A_\omega)$, seen as a
linear form   on $ Gr_{0}\mathcal{P}^{\Gamma}$.  
Using this we see that $\mytau$ satisfies the following
Schwinger--Dyson equation: for every path $u$ from $t(e)$ to $s(e)$,
\begin{multline}\label{SD1}
\mu(t(e)) \mytau \left(
	u \bigg( \mu(s(e)) \sigma(e) \ e -  
		\sum_i t_i \sum_w \mu(s(w))\sigma_{B_i}(w) D_{e^{o}}(w) \bigg) \right)
		\\
-{\mu(s(e))\mu(t(e))} \mytau\otimes\mytau \left(\partial_{e^{o}} u\right)=0
\end{multline} 
where 
$$\partial_{e^o} w=\sum_{w=w_1e^o w_2} w_1\otimes w_2\quad
D_{e^o} w=\sum_{w=w_1e^o w_2} 1_{s(w_1)=t(w_2)}  w_2w_1 \,.
\label{eq:SD-allWords}$$ %A2
Moreover,  note that as $w\in L$,  $D_{e^o} w$ is a linear combination of paths starting at $s(e)$ and finishing at $t(e)$. Finally, considering the trivial loop $1_v$ at a vertex
$v$, we have that $\mytau(1_v)=1$ by construction.

Dividing \eqref{SD1} by $\mu(s(e))\mu(t(e))$ and then by $\sigma(e)$, we obtain the following equivalent equation:

\begin{equation}\label{SD1divided}
 \mytau(ue) = \mytau \left(u \bigg(
		\sum_i t_i \sum_w \sigma(e)^{-1} \frac{\mu(s(w))}{\mu(s(e))}
			\sigma_{B_i}(w) D_{e^{o}}(w) \bigg)  \right)
+ \sigma(e)^{-1} \mytau\otimes\mytau \left(\partial_{e^{o}} u\right)
\end{equation} 

We can use $\mytau$ to define a collection of linear maps from 
$\mathcal{P}^\Gamma$ with values in $\mathcal{P}^\Gamma$, which we will for now denote by $\mytau_{i,j}$.  

By definition, if $P\in \mathcal{P}^\Gamma$ is the delta function on the loop $a_1,a_2,\dots,a_{i},e_1,\dots,e_{j},b_1,\dots,b_{r}$, then $$\mytau_{i,j}(P)=\delta_{s(b_1)=t(a_{i})} 
\mytau ({e_1}\dots {e_j}) Q$$where $Q$ is the delta function supported on the loop 
$a_1\dots a_{i} b_1\dots b_{r}$. %%D We only deal with loops?? The above vanishes if $e_1\cdots e_j$ is not a loop. %A2:Dima check

Then the Schwinger--Dyson equation is
an equation on these linear maps. 
We will use the following graphical notation for the result of $\mytau_{i,j}$ applied to the delta function supported at the loop $a_1,\dots,a_i,e_1,\cdots,e_j, b_1,\dots,b_t,\in \mathcal{P}^\Gamma$ (we suppress $i,j$ since they can be read off from the numbers of the various strings):
\begin{equation}\label{def-of-mu}
\begin{array}{c}
\begin{tikzpicture} [ xxx/.style={
	rectangle,
		rounded corners,
		minimum size=12mm,
		line width = 1.5pt}]
\node (A) at (0,0) [xxx,draw] {$\qquad\qquad\quad A\quad\qquad\qquad$};
\node [below] at (node cs:name=A, angle=160) {$\scriptstyle a_1$};
\node [above] at (node cs:name=A, angle=155) {$\scriptstyle \cdots$};
\node [below] at (node cs:name=A,angle=150) {$\scriptstyle a_i$};
\node [below] at (node cs:name=A, angle=20) {$\scriptstyle b_t$};
\node [above] at (node cs:name=A, angle=25) {$\scriptstyle \cdots$};
\node [below] at (node cs:name=A,angle=30) {$\scriptstyle b_1$};
\node [below] at (node cs:name=A, angle=125) {$\scriptstyle e_1$};
\node [below] at (node cs:name=A, angle=100) {$\scriptstyle e_2$};
\node [above] at (node cs:name=A,angle=75) {$\scriptstyle \cdots $};
\node [below] at (node cs:name=A,angle=55) {$\scriptstyle e_j$};
\node (Tr) at (0,2) [xxx,draw]{$\mytau$};
\foreach \ang in {160,150,20,30} {
	\draw (node cs:name=A, angle=\ang) -- +(0,1.5);
}
\foreach \ang in {125, 100, 55} {
	\draw (node cs:name=A, angle=\ang) .. controls +(up:0.2) and +(-90:0.2) .. (node cs:name=Tr, angle=-\ang);
}
\end{tikzpicture}
\end{array}
\end{equation}

We have the following lemma.
 
\begin{lemma}\label{picture-SD}
The Schwinger--Dyson equation 
 \eqref{SD1}  is equivalent to the following diagrammatic equation, which is required to hold whenever some arbitrary element of ${\mathcal P}^\Gamma$ is inserted into the blank boxes (same element to each blank box) and the strings at the bottom of each of the three pictures are connected to an arbitrary element of ${\mathcal P}^\Gamma$ (same element for each picture):
\begin{gather}\label{SDP}
\begin{array}{c}
\begin{tikzpicture} [ xxx/.style={
	rectangle,
		rounded corners,
		minimum size=12mm,
		line width = 1.5pt},
		manyStrings/.style={
			line width=5pt}]
\node (A) at (0,0) [xxx,draw]{$\quad\ \ \quad$};
\node (Tr) at (0,2) [xxx,draw]{$\quad\  \mytau\ \quad$};
%\foreach \ang in {130,110,70,50} {
%	\draw (node cs:name=A,angle=\ang) .. controls +(up:0.2) and +(-90:0.2) .. (node cs:name=Tr, angle=-\ang);
%}
%\node [above] at (A.north) {$\scriptstyle \cdots$};
\draw [manyStrings] (Tr.south) -- (A.north);
%\foreach \ang in {-130,-70} {
%	\draw (node cs:name=A,angle=\ang) -- +(0,-0.5);
%}
\draw [manyStrings] (A.south) -- +(0,-0.5);
%\node [below] at (node cs:name=A,angle=-105) {$\scriptstyle\cdots$};
%
%
\draw [rounded corners] (node cs:name=Tr,angle=-50) 
-- +(0,-0.5)
-- +(0.75,-0.5)
-- +(0.75,-2.3)
-- +(0,-2.3)
-- +(0,-2.5);
\end{tikzpicture}
\end{array}
= 
\sum_{
	i\textrm{ odd}
} {}
\begin{array}{c}
\begin{tikzpicture}  [ xxx/.style={
	rectangle, draw,
		rounded corners,
		minimum size=12mm,
		line width = 1.5pt},
		manyStrings/.style={
			line width=5pt}]
\node (A) at (0,0) [xxx]{$\qquad\qquad\qquad\qquad\qquad$};
\node (Tr) at (-1.5,2) [xxx] {$ \mytau$};
\node  (Trr) at (1.5,2) [xxx]{$\mytau $};
\draw [manyStrings] (A.south) -- +(0,-0.5);
\draw [manyStrings] (Tr.south) .. controls +(down:0.3) and +(up:0.3) .. 
	(node cs:name=A, angle=150);
\draw [manyStrings] (Trr.south) .. controls +(down:0.3) and +(up:0.3) .. 
	(node cs:name=A, angle=30);
\draw[rounded corners] (A.north) -- +(0,2.1) -- +(2.7,2.1)
-- +(2.7,-1.4) -- +(0.3,-1.4) -- +(0.3,-1.75);
\node [below] at (A.north) {$\circled{i}$};
\end{tikzpicture}
\end{array}  %\\  \nonumber 
%\qquad\qquad 
+ \sum_{i \textrm{ even}}
\begin{array}{c}
\begin{tikzpicture}  [ xxx/.style={
	rectangle, draw,
		rounded corners,
		minimum size=12mm,
		line width = 1.5pt},
		manyStrings/.style={
			line width=5pt}]
\node (Tr) at (0,2) [xxx] {$\qquad \mytau\qquad$};
\node (A) at (-1,0) [xxx]  {};
\node (W) at (1,0) [xxx] {${W}$};
\draw [manyStrings] (A.north) .. controls +(up:0.4) and +(down:0.4) .. (node cs:name=Tr, angle=-120);
\draw [manyStrings] (A.south) -- +(0,-0.7);
\draw [manyStrings] (node cs:name=W,angle=120) .. controls +(up:0.4) and +(down:0.4) .. (node cs:name=Tr, angle=-60);
\draw [manyStrings, rounded corners] (node cs:name=W,angle=60) -- +(0,0.25) -- +(1.0,0.25) -- +(1.0,-1.50) -- (0,-0.85) -- (Tr.south);
\draw [rounded corners] (W.north) -- +(0,0.5) -- +(1.75,0.5) -- +(1.75,-1.73) -- (-0.7,-1.1) -- (-0.7,-1.325);
\node [below] at (W.north) {$\circled{i}$};
\end{tikzpicture}
\end{array}
\end{gather}where $W=\sum_i t_i B_i = \sum_i \sum_{w} \sigma_{B_i}(w) w$\end{lemma}

Here we use the following conventions.  Thick lines indicate an arbitrary number of strings.  Also, given $ 
\begin{array}{c}
\begin{tikzpicture} [ xxx/.style={
	rectangle,
		rounded corners, draw,
		minimum size=4mm,
		line width = 1.5pt},
		manyStrings/.style={
			line width=5pt}]
\node[xxx] (A) at (0,0){$\quad x\quad $};
\draw[manyStrings] (A.north) -- + (0,0.3);
\node[right] at (A.165){\scriptsize$\raisebox{-1.5em}{*}$};
\end{tikzpicture}
\end{array} \in \mathcal{P}^\Gamma$ we'll set
$
\begin{array}{c}
\begin{tikzpicture} [ xxx/.style={
	rectangle,
		rounded corners, draw,
		minimum size=8mm,
		line width = 1.5pt},
		manyStrings/.style={rounded corners,
			line width=5pt}]
\node[xxx] (A) at (0,0){$\quad x\quad $};
\draw[manyStrings] (A.south) -- +(0,-0.3);
\draw[manyStrings] (A.north) -- + (0,0.3);
\end{tikzpicture}
\end{array} =
\begin{array}{c}
\begin{tikzpicture} [ xxx/.style={
	rectangle,
		rounded corners, draw,
		minimum size=8mm,
		line width = 1.5pt},
		manyStrings/.style={rounded corners,
			line width=5pt}]
\node[xxx] (A) at (0,0){$\quad x\quad $};
\draw[manyStrings] (A.130) -- ++(0,0.2) -- ++(-0.8,0) -- ++(0,-1.2) -- ++(1.1,0) -- ++(0,-0.2) ;
\draw[manyStrings] (A.north) -- + (0,0.3);
\draw[manyStrings] (A.50) -- ++(0,0.2) -- ++(0.8,0) -- ++(0,-1.2) -- ++(-1.1,0) -- ++(0,-0.2) ;
\node[below] at (A.145){\scriptsize$\raisebox{-0.8em}{\ \ *}$};
\end{tikzpicture}
\end{array} 
$, where by convention there is an arbitrary fixed number (same in all diagrams) of strings between the marked initial segment of $x$ and the last string that is bent around $x$ on the left.

\begin{proof} Consider  $e,f,g,e_1,\dots,e_n\in E(\Gamma)$ so that the paths $e_1,e_2,\dots,e_k,g$ and $e,f$ form loops.  Consider the following planar operation
applied to $P=e_1,\dots,e_k,g$ and $e,f$ (more precisely, $P$ is the element of $\mathcal{P}^\Gamma$ which is the delta function on the loop $e_1,\dots,e_k,g$, etc.):
\begin{equation}\label{comp-with-boxes}
\begin{array}{c}
\begin{tikzpicture}  [ xxx/.style={
	rectangle, draw,
		rounded corners,
		minimum size=12mm,
		line width = 1.5pt},
		manyStrings/.style={
			line width=5pt}]
\node (A) at (0,0) [xxx]  {$\qquad\qquad\qquad\qquad$};
\node (B) at (3,0) [xxx] {$\quad$};
\foreach \ang in {150,110,70,30} {
	\draw (node cs:name=A, angle=\ang) -- +(0,0.5);
}
	\node [below] at (node cs:name=A,angle=150) {$\scriptstyle e_1$};
	\node [below] at (node cs:name=A,angle=120) {$\scriptstyle e_2$};
	\node [below] at (node cs:name=A,angle=60) {$\scriptstyle e_{k-1}$};
	\node [below] at (node cs:name=A,angle=30) {$\scriptstyle e_k$};
	\node [above] at (node cs:name=A,angle=90) {$\scriptstyle \cdots$};

\draw (B.north) -- +(0,0.5);
\draw[rounded corners] (A.south) -- +(0,-0.25) -- +(3,-0.25) -- (B.south);
\node [above] at (A.south) {$g$};
\node [above] at (B.south) {$f$};
\node [below] at (B.north) {$e$};
\end{tikzpicture}
\end{array}
= \delta_{f=g^o} \sigma(f)
\begin{array}{c}
\begin{tikzpicture}  [ xxx/.style={
	rectangle, draw,
		rounded corners,
		minimum size=12mm,
		line width = 1.5pt},
		manyStrings/.style={
			line width=5pt}]
\node (A) at (0,0) [xxx]  {$\qquad\qquad\qquad\qquad$};
\foreach \ang in {150,110,70,30,22} {
	\draw (node cs:name=A, angle=\ang) -- +(0,0.5);
}
	\node [below] at (node cs:name=A,angle=150) {$\scriptstyle e_1$};
	\node [below] at (node cs:name=A,angle=120) {$\scriptstyle e_2$};
	\node [below] at (node cs:name=A,angle=60) {$\scriptstyle e_{k-1}$};
	\node [below] at (node cs:name=A,angle=30) {$\scriptstyle e_k$};
		\node [below] at (node cs:name=A,angle=22) {$\scriptstyle e$};
	\node [above] at (node cs:name=A,angle=90) {$\scriptstyle \cdots$};
	\node at (0,-0.7) {};
\end{tikzpicture}
\end{array}
\end{equation}
Let $e_1\dots e_k\ g\ a_n\dots a_1$ and $e\ f$ be two loops in $\Gamma$.  Consider the following equation:
\begin{gather}\label{SDef}
\begin{array}{c}
\begin{tikzpicture} [ xxx/.style={
	rectangle,
		rounded corners,
		minimum size=12mm,
		line width = 1.5pt},
		manyStrings/.style={
			line width=5pt}]
\node (A) at (0,0) [xxx,draw]{$\qquad\qquad\qquad$};
\node [below] at (A.north) {$\scriptstyle{e_1\cdots e_k}$};
\node [above] at (node cs:name=A, angle=-130) {$\scriptstyle{a_1\cdots a_n}$};
\node (Tr) at (0,2) [xxx,draw]{$\mytau$};
%\foreach \ang in {130,110,70,50} {
%	\draw (node cs:name=A,angle=\ang) .. controls +(up:0.2) and +(-90:0.2) .. (node cs:name=Tr, angle=-\ang);
%}
%\node [above] at (A.north) {$\scriptstyle \cdots$};
\draw [manyStrings] (Tr.south) -- (A.north);
%\foreach \ang in {-130,-70} {
%	\draw (node cs:name=A,angle=\ang) -- +(0,-0.5);
%}
\draw [manyStrings] (node cs:name=A, angle=-130) -- +(0,-0.5);
%\node [below] at (node cs:name=A,angle=-105) {$\scriptstyle\cdots$};
%
%
\node (B) at (2,0) [xxx,draw]{$$};
\node [below] at (B.north) {$\scriptstyle e$};
\node [above] at (B.south) {$\scriptstyle f$};
\node [above] at (A.south) {$\scriptstyle g$};
\draw (node cs:name=Tr,angle=-60) .. controls +(down:0.3) and +(up:0.3) .. (B.north);
\draw [rounded corners] (A.south) -- +(0,-0.25) -- +(2,-0.25) -- (B.south);
\end{tikzpicture}
\end{array}
= 
\sum_{
	i\textrm{ odd}
} {}
\begin{array}{c}
\begin{tikzpicture}  [ xxx/.style={
	rectangle, draw,
		rounded corners,
		minimum size=12mm,
		line width = 1.5pt},
		manyStrings/.style={
			line width=5pt}]
\node (A) at (0,0) [xxx]{$\qquad\qquad\qquad\qquad\qquad\qquad$};
\node (Tr) at (-1.5,2) [xxx] {$ \mytau$};
\node  (Trr) at (1.5,2) [xxx]{$\mytau $};
\node (B) at (0.5,-1.7) [xxx]{};
\node[below] at (B.north){$\scriptstyle e$};
\node[above] at (B.south){$\scriptstyle f$};
\node [above] at (node cs:name=A, angle=-160) {$\scriptstyle a_1\cdots a_n$};
\draw [manyStrings] (node cs:name=A, angle=-160) -- +(0,-1.5);
\draw [manyStrings] (Tr.south) .. controls +(down:0.3) and +(up:0.3) .. 
	(node cs:name=A, angle=150);
\node [below] at (node cs:name=A, angle=150) {$\scriptstyle e_1\cdots e_{i-1}$};
\draw [manyStrings] (Trr.south) .. controls +(down:0.3) and +(up:0.3) .. 
	(node cs:name=A, angle=30);
\node [below] at (node cs:name=A, angle=30) {$\scriptstyle e_{i+1}\cdots e_k$};
\draw[rounded corners] (A.north) -- +(0,2.1) -- +(2.7,2.1)
-- +(2.7,-1.4) -- +(0.5,-1.4) -- (B.north);
\node [below] at (A.north) {$\scriptstyle e_{i}$};
\node [above] at (node cs:name=A, angle=-130) {$\scriptstyle g$};
\draw [rounded corners] (node cs:name=A, angle=-130) -- +(0,-2) -- +(1,-2) -- (B.south);
\end{tikzpicture}
\end{array} \\ \nonumber
\qquad\qquad + \sum_{i \textrm{ even}}
\begin{array}{c}
\begin{tikzpicture}  [ xxx/.style={
	rectangle, draw,
		rounded corners,
		minimum size=12mm,
		line width = 1.5pt},
		manyStrings/.style={
			line width=5pt}]
\node (Tr) at (0,2) [xxx] {$\qquad \mytau\qquad$};
\node (A) at (-1.2,0) [xxx]  {$\qquad\qquad$};
\node (W) at (2,0) [xxx] {$\quad\qquad W\qquad\quad $};
\node (B) at (-0,-2) [xxx] {};
\node [below] at (A.north){$\scriptstyle e_1\cdots e_k$};
\draw [manyStrings] (A.north) .. controls +(up:0.4) and +(down:0.4) .. (node cs:name=Tr, angle=-120);
\draw [manyStrings] (node cs:name=A, angle=-120) -- +(0,-2.2);
\node [above] at (node cs:name=A, angle=-120){$\scriptstyle a_1\cdots a_n$};
\draw [manyStrings] (node cs:name=W,angle=150) .. controls +(up:0.4) and +(down:0.4) .. (node cs:name=Tr, angle=-60);
\draw [manyStrings, rounded corners] (node cs:name=W,angle=30) -- +(0,0.25) -- +(0.7,0.25) -- +(0.7,-1.50) -- (0,-0.85) -- (Tr.south);
\draw [rounded corners] (W.north) -- +(0,0.5) -- +(2,0.5) -- +(2,-1.73) -- (-0,-1.1) -- (B.north);
\node [below] at (W.north) {$\scriptstyle \circled{i}$};
\node [below] at (B.north) {$\scriptstyle e$};
\node [above] at (B.south) {$\scriptstyle f$};
\node [above] at (node cs:name=A, angle=-70) {$\scriptstyle g$};
\draw [rounded corners] (node cs:name=A,angle=-70) -- +(0,-2.2) -- +(0.95,-2.2) -- (B.south);
\end{tikzpicture}
\end{array}
\end{gather}

It is easily seen (by noticing that $g$ and $f$ are arbitrary subject to $t(f)=s(g)$)  that \eqref{SDef} is equivalent to \eqref{SDP}.

Next, note that both sides of \eqref{SDef} 
are zero unless the region containing the point at infinity 
 in all diagrams is labeled
by $v=t(e)$ (so in particular $t(a_1)=s(a_n)=v$).    
Consider now the three diagrams comprising equation \eqref{SDef}. 
Put $$R=\delta_{f=g^o}\sigma(f) a_1\dots a_n.$$
The diagram in the upper-left corner is exactly
$$Q_1=\mytau(e_1\dots e_k e)R.$$
Consider now the  diagram in the upper-right corner.  The term in the summation is zero unless $e_i = e^o$. %A2
  In particular, the region to the right of the string emanating from $e_i$ can be labeled by $s(e)$
and the region to the right of the string emanating from $e_{i-1}$ %A2
can be labeled $t(e)$.  By definition of $\mytau_{ij}$, we then get that this diagram is equal to
$$ Q_2=\mytau\otimes\mytau \left(\sum_{i \textrm{ odd}} e_1 \dots e_{i-1}\otimes e_{i+1}\dots e_k \right) \sigma(e)^{-1}R.$$
Note that we can actually replace the sum over odd $i$ by the sum over all $i$, since $\mytau(e_1\dots e_{i-1})=0$ unless $i$ is odd.

 Assume first that $W$ is the delta function supported on the
 loop $g_1 \dots g_s$.   Let us now consider for $i$ even the diagram
  $$ D_i=\begin{array}{c}
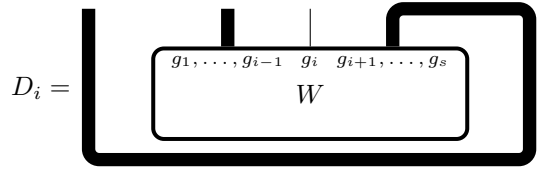
\begin{tikzpicture} [ xxx/.style={
	rectangle,draw,
		rounded corners,
		minimum size=12mm,
		line width = 1.5pt},
		manyStrings/.style={rounded corners,
			line width=5pt}]
\node (W) at (0,0)[xxx] {$\qquad\qquad\quad  W\quad \qquad \qquad$};
\draw[manyStrings](W.150) -- +(0,0.5);
\draw[manyStrings](W.30) -- ++(0,0.5) -- ++(1.8,0) -- ++(0, -2) -- ++(-5.8,0) -- ++(0,2);
\draw(W.north) -- +(0,0.5);
\node[below] at (W.150){\scriptsize ${ g_1,\dots,g_{i-1}}$};
\node[below] at (W.30){\scriptsize $g_{i+1},\dots,g_{s}$};
\node[below] at (W.north){\scriptsize  $ g_i$};
 \end{tikzpicture}
 \end{array}
$$
Since all strings emanating from $g_{i+1},\dots,g_s$ make a $360^o$ rotation in the drawing,
 according to \eqref{eq:simplificationRules} we have  $$D_i = \sigma(g_{i+1})^2 \cdots \sigma(g_s)^2 D'_i$$
where $$ 
D'_i= \begin{array}{c}
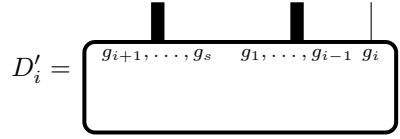
\begin{tikzpicture} [ xxx/.style={
	rectangle,draw,
		rounded corners,
		minimum size=12mm,
		line width = 1.5pt},
		manyStrings/.style={rounded corners,
			line width=5pt}]
\node (W) at (0,0)[xxx] {$\qquad\qquad\quad  \ \ \ \quad \qquad \qquad$};
\draw[manyStrings](W.150) -- +(0,0.5);
\draw[manyStrings](W.40) -- +(0,0.5);
\draw (W.20) -- +(0,0.5);
\node[below] at (W.150){\scriptsize ${ g_{i+1},\dots,g_{s}}$};
\node[below] at (W.40){\scriptsize $g_{1},\dots,g_{i-1}$};
\node[below] at (W.20){\scriptsize  $ g_i$};
 \end{tikzpicture}
 \end{array}
 $$ 
 To see this, we can compare the results of glueing all strings of $D_i$ to some diagrams $B_1,\dots,B_k$ versus the results of glueing the corresponding strings of $D'_i$ to the same diagrams.  As we apply  \eqref{eq:simplificationRules}  to remove strings, the strings starting from $D_i$ associated to $g_j$, $i+1\leq j\leq s$  contribute an extra factor of $\sigma(g_j)^2$ as compared to their contribution if they were to start from $D'_i$.
 
 Using this, the value of the bottom diagram in \eqref{SDef} is
 \begin{eqnarray*}
Q_3&=&
\sum_{i \textrm{ even}}  \sigma(g_{i+1})^2\dots \sigma(g_s)^2 
 \mytau(e_1 \dots e_k g_{i+1}\dots g_{s} g_{1}\dots {g_{i-1}})
 \delta_{e^o=g_i}\sigma(e)^{-1} R \\
&=&
\sum_{i \textrm{ even}}  \frac{\mu(s(g_1))}{\mu(t(g_i))}  \mytau(e_1 \dots e_k g_{i+1}\dots g_{s} g_{1}\dots {g_{i-1}})\delta_{e^o=g_i}\sigma(e)^{-1}R\\
& =&
 \frac{\mu(s(g_1))}{\mu(s(e))} 
 \mytau(e_1 \dots e_k  D_{e^o}W) \sigma(e)^{-1}R\\
\end{eqnarray*}
By linearity, the same equation holds for arbitrary $W$: 
\begin{equation*}
Q_3= \sum_i t_i \sum_w \sigma_{B_i}(w) 
 \frac{\mu(s(w))}{\mu(s(e))} 
 \mytau(e_1 \dots e_k  D_{e^o}w) \sigma(e)^{-1}R\\
\end{equation*}

Now, \eqref{SDef} is equivalent to saying that $$Q_1 = Q_2 + Q_3.$$
Since $e,f,a_1,\dots,a_n$ are arbitrary, we see that  \eqref{SDef} is equivalent to \eqref{SD1divided} and thus to \eqref{SD1}
\end{proof} 

We next prove the uniqueness of the solutions to the Schwinger-Dyson equation in a small parameters regime.

\begin{lemma}\label{uniq}
 For any $c>0$ and $K>\ell(c)$,
there exists $t(c,K)$ so that
when $|t_i|\le t(c,K)$ for all $i\in\{1,\ldots,k\}$,
there exists a unique solution  $\mytau$
to \eqref{SD1}  which satisfies
\begin{itemize}
\item $\mytau(xx^*)\ge 0$, 
\item $\mytau({1_v})=1$, %%DS: we are dealing with *normalized* trace
where $1_v$ is the trivial path consisting of no edges starting and ending at $v\in \Gamma$,
\item   $\mytau((e^0e)^{p})
\le K^p$ for all $e\in E_+$ and $p\in\mathbb N$.
\end{itemize}
\end{lemma}
\begin{proof} The proof follows the arguments of \cite[Section 2.4]{guionnet-edouard:combRM}
that we however repeat for further uses.
Existence is already guaranteed. To prove uniqueness
we assume we have two solutions $\tau_1$ and $\tau_2$ and let
$$\Delta(n)=\sup_{w\in L, |w|\le n}|\tau_1(w)-\tau_2(w)|$$
with $|w|$ the number of letters
in the word $w$. 
Then equation \eqref{SD1}
implies that, if we let $m=\min \left(\mu({s(e)})\mu(t(e))\right)^{\frac{1}{2}}>0$,
\begin{equation}\label{SDu}
\Delta(n+1)\le 2 m^{-1} \sum_{p=1}^{n-1}\Delta(p) K^{n-1-p}
+A(t) \Delta(n+D-1)
\end{equation}
with $D$ the degree of $W$ and if $D_{e^0} W=\sum_it_i \sum_{j=1}^{k_i^e}
q_{ij}^e$ with some monomials $q_{ij}^e$ with degree at most $D-1$, then
$$A(t)=\max_e (\mu({s(e)})\mu(t(e)))^{-\frac{1}{2}}
\sum_i|t_i |{k_i^e}$$
For $\gamma<1/K$, the sum $\Delta(\gamma)=\sum_{p\ge 1}\gamma^p \Delta(p)$ is finite
and satisfies
$$\Delta(\gamma)\le \frac{\gamma^2}{1-\gamma K} \Delta(\gamma)+A(t)\gamma^{-D+2}\Delta(\gamma).$$
We then choose $t$ small enough
so
that $$ \frac{\gamma^2}{1-\gamma L} +A(t)\gamma^{-D+2}<1$$
for some $\gamma\in ]0,L^{-1}[$, which guarantees that $\Delta(\gamma)=0$
and therefore $\delta(n)=0$ for all $n\ge 0$.
\hfill\qed \end{proof}

In the next section we study  Schwinger--Dyson equations for laws
on $\mathcal{P}^\Gamma$ and $P$.

\subsubsection{Proof of Theorem \ref{beta0} in the finite graph case}
In the case $t=0$, we deduce Theorem \ref{beta0}. First, note that the previous lemma already entails the convergence of $\int M^{-1}\Tr(A_w)d\mu^M(A)$, $w\in L$,
to the unique solution of the Schwinger-Dyson equation $\tau$. Hence, we only need to identify this limit. It is enough to show that if we define
$(\Tr_0)_{ij}$   as $\tau_{ij}$ but with $\Tr_0$ instead of  $\tau$, 
it will satisfy the same Schwinger-Dyson equation.
Equation \eqref{SDP} implies, for $t=0$
$$
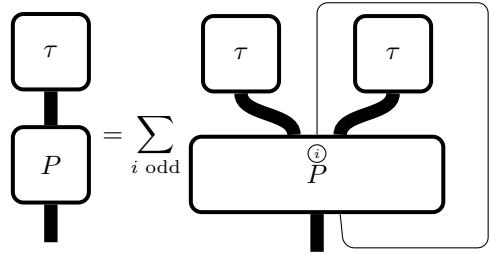
\begin{array}{c}
\begin{tikzpicture}  [xxx/.style={
	rectangle, draw,
		rounded corners,
		minimum size=10mm,
		line width = 1.5pt},
		manyStrings/.style={
			line width=5pt}]
\node at (0,0) (P) [xxx] {$P$};
\node at (0,1.5) (Tr) [xxx] {$\tau$};
\draw [manyStrings] (P.north) -- (Tr.south);
\draw [manyStrings] (P.south) -- +(0,-0.5);
\end{tikzpicture}
\end{array}
= \sum_{i \textrm{ odd}}
\begin{array}{c}
\begin{tikzpicture}  [ xxx/.style={
	rectangle, draw,
		rounded corners,
		minimum size=10mm,
		line width = 1.5pt},
		manyStrings/.style={
			line width=5pt}]
\node at (0,0) (P) [xxx] {$\qquad\qquad P\qquad\qquad$};
\node at (-1,1.6) (Tr) [xxx] {$\tau$};
\node at (1,1.6) (Trr) [xxx] {$\tau$};
\draw [manyStrings] (P.south) -- +(0,-0.5);
\draw [manyStrings] (node cs:name=P, angle=120) .. controls +(up:0.3) and +(down:0.3) .. (Tr.south);
\draw [manyStrings] (node cs:name=P, angle=60) .. controls +(up:0.3) and +(down:0.3) .. (Trr.south);
\draw [rounded corners] (P.north) -- +(0,1.75) -- +(2.25,1.75) -- +(2.25, -1.5) -- +(0.35, -1.5) -- (node cs:name=P, angle=-60);
\node [below] at (P.north) {$\scriptscriptstyle \circled{i}$};
\end{tikzpicture}
\end{array}
$$
This equation clearly has a unique solution, since it defines 
the maps $\tau_{i,j}$ recursively in terms of $\tau_{i,j'}$ with $j<j'$.  
We claim that in fact $\tau_{i,j}=(\Tr_0)_{i,j}$ where this notation means that we
replaced $\mytau$ by $\Tr_0$ in the definition of $\Tr'_{i,j}$.
Thus $(\Tr_0)_{i,j}$ is given by
$$
Tr_0 = 
\begin{array}{c}
\begin{tikzpicture}  [ xxx/.style={
	rectangle, draw,
		rounded corners,
		minimum size=8mm,
		line width = 1.5pt},
		manyStrings/.style={
			line width=5pt}]
\node at (0,0) (P) [xxx] {$P$};
\node at (0,1.5) (Tr) [xxx] {$\sum TL$};
\draw [manyStrings] (P.north) -- (Tr.south);
\draw [manyStrings] (P.south) -- +(0,-0.5);
\end{tikzpicture}
\end{array}
$$
where $\sum TL$ stands for the sum of all $\TL$ diagrams.  Here the $i+1$-st string of $P$ is the first one of the $j$ strings connected to $\sum \TL$.
If we follow the rightmost top string of $P$, it will be connected to one of the other 
vertical strings of $P$ (and, for parity reasons, it will be an odd string).  From this we see 
that $\tau_{i,j}$ satisfy the same recursive relation as $(\Tr_0)_{i,j}$.
\qed

\subsection{Proof of Proposition \ref{propfinite}}
Let us now consider the case  $t\neq 0$.  We have already shown part (b).  We have also shown that any limit point of $\left\{\mu_t^{M,K}(\tr(\cdot))\right\}_{M\to \infty}$ satisfies the Schwinger--Dyson equation; it follows from uniqueness of the solution that all such limit points are equal.  Thus the limit
$$ 
\mytau_t (X_w)(v) =\lim_{M\ra\infty} 
\mu^{M,K}_t(\tr(A_{w}))$$  
exists and we obtain part (a) of the Proposition.

To verify part (c), it is sufficient to show that $(\ourtrt)_{i,j}$ satisfies the Schwinger--Dyson equation:
\begin{equation}
\begin{array}{c}
\begin{tikzpicture} [ xxx/.style={
	rectangle,
		rounded corners,
		minimum size=8mm,
		line width = 1.5pt},
		manyStrings/.style={
			line width=5pt}]
\node (A) at (0,0) [xxx,draw]{$\qquad P \qquad$};
\node (Tr) at (0,2) [xxx,draw]{$\qquad \ourtrt \qquad$};
%\foreach \ang in {130,110,70,50} {
%	\draw (node cs:name=A,angle=\ang) .. controls +(up:0.2) and +(-90:0.2) .. (node cs:name=Tr, angle=-\ang);
%}
%\node [above] at (A.north) {$\scriptstyle \cdots$};
\draw [manyStrings] (Tr.south) -- (A.north);
%\foreach \ang in {-130,-70} {
%	\draw (node cs:name=A,angle=\ang) -- +(0,-0.5);
%}
\draw [manyStrings] (A.south) -- +(0,-0.5);
%\node [below] at (node cs:name=A,angle=-105) {$\scriptstyle\cdots$};
%
%
\draw [rounded corners] (node cs:name=Tr,angle=-40) 
-- +(0,-0.5)
-- +(1,-0.5)
-- +(1,-2.3)
-- +(0,-2.3)
-- (node cs:name=A, angle=-40);
\end{tikzpicture}
\end{array}
= 
\sum_{
	i\textrm{ odd}
} {}
\begin{array}{c}
\begin{tikzpicture}  [ xxx/.style={
	rectangle, draw,
		rounded corners,
		minimum size=12mm,
		line width = 1.5pt},
		manyStrings/.style={
			line width=5pt}]
\node (A) at (0,0) [xxx]{$\qquad\qquad P \qquad\qquad$};
\node (Tr) at (-1.5,2) [xxx] {$\ourtrt$};
\node  (Trr) at (1.5,2) [xxx]{$\ourtrt$};
\draw [manyStrings] (A.south) -- +(0,-0.5);
\draw [manyStrings] (Tr.south) .. controls +(down:0.3) and +(up:0.3) .. 
	(node cs:name=A, angle=150);
\draw [manyStrings] (Trr.south) .. controls +(down:0.3) and +(up:0.3) .. 
	(node cs:name=A, angle=30);
\draw[rounded corners] (A.north) -- +(0,2.1) -- +(2.7,2.1)
-- +(2.7,-1.4) -- +(0.3,-1.4) -- (node cs:name=A, angle=-65);
\node [below] at (A.north) {$\circled{i}$};
\end{tikzpicture}
\end{array} 
 + \sum_{i \textrm{ even}}
\begin{array}{c}
\begin{tikzpicture}  [ xxx/.style={
	rectangle, draw,
		rounded corners,
		minimum size=12mm,
		line width = 1.5pt},
		manyStrings/.style={
			line width=5pt}]
\node (Tr) at (0,2) [xxx] {$\qquad \ourtrt \qquad$};
\node (A) at (-1,0) [xxx]  {$P$};
\node (W) at (1,0) [xxx] {$W$};
\draw [manyStrings] (A.north) .. controls +(up:0.4) and +(down:0.4) .. (node cs:name=Tr, angle=-120);
\draw [manyStrings] (A.south) -- +(0,-0.6);
\draw [manyStrings] (node cs:name=W,angle=120) .. controls +(up:0.4) and +(down:0.4) .. (node cs:name=Tr, angle=-60);
\draw [manyStrings, rounded corners] (node cs:name=W,angle=60) -- +(0,0.25) -- +(1.0,0.25) -- +(1.0,-1.50) -- (0,-0.85) -- (Tr.south);
\draw [rounded corners] (W.north) -- +(0,0.5) -- +(1.75,0.5) -- +(1.75,-1.73) -- (-0.63,-1.1) -- (node cs:name=A, angle=-60);
\node [below] at (W.north) {$\circled{i}$};
\end{tikzpicture}
\end{array}
\end{equation} 
To see that this equation holds, we note that it simply expresses the fact that the rightmost top string of $P$ must either come back to $P$, or be connected to a vertex coming from a  copy of the
potential $W$.  This shows part (c).

Finally, for part (d), let $\mathcal{P}\subset\mathcal{P}^\Gamma$ be a subfactor planar algebra, and assume that $W\in \mathcal{P}$ (as an example, one could consider $\TL\subset \mathcal{P}^\Gamma$ via the embedding $B\mapsto w_B$ given by equation \eqref{defwT}).  

We note that if $Q=\sum_w \alpha_w w \in \mathcal{P}$ and  $v\in \Gamma$, then
$\tau_t(Q)(v) = \mytau_t(\sum_w \alpha_w A_w )(v)= \ourtrt(Q)(v)$ by linearity.  

The infinite power series defining $\ourtrt(Q)$ converges, and its finite partial sums involve values diagrams connecting $Q$ to several copies of $W$ and having no unconnected strings.  Since $W\in \mathcal{P}$ and $\mathcal{P}$ is a planar subalgebra of $\mathcal{P}^\Gamma$, the values of these diagrams are the same whether they are evaluated in $\mathcal{P}$ (yielding an element of $(\mathcal{P})_0$, since no unconnected strings are present) or $\mathcal{P}^\Gamma$ (yielding an element of $(\mathcal{P}^\Gamma)_0$).  Since $\mathcal{P}$ is a subfactor planar algebra, $ \mathbb{C}\cong (\mathcal{P})_0$ is the subset of $(\mathcal{P}^\Gamma)_0$ consisting of elements which are constant on all vertices.  It follows that the value of $\ourtrt(Q)(v)$ does not depend on $v$.
 \qed

\subsubsection{Free energy}
We recap here how to get from the convergence of
the tracial state that of the free energy
$$F^{M,K}_t=\frac{1}{M^2}\log \frac{Z^{M,K}_t}{Z^{M}_0} =\frac{1}{M^2}\log \int 1_{\|A\|_\infty\le K}
e^{M\sum_{i=1}^k t_i \sum_{v\in V}\mu(v)\sum_{w\in L_{B_i}(v)} \sigma_B(w) \Tr(A_w)}
\mu^{M}(d A)$$
To that end we observe that, by differentiating with respect to $\alpha$,
\begin{eqnarray*}
F^{M,K}_t&=&  F^{M,K}_0+
\sum_{i=1}^k t_i \sum_{v\in V}\mu(v)^2\sum_{w\in L_{B_i}(v)}
 \int_0^1  \tr(A_w) d\mu^{M,K}_{\alpha t}(A) d\alpha\\
\end{eqnarray*}
Note that $F^{M,K}_0$ goes to zero as $K$ has been chosen large enough. 
Thus, by Proposition \ref{propfinite}, for $|t_i|$ smaller than $t(c,L)$, 
the last integral converges uniformly for all $\alpha \in [0,1]$ 
and therefore by bounded convergence theorem, 
we get
\begin{eqnarray*}
\lim_{M\ra\infty} F^{M,K}_t&=&
\sum_{i=1}^k t_i \sum_v \mu(v)^2\sum_{n_i\ge 0}
\prod_j \frac{1}{n_j!} \int_0^1 \prod(\alpha t_j)^{n_j+1_{j=i}}
\delta^{\sharp loops}d\alpha ,\\
&=&\left(\sum_v \mu(v)^2\right)\sum_{\sum n_i\ge 1}
\prod_j \frac{1}{n_j!}t_j^{n_j}
\delta^{\sharp loops}. \\
\end{eqnarray*}
 \subsubsection{The case of $\mathbb{A}_\infty$ and $\TL$ for $\delta \geq 2$}\label{Ainfty}

The previous construction only allows a countable set of values
of $\delta$'s (which however contains all the possible  $\delta<2$
and  accumulates at $2$). To get all possible 
values of $\delta$, we need to consider infinite graphs. 
However, our construction below relies heavily on the fact that
the entries of the eigenvector $\mu$ go to infinity  exponentially fast
 with their distance to a distinguished vertex of the graph. 
We will therefore restrict ourselves to the graph  $\mathbb{A}_\infty$
for which we shall prove that this property holds. Since not all subfactor
planar algebras can be embedded into the graph planar algebra of $\mathbb{A}_\infty$,
 we shall restrict ourselves to Temperley--Lieb algebra $\TL$
in this section.
In this case, it is
enough to consider the infinite graph $\mathbb{A}_\infty$ since this graph possesses an eigenvalue $\delta$ for any possible $\delta\geq 2$
with an eigenvector with
positive entries. We let $\mu$ be the corresponding 
eigenvector and denote by $n\in \mathbb N$ the
vertices of $\mathbb{A}_\infty$. Noting that by definition for $n\ge 3$
$$\mu(n-1)+\mu(n+1)=\delta \mu(n)$$
we see that $\mu(n)\approx \delta^n$ as $n$ goes to infinity.
We let, as in the finite graph case, $\mu^M$ be the  law
of the independent $M_{s(e)}\times M_{t(e)}$
matrices with covariance $1/M\sqrt{\mu(s(e))\mu(t(e))}$
for $e\in E_+$ (which is now infinite).
Edges far from the origin will have variance decreasing exponentially
fast with the distance to the origin (recall that $\delta>1$).
To construct the law on the infinite graph, we let $\Sigma_n$ be the 
sigma-algebra generated by the $A_e$ for $e\in E_n$, the set
of edges with  distance less than
$n$ from the origin. We let $V_n$ be the vertices connected to the origin
by a path in $E_n$, and  $L(n)$ (resp. $L_B(n)$) the set of loops  (resp.
of $L_B$)
 with all edges in $E_n$.
The idea is to consider the Gibbs
measure indexed by infinite graphs as the
limit of the conditional expectation with respect to $\Sigma_n$.
More precisely,
we  define
$$d\mu^{M,K,n}_{t}(dA)=\frac{\prod_{e\in V_n}
1_{\|A_e\|_\infty\le K}}{Z^{M,K,n}_V}
\exp\left\{ M\sum_{i=1}^k t_i\sum_{v\in V_n}\mu(v) \sum_{w\in L_{B_i}(n)}\sigma_{B_i}(w) \Tr(A_w) \right\}d\mu^M(dA).$$
We shall prove

\begin{proposition}\label{propinfinite}
Let $K>2$. Then, there exists $t(K)>0$ so that whenever
$\max_{1\le i\le k}
|t_i|\le t(K)$, for any vertex $v$
 for any loop $w\in L(v)$ there exists
a limit
$$\mytau_t(X_w)(v)=\lim_{n\ra\infty}
\lim_{M\ra\infty} 
\mu^{M,K,n}_t(\tr(A_{w}))\,.$$
Moreover, $\mytau_t(P^v_B)(v)=\ourtrt(B)$ for any Temperley--Lieb tangle 
$B$ and any vertex $v\in V_+$ (resp. $V_-$) if the marked point of $B$ is in an
unshaded  (resp. shaded) region . Finally,
\begin{eqnarray*}
\lim_{n\ra\infty}
\lim_{M\ra\infty}\frac{1}{M^2\sum_{v\in V_n} \mu(v)^2}\log Z^{M,K, n}_t
&=&\sum_{\sum n_i\ge 1}
\prod \frac{1}{n_j!}t_j^{n_j}
\delta^{\sharp loops} \\
\end{eqnarray*}
\end{proposition}
\begin{proof}
In fact, due again to strict convexity,
the negative of   the Hessian   of the log-density 
of $\mu^{M,K,n}_t$ is bounded below independently of $n$
on the set $\|A\|_\infty\le K$.  Moreover, because $(\mu(s(e))\mu(t(e))^{\frac 12}$ is at least of
order $\delta^n$ if $s(e)\in E_n^c$, we can choose $t$ small
enough so that this Hessian is bounded below by the 
diagonal matrix which takes the value $C\delta^{n}$ for $v\in
E_{n+1}\backslash E_n$ with some positive constant $C$.
This is enough to guarantee Brascamp-Lieb inequalities and
concentration of measure which do not depend on the dimension 
and in particular on $n$. 
In particular 
$$\mu^{M,K,n}_{t}(\max_{e\in E_{k+1}\backslash E_k}
\|A_e\|_\infty\ge (\ell-\ell_0)\delta^{-k/2} )\le e^{-\eta M\ell} $$
for some $\ell_0 <\infty$, $\ell>\ell_0$  and $\eta>0$.
Moreover, we can apply concentration inequalities
and use the fact that $A_w$ is a Lipschitz function
of the entries with Lipschitz norm of order $M^{-\frac 12}$ 
with overwhelming 
probability because of the above control (since $\delta>1$ we obtain absolutely
converging sequences) to see that
\eqref{conc} still holds under $\mu^{M,K,n}_{t}$, with a constant
$C(w,K)$ independent of $n$. Therefore, by the same arguments,
we obtain the convergence of $\mu^{M,K,n}_{t}(\tr(A_w))$, 
for any loop $w\in L(n)$
as $M$ go to infinity. The  rescaled (by $\mu$) limit $\tau^n$
 satisfies a Schwinger--Dyson equation
which has a unique solution, exactly as in Lemma \ref{uniq}, again
because
the controls are uniform in $n$, the range of $t$ for which we
have uniqueness does
not depend on $n$. 
Note that this equation is the same for $\tau^n$ and $\tau^m$
provided we consider only loops in $L(n\wedge m -D)$ where $D$ is the
maximal number of boundary points in the potential. Define
$$\Delta(k,p):=\max_{|w|\le k\atop w \in L(p)}|\tau^n(w)-\tau^m(w)|$$ 
We see that
$\Delta(k,p)$ satisfies \eqref{SDu} for $k +p \le n\wedge m$.
For larger $k$ we just bound the additional term uniformly by
$A_{n\wedge m}(t)K^{D-2+k}$ with
$$ A_{n}(t)=\max_{e\in E_n^c}\sqrt{\mu(s(e))\mu(t(e))}^{-1}\max_i\sum |t_i|k_i\,.$$
 We then find that if $\Delta(\gamma,p)=\sum_{k\ge 0}
\gamma^k \Delta(k,p)$, for say $p=n\wedge m/2$,
$$\Delta(\gamma,p)\le \frac{\gamma^{2}}{m(1-K\gamma)} \Delta(\gamma)+A(t)\gamma^{2-D}\Delta(\gamma)
+\sum_{k\ge m\wedge n/2} A_{n\wedge m/2}(t)K^{D-2+k} \gamma^{k-1}$$
This yields, if $(C')^{-1}=1-\frac{\gamma^{2}}{m(1-K\gamma)} -A(t)\gamma^{2-D}$ is positive (which is always possible if $t$ is small enough
for some $\gamma\in (0,K^{-1})$),
$$\Delta(\gamma, \frac{n\wedge m}{2})\le C'\sum_{p\ge m\wedge n} A_{n\wedge m/2}(t)K^{D-2+k} \gamma^{k-1}
\approx C'' (\gamma K/\delta)^{n\wedge m/2}$$
which goes to zero as $n\wedge m$ goes to infinity for $K\gamma\delta^{-1}<1$. This
shows that $\mytau^n$ converges as $n$ goes to infinity.
The limit satisfies Schwinger--Dyson equation
with appropriate $\delta$ and therefore corresponds, when 
restricted to $P$, with $\ourtrt$. The statement
on the free energy is derived as in the previous subsection.
\end{proof}

\section{Combinatorics of a few models}\label{few}
In this section, we actually compute 
the generating functions of a few loop models  we
have just constructed.
\subsection{The cup matrix model}\label{cupmodsec}
We consider the case of
the matrix model with tangles  $B_n={\rm cup}^n$ obtained by
drawing $n$ non nested strings in a white
tangle with black inside.
With $k$ finite and fixed, we thus consider the law $\ourtrt$
corresponding to the tangles $(B_1,\ldots,B_k)$ and
the coefficients $(t_1,\ldots, t_k)$.
By Propositions \ref{propfinite} and \ref{propinfinite},
$B_n$ is represented in $\mathcal{P}^\Gamma$ by $(\sum_{e\in E_{+}}\sigma(e) X_e X_e^*)^n$,
e.g.
\[
B_2=\vcenter{\hbox{
\begin{tikzpicture}
\vertexhw{a}(0,0)
\node at ([xshift=-0.1cm,yshift=0.2cm]a.north west) {$\scriptstyle X_{(w,v)}$};
\node at ([xshift=-0.1cm,yshift=-0.25cm]a.south west) {$\scriptstyle X_{(v,u)}$};
\node at ([xshift=0.1cm,yshift=0.2cm]a.north east) {$\scriptstyle X_{(v,w)}$};
\node at ([xshift=0.1cm,yshift=-0.25cm]a.south east) {$\scriptstyle X_{(u,v)}$};
\node at (0,-0.25) {$\color{white}\scriptstyle u$};
\node at (0.25,0) {$\scriptstyle v$};
\node at (0,0.25) {$\color{white}\scriptstyle w$};
\end{tikzpicture}}},
\qquad
B_3=
\vcenter{\hbox{
\begin{tikzpicture}
\begin{scope}
\clip[draw,rounded corners] (3,0.6) -- ++(0.6,0) -- ++(0.3,0.5) -- ++(-0.3,0.5) -- ++(-0.6,0) -- ++(-0.3,-0.5) -- cycle;
\coordinate (b) at (3.3,1.1);
\draw[fill=gray] (3,0.6) .. controls (b) .. ++(0.6,0) ++(0.3,0.5) .. controls (b) .. ++(-0.3,0.5) ++(-0.6,0) .. controls (b) .. ++(-0.3,-0.5);
\end{scope}
\node at (3,1.3) {$\color{white}\scriptstyle u$};
\node at (3.6,1.3) {$\color{white}\scriptstyle w$};
\node at (3.3,0.75) {$\color{white}\scriptstyle x$};
\node at (3.3,1.4) {$\scriptstyle v$};
\node at (3.8,1.8) {$\scriptstyle X_{(w,v)}$};
\node at (2.8,1.8) {$\scriptstyle X_{(v,u)}$};
\node at (2.3,1.1) {$\scriptstyle X_{(u,v)}$};
\node at (2.8,0.35) {$\scriptstyle X_{(v,x)}$};
\node at (3.8,0.35) {$\scriptstyle X_{(x,v)}$};
\node at (4.3,1.1) {$\scriptstyle X_{(v,w)}$};
\end{tikzpicture}}},
\qquad
\text{etc}
\]
and the associated Gibbs 
measure is, by Example \ref{extangle},
$$d\mu^{M,K}_t(A)=
\frac{{\bf 1}_{\|A\|\le K}}{Z^{M,K}_t}
e^{M\sum_{i=1}^n t_i \sum_{v\in V}
\mu(v)  \sum_{w\in L_{B_i}(v)}\sigma(w)
\Tr((\sum_{e\in E_{+}}\sigma(e) A_e A_e^*)^i)}
d\mu^M(A).$$
Note that  the families $(A_e, s(e)=v)$, $v\in V_+$, 
are independent. Since we evaluate the above expression
at words in $Z_v=(\sum_{e:s(e)=v} \sigma(e) A_e A_e^*)$ for 
a given $v$ we may as well consider the law 
of $Z_v$ for a fixed $v$.
Recalling that $A_e$ has variance $\sqrt{M_{s(e)}M_{t(e)}}^{-1}$,
we find that   $Z_v= \sum_{e: s(e)=v} \sigma(e) X_e X_e^*= Y_vY_v^*$
with a $M_{v}\times (\sum_{e:s(e)=v} M_{t(e)})$ matrix 
$Y_v$ with i.i.d.\ centered Gaussian entries 
with covariance $M_v^{-1}$.
The law of the eigenvalues  $(\lambda_1,\cdots,\lambda_{\mu(v) M})$
of such a matrix is asymptotically equivalent
(since we can again by Brascamp--Lieb inequality
remove the cutoff on $A$ and transform it as a cutoff on
$Z_v$ for some $K'$ large enough)   to
$$%\begin{multline*}
d\mu^{M,K'}_t(\lambda)=\frac{1}{Z^{M,K'}_t}
\prod_{i\neq j} |\lambda_i-\lambda_j|   \prod_i 
 \lambda_i^{\sum_{e:s(e)=v} M_{t(e)}}  e^{ M_v\sum_{i=1}^{M_v}
(\sum_{n=1}^{k}t_n 
\lambda_i^{n} -\lambda_i)}
\prod 1_{\lambda_i\in [0,K']}d\lambda_i.$$%\end{multline*}
Since $(\sum_{e:s(e)=v} M_{t(e)})/M_v$ 
converges as $M$ goes to infinity
to $\delta$, it is classical  \cite[Theorem 2.6.1]{AGZ} to
prove a large deviation principle
for the law of the empirical
measure of the $\lambda_i$ under $\mu^{M,K'}_t$
from which one easily derives the convergence of
the empirical measure. 
Therefore we deduce that
\begin{lemma}\label{cupmodlem} For all $K>2$, and $t=(t_1,\ldots,t_k)$ small
enough,  all  $n\ge 0$, all $v\in V_+$,
\begin{eqnarray}
\ourtrt(B_n)&=&\lim_{M\ra\infty} \mu^{M,K}_t(\tr(P^v_{B_k}))=\nu_t(x^n)\label{cocot}
\end{eqnarray}
with $\nu_t$ the only probability  measure on $[-K,K]$ which maximizes,
with $P(x)=\sum_{n=1}^k t_n x^n$,
$$I_t(\nu):=\Sigma(\nu)+\delta\int\log|x|d\nu(x)+ \int P(x) d\nu(x)-\int x d\nu(x)$$
where $\Sigma$ is the free  entropy
$$\Sigma(\mu)=\iint\log|x-y|d\nu(x)d\nu(y).$$
\end{lemma}
We can obtain an equation for
$\nu_t$ by writing that $I_t(\nu_t)\ge I_t(\nu_t^\e)$
with $\nu_t^\e$ the law of $x+\e h(x)$ under $\nu_t$
and $h$ a smooth real-valued  function. Expressing that
the linear term in $\e$ must vanish and ultimately taking $h(x)=(z-x)^{-1}$
we deduce that $\nu_t$ is solution of the Schwinger--Dyson type
equation
$$\left(\int \frac{1}{z-x}d\nu_t(x)\right)^2
+\delta\int \frac{1}{x(z-x)}d\nu_t(x)+\int \frac{P'(x)-1}{z-x} d\nu_t(x)=0$$
which yields with $G_t(z)=\int \frac{1}{z-x}d\nu_t(x), Q(z)
=\int \frac{P'(x)-P'(z)}{z-x} d\nu_t(x)$, and $c=\delta
\int x^{-1}d\nu_t(x)$,
$$G_t(z)^2+ \left(\frac{\delta}{z}+P'(z)-1\right)G_t(z) +Q(z)+\frac{c}{z}=0$$
so that
$$G_t(z)=\frac{1}{2z}\left( z(1-P'(z))-
\delta-\sqrt{ ( z(P'(z)-1)+
\delta)^2 -4z(z Q(z)-c)}\right).$$
It can be shown that
for $t$ small enough, $\nu_t$ has connected 
support and deduce a formula for $G_t$,
see e.g.~\cite{brezin-itzykson-parisi-zuber:planarDiagrams}. Indeed, 
according to \cite[Lemma 2.6.2]{AGZ}, for small $t_i$'s,
$\nu_t$ is characterized by the fact that a certain strictly concave 
function is smaller than some constant outside of
its support and equal to it at the boundary of the support;
this is only possible if the support is connected.
Since $G_t$ must
be analytic outside the support $[a,b]$
of $\nu_t$, the formula for $G_t$ can be uniquely determined
by the fact that there exists a polynomial $R$ with degree $k-1$,
and $a<b$
so that
$$( z(P'(z)-1)+
\delta)^2 -4z(z Q(z)-c)=(z-a)(b-z)R(z)^2$$

\subsection{The \texorpdfstring{$O(n,m)$}{O(n,m)} model}\label{onmsec}  Consider the potential
$$V=
\begin{array}{c}
\begin{tikzpicture}
\draw[rounded corners](0,0) rectangle (2,1.3);
\draw (0.25,1.3) .. controls +(-90:3mm) and +(180:3mm) .. (0.75,0.65) 
	.. controls +(0:3mm) and +(-90:3mm) .. (1.25,1.3);
\draw[densely dotted] (0.75,1.3) .. controls +(-90:3mm) and +(180:3mm) .. (1.25,0.65) 
	.. controls +(0:3mm) and +(-90:3mm) .. (1.75,1.3);
\end{tikzpicture} 
\end{array}
$$ considered in the planar algebra of two stitched Temperley--Lieb algebras (see \SectMark\ref{subsec:coupledTL}) 
and recall the notation of \SectMark\ref{CoupledTLInside}.

In the finite-depth case, the random matrix model is given by considering
block matrices indexed by pairs $(e,v)$ where $e$ is an edge in $\Gamma_{r}$
and $v$ is a vertex in $\Gamma_{b}$ or $e$ is an edge in $\Gamma_{b}$
and $v$ is a vertex in $\Gamma_{r}$ with the potential of the form\begin{eqnarray*}
\sum_{\substack{e\in\Gamma_{r}\\ v\textrm{ black vertex}}}\sigma(e)X_{(e,v)}X_{(e,v)}^{*}+\sum_{\substack{e\in\Gamma_{b}\\ v\textrm{ red vertex}}}\sigma(e)X_{(e,v)}X_{(e,v)}^{*}\\
+\beta\sum_{e\in\Gamma_{r}}\sum_{f\in\Gamma_{b}}\sigma(e)\sigma(f)X_{(e,s(f))}X_{(f,t(e))}X_{(e,t(f))}^{*}X_{(f,s(e))}^{*}.\end{eqnarray*}

If $\delta_{r}=n$ and $\delta_{b}=m$, we can choose $\Gamma_{r}$
(resp., $\Gamma_{b}$) to be the graph with one odd vertex and $n$
(resp., $m$) even vertices, with exactly one edge between the odd
vertex and every even vertex. the resulting combinatorics is exactly
the same as of the $O(n,m)$ model with the potential\[
\beta\sum_{i=1}^{n}\sum_{j=1}^{m}X_{i}Y_{j}X_{i}^{\dagger}Y_{j}^{\dagger}\]
which was studied for instance in \cite{dif}, in connection
with the problem of enumerating meanders. 
%
%
%  I did not touch anything below this line so far.
%
%
\subsection{The double cup matrix model}\label{doublecupsec}
The potential is a linear combination of two tangles; 
the tangle with two cups
with black inside (with coefficient $A$) and the tangle
with two cups with  white inside (with coefficient $B$ respectively). We denote by $V$
the element of the planar algebra 
 associated to it, namely (see Example \ref{extangle}):
$$V(X_e, e\in E)=t_2\left(\sum_{e\in E_-}\sigma(e) X_eX_e^*\right)^2
+ t_1 \left(\sum_{e\in E_+}\sigma(e) X_eX_e^*\right)^2.
$$
Diagrammatically, this corresponds to tangles of the form
\begin{center}
\begin{tikzpicture}
\vertexvw{a}(0.2,0);
\node at ([xshift=-0.1cm,yshift=0.2cm]a.north west) {$\scriptstyle X_{(v,u)}$};
\node at ([xshift=-0.7cm,yshift=-0.25cm]a.south west) {$\scriptstyle X_{(u,v)}=X^*_{(v,u)}$};
\node at ([xshift=0.7cm,yshift=0.2cm]a.north east) {$\scriptstyle X_{(w,v)}=X^*_{(v,w)}$};
\node at ([xshift=0.1cm,yshift=-0.25cm]a.south east) {$\scriptstyle X_{(v,w)}$};
\node at (-0.05,0) {$\color{white}\scriptstyle u$};
\node at (0.2,0.2) {$\scriptstyle v$};
\node at (0.47,0) {$\color{white}\scriptstyle w$};
\vertexvb{b}(4.2,0);
\node at ([xshift=-0.1cm,yshift=0.2cm]b.north west) {$\scriptstyle X_{(v,u)}$};
\node at ([xshift=-0.7cm,yshift=-0.25cm]b.south west) {$\scriptstyle X_{(u,v)}=X^*_{(v,u)}$};
\node at ([xshift=0.7cm,yshift=0.2cm]b.north east) {$\scriptstyle X_{(w,v)}=X^*_{(v,w)}$};
\node at ([xshift=0.1cm,yshift=-0.25cm]b.south east) {$\scriptstyle X_{(v,w)}$};
\node at (3.95,0) {$\scriptstyle u$};
\node at (4.2,0.2) {$\color{white}\scriptstyle v$};
\node at (4.47,0) {$\scriptstyle w$};
\end{tikzpicture}
\end{center}
where the only difference between the two pictures is that on the left,
$v\in V_+$ and $u,w\in V_-$, and vice versa on the right.

The associated Gibbs measure $\mu^{M,K}_t(dA)$ is given by
$$\frac{{\bf 1}_{\|A\|\le K} }{Z^{M,K}_t}
e^{ M^2 \sum_{v\in V} (t_21_{v\in V_-}+t_1 1_{v\in V_+})\mu(v)\tr
\left(\sum_{e:s(e)=v}\sigma(e) A_{e}A_{e}^*\right)^2
}\mu^M(dA)$$
To analyze the asymptotics of this measure,
we shall first perform a Hubbard--Stratonovitch transformation
%Gaussian disintegration
and then study the resulting Gibbs measure. 
We first relate this measure with our problem.
For the sake of simplicity, we restrict ourselves to 
$\delta$'s corresponding to finite graphs, and even further
to the graphs $\mathbb{A}_n$ with $n$ vertices and $n-1$ edges.
This is enough to characterize the generating functions
by analyticity (since the index $\delta$ corresponding to
these graphs has an accumulation point at $2$). In fact the restriction
to finite graphs allows us to avoid dealing with a Gibbs measure on infinitely
many matrices whereas the restriction to $\mathbb{A}_n$ ensures the uniqueness
of the minimizers of the entropy described in Proposition \ref{propldp},
and therefore their characterization.

\subsubsection{The partition function and an auxiliary matrix 
model}\label{pottsmatrixmodel}
Let us first consider the partition function
$$Z^{M,K}_t=
\int {\bf 1}_{\|X^M\|\le K} 
e^{ M^2 \sum_{v\in V} (t_21_{v\in V_-}+t_1 1_{v\in V_+})\mu(v)\tr
\left(\sum_{e:s(e)=v}\sigma(e) A_{e}A_{e}^*\right)^2
}\mu^M(dA)$$
and introduce  independent
 $M_v\times M_v$ matrices  $G_v$ from the GUE with covariance $1/M_v$.
Then, assuming $t A,t B$ positive 
and putting  $\alpha(v)=\sqrt{8t_2}$ if $v\in V_-$,
$\alpha(v)= \sqrt{8t_1}$ if $v\in V_+$, $\mu(v)'=\mu(v)\sqrt{M_v(\mu(v)M)^{-1}}$ (which approximately equals $\mu(v)$),
we get
$$Z^{M,K}_t=
\int {\bf 1}_{\|A\|\le K} 
e^{-\frac{M}{2} \sum_{v\in V} \alpha(v) \mu(v)' \Tr
\left(G_v \sum_{e:s(e)=v}\sigma(e) A_{e}A_{e}^*\right)
}\mu^M(dA, dG)
.$$
Note that again for $\alpha(v)$ small enough, the
integral has a strictly log-concave density and therefore Brascamp-Lieb inequalities show that the matrices $G_v$ are also
bounded with large probability and so there exists $K'$ large enough 
so that, we have
$$Z^{M,K}_t\sim
\int {\bf 1}_{\|G\|\le K'} 
e^{-\frac{M}{2} \sum_{v\in V}\alpha(v)\mu(v)' \Tr
\left(G_v \sum_{e:s(e)=v}\sigma(e) A_{e}A_{e}^*\right)
}\mu^M(dA, dG)
$$
where we used the standard notation $A_M\sim B_M$, for
two sequences $A_M, B_M$,  for a shorthand for
$A_MB_M^{-1}$ converges to one as $M$ goes to infinity.
Diagrammatically this corresponds to the following ``breaking'' of the tangles:
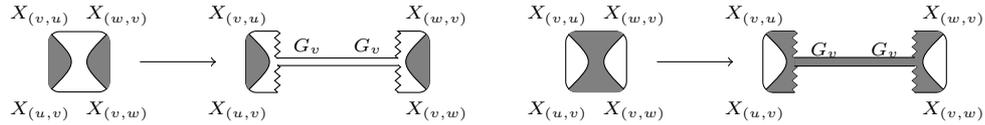
\begin{center}
\begin{tikzpicture}
\vertexvw{a}(0.2,0);
\node at ([xshift=-0.1cm,yshift=0.2cm]a.north west) {$\scriptstyle X_{(v,u)}$};
\node at ([xshift=-0.1cm,yshift=-0.25cm]a.south west) {$\scriptstyle X_{(u,v)}$};
\node at ([xshift=0.1cm,yshift=0.2cm]a.north east) {$\scriptstyle X_{(w,v)}$};
\node at ([xshift=0.1cm,yshift=-0.25cm]a.south east) {$\scriptstyle X_{(v,w)}$};
\draw[->] (1,0) -- (2,0);
\halfvertexwl{a1}(2.6,0)
\halfvertexwr{a2}(4.6,0)
\node at ([xshift=-0.1cm,yshift=0.2cm]a1.north west) {$\scriptstyle X_{(v,u)}$};
\node at ([xshift=-0.1cm,yshift=-0.25cm]a1.south west) {$\scriptstyle X_{(u,v)}$};
\node at ([xshift=0.1cm,yshift=0.2cm]a2.north east) {$\scriptstyle X_{(w,v)}$};
\node at ([xshift=0.1cm,yshift=-0.25cm]a2.south east) {$\scriptstyle X_{(v,w)}$};
\draw ([yshift=0.05cm]a1.east) -- ([yshift=0.05cm]a2.west) node[above=-2pt,pos=0.25] {$\scriptstyle G_v$} node[above=-2pt,pos=0.75] {$\scriptstyle G_v$};
\draw ([yshift=-0.05cm]a1.east) -- ([yshift=-0.05cm]a2.west);
\vertexvb{b}(7,0);
\node at ([xshift=-0.1cm,yshift=0.2cm]b.north west) {$\scriptstyle X_{(v,u)}$};
\node at ([xshift=-0.1cm,yshift=-0.25cm]b.south west) {$\scriptstyle X_{(u,v)}$};
\node at ([xshift=0.1cm,yshift=0.2cm]b.north east) {$\scriptstyle X_{(w,v)}$};
\node at ([xshift=0.1cm,yshift=-0.25cm]b.south east) {$\scriptstyle X_{(v,w)}$};
\draw[->] (7.8,0) -- (8.8,0);
\halfvertexbl{b1}(9.4,0)
\halfvertexbr{b2}(11.4,0)
\node at ([xshift=-0.1cm,yshift=0.2cm]b1.north west) {$\scriptstyle X_{(v,u)}$};
\node at ([xshift=-0.1cm,yshift=-0.25cm]b1.south west) {$\scriptstyle X_{(u,v)}$};
\node at ([xshift=0.1cm,yshift=0.2cm]b2.north east) {$\scriptstyle X_{(w,v)}$};
\node at ([xshift=0.1cm,yshift=-0.25cm]b2.south east) {$\scriptstyle X_{(v,w)}$};
\draw[braid,double=gray] (b1.east) -- (b2.west) node[above=-2pt,pos=0.25] {$\scriptstyle G_v$} node[above=-2pt,pos=0.75] {$\scriptstyle G_v$};
\end{tikzpicture}
\end{center}

We next integrate over the matrices 
$X^M$, recalling that the entries of $A_e$ have covariance $(M_{s(e)}M_{t(e)})^{-1/2}$. 
Up to a constant, this provides the term
$$\prod_{e\in E_+} e^{- \Tr\otimes
 \Tr(\log(I+\alpha(s(e))'I\otimes G_{s(e)}+
\alpha(t(e))' G_{t(e)}\otimes I))}$$
where we noticed that $\int e^{-\gamma x^2}dx=\sqrt{2\pi}
\gamma^{-\frac 1 2}$ and the matrices $X_e$ have complex 
entries (so that each term appears twice). Here $$\alpha(s(e))'\alpha(s(e))^{-1}=
\alpha(t(e))'\alpha(t(e))^{-1}=
 (\mu(t(e))'/ \mu(t(e)))^{1/2}$$ is approximately equal to one.
We can finally diagonalize the matrices $G_v$ to get
\begin{multline*}
Z^{M,K}_t\sim \int
 \prod_{v\in V} 1_{|\lambda^v|\le K'}
\Delta(\lambda^v) d\lambda^v  \exp\left[-\sum_{v\in V} \frac{M_v}{2} \sum_{i=1}^{M_v}( \lambda_i^v)^2\right]\\ \cdot
\prod_{e\in E_+} \exp\left[ -
\sum_{\substack{1\le i\le M_{s(e)}\\ 1\le j\le M_{t(e)}}} \log(1+\alpha(s(e)) \lambda_i^{s(e)}+
\alpha(t(e)) \lambda_j^{t(e)})\right] 
\end{multline*}
with $\lambda^v =(\lambda_1^v,\ldots,\lambda_{M_v}^v)$
the eigenvalues of $G_v$ and $\Delta(\lambda^v)=\prod_{1\le i\neq j\le M_v}
|\lambda_i^v-\lambda_j^v|$. The asymptotics of $\frac{1}{M^2}\log Z^{M,K}_t$
can be obtained via the 
 global asymptotics of the eigenvalues $(\lambda_v, v\in V)$,
that is the convergence
$$\lim_{M\ra \infty} E[\frac{1}{M_v}\sum_{i=1}^{M_v}
(\lambda_i^v)^p] =\int x^p d\nu_\nu(x)\qquad p\in\mathbb N,\quad v\in V_\pm.$$
under the associated Gibbs measure $P^{M,K}_t(d\lambda)$ given by
\begin{multline*}
\frac{ 1_{|\lambda^v|\le K'}}{Z^{M,K}_t}
 \prod_{v\in V}
\Delta(\lambda^v) d\lambda^v   \cdot \prod_{ e\in E_+} \exp\Bigg[ -
\sum_{i=1}^{M_{s(e)}}\sum_{j=1}^{M_{t(e)}}  \log\Big(1+\alpha(s(e)) \lambda_i^{s(e)}\\ +
\alpha(t(e)) \lambda_j^{t(e)}\Big)\Bigg] \cdot
\exp\left[-\sum_{v\in V}\frac{ M_v}{2} \sum_{i=1}^{M_v}( \lambda_i^v)^2\right].\end{multline*}

In the sequel, we shall prove not only
this convergence but the existence of two
probability measures $\nu_-$ and $\nu_+$
so that

\begin{equation}\label{convsm}
\int x^p d\nu_v(x)=\int x^p d\nu_\pm(x)\qquad p\in\mathbb N,\quad v\in V_\pm\,.
\end{equation}

Before attacking this question, let us summarize 
what information the auxiliary 
probability measure $P^{M,K}_t$ 
and \eqref{convsm}
tells us about our original
question. Recall that $\ourtrt$ is the
tracial states constructed 
with the two tangles with two strings and opposite
shading. We claim that $\nu_\pm$ gives the law
of a cup under $\ourtrt$ in the following sense.
\begin{proposition}\label{prop:conv}
Assume \eqref{convsm} and recall that $\alpha=\sqrt{8t_2}$.
Let $B_n$ be the tangle with $n$ non nested strings
and black shading inside and put for small $\gamma$
$$C(\gamma,A,B)=\sum_{n\ge 0}\gamma^n\ourtrt(B_n)$$
and $M(z)=\int \sum_{n\ge 0} z^n x^n d\nu_+(x)$.
Then,  $\gamma(z)=\frac{\sqrt{8 t_2} z}{1-z^2 M(z)}$
is invertible from a neighborhood of
the origin into a neighborhood of the
origin, with inverse $z(\gamma)$ and
$$C(\gamma,A,B)=\frac{ \alpha z(\gamma)}{\gamma} M( z(\gamma))=\frac{ \alpha }{\gamma z(\gamma)}\left( 1-\frac{\alpha z(\gamma)}{\gamma}\right)
.$$
\end{proposition}
\begin{proof}
Observe first that $P^{M,K}_t$ is the law of the eigenvalues
of $(G_v^M,v\in V)$ under
$$\nu^{M,K}_t(dA,dG)=\frac{{\bf 1}_{\|G\|\le K'}}{Z^{M,K'}_t} 
\exp\Bigg\{-\frac{M^2}{2}  \sum_{v\in V}\alpha(v)  \tr
\Big(G_v \sum_{s(e)=v}\sigma(e) A_{vt}A_{vt}^*\Big)
\Bigg\} \mu^M(dA,dG).$$
Adapting the previous considerations we see that 
$G_v \sum_{e:s(e)=v}\sigma(e) X_{e}X_{e}^*$ correspond
to an element of a planar
algebra  with one string and one strip, the strip being in the
white shading iff $v\in V_-$, both being independent in the sense that
they can not be glued together, and the strips
requiring to be coupled with another strip
corresponding to the same vertex, the weight being one. We can also
show the convergence (in a small parameters 
regime) of the law $\nu^{M,K}_t$ restricted to
the planar algebra generated by elements with non-crossing strings and
strips. We denote
the limit by $\mytau_t$. $\ourtrt$ corresponds to the
case where we restrict ourselves to
elements  with only strings (since
then expectation over the strips, that is the Gaussian variables 
can be taken) whereas $\nu_\pm$ corresponds to
restricting ourselves to element
with   strips only (inside a white or a black shading).
To relate both states let us consider the
expectation of an element $B_{n,p}$ with $n$ non nested cups
with black shading inside, followed by $p$ strips in the white region.
Let $C(p,n,\ell,k)$ be the number of possible configurations
build above this tangle with $\ell$ (resp. $k$) tangles 
with one string and one
strip in the white (resp. black) shaded region.
We get an induction relation by gluing the
first strip which yields for $ p\ge 1$,
\begin{eqnarray*}
C(p,n,\ell,k)&=&\sum_{p_1=0}^{p-2} \sum_{\ell_1\le \ell}\sum_{k_1\le
k}C^{k_1}_kC^{\ell_1}_\ell C(p_1,0,\ell_1,k_1)C(p-p_1-2,n,\ell-\ell_1,k-k_1)\\
&&\qquad\qquad +\ell C(p-1,n+1,\ell-1,k).\\
\end{eqnarray*}
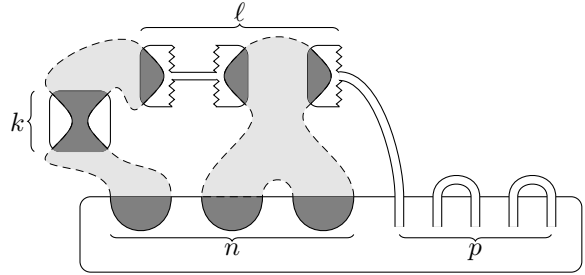
\begin{center}
%\resizebox{7cm}{!}{\includegraphics{bn}}
\begin{tikzpicture}
\draw[rounded corners] (0,-1) rectangle (6.6,0);
\draw[fill=gray] (0.4,0) .. controls (0.4,-0.6) and (1.2,-0.6) .. (1.2,0);
\draw[fill=gray] (1.6,0) .. controls (1.6,-0.6) and (2.4,-0.6) .. (2.4,0);
\draw[fill=gray] (2.8,0) .. controls (2.8,-0.6) and (3.6,-0.6) .. (3.6,0);
\draw[decorate,decoration=brace] (3.6,-0.5) -- (0.4,-0.5) node[pos=0.5,auto=left] {$n$};
\draw[decorate,decoration=brace] (6.2,-0.5) -- (4.2,-0.5) node[pos=0.5,auto=left] {$p$};
\draw[decorate,decoration=brace] (0.8,2.2) -- (3.4,2.2) node[pos=0.5,auto=left] {$\ell$};
\draw[decorate,decoration=brace] (-0.6,0.6) -- (-0.6,1.4) node[pos=0.5,auto=left] {$k$};
\halfvertexwl{a}(1,1.6)
\halfvertexwr{b}(2,1.6)
\halfvertexwl{c}(3.2,1.6)
\vertexvb{d}(0,1)
\draw[braid] (4.2,-0.4) -- (4.2,0) .. controls ++(0,0.5) and ++(0.5,0) .. (c.east);
\draw[braid] (4.7,-0.4) -- (4.7,0) .. controls ++(0,0.3) and ++(0,0.3) .. (5.2,0) -- (5.2,-0.4);
\draw[braid] (5.7,-0.4) -- (5.7,0) .. controls ++(0,0.3) and ++(0,0.3) .. (6.2,0) -- (6.2,-0.4);
%\draw[braid] (a.east) -- (b.west);
%manually correcting annoying bug
\draw ([yshift=0.05cm]a.east) -- ([yshift=0.05cm]b.west);
\draw ([yshift=-0.05cm]a.east) -- ([yshift=-0.05cm]b.west);
\draw[connect] (3.6,0) .. controls ++(0,0.5) and ++(-0.3,-0.3) .. (c.south west) .. controls (c.east) .. (c.north west) [bend right] to (b.north east) .. controls (b.west) .. (b.south east) .. controls ++(0.3,-0.3) and ++(0,0.5) .. (1.6,0) -- (2.4,0) .. controls ++(0,0.3) and ++(0,0.3) .. (2.8,0);
\draw[connect] (0.4,0) .. controls ++(0,0.5) and ++(-0.3,-0.3) .. (d.south west) .. controls (d) .. (d.north west) .. controls ++(-0.3,0.3) and ++(-0.3,0.3) .. (a.north west) .. controls (a.east) .. (a.south west) .. controls ++(-0.2,-0.2) and ++(0.2,0.2) .. (d.north east) .. controls (d) .. (d.south east).. controls ++(0.3,-0.3) and ++(0,0.5) .. (1.2,0);
\end{tikzpicture}
\end{center}
The first term appears when the strip is glued 
with another strip of the tangle whereas the second one shows up
when it is glued with a strip of an element with a strip
and a string with the right shading.

We let
$$C(z,\gamma,\alpha,\beta)=\sum \frac{z^p \gamma^n \alpha^\ell\beta^k}
{\ell! k!} C(p,n,\ell,k)$$
and conclude from the induction relation that
$$C(z,\gamma,\alpha,\beta)-C(0,\gamma,\alpha,\beta)
=z^2 C(z,0,\alpha,\beta)C(z,\gamma,\alpha,\beta)
+\frac{\alpha z}{\gamma}(C(z,\gamma,\alpha,\beta) -C(z,0,\alpha,\beta))$$
which gives:
$$C(z,\gamma,\alpha,\beta)=\frac{\gamma C(0,\gamma,\alpha,\beta)-\alpha 
zC(z,0,\alpha,\beta)}{\gamma-\alpha z-\gamma z^2 C(z,0,\alpha,\beta)}$$
Since $C(z,\gamma,\alpha,\beta)$ is analytic in $z,\gamma$ small enough,
we deduce that if the denominator vanishes so that 
$$\gamma=(1-z^2 C(z,0,\alpha,\beta))^{-1}\alpha z=:\gamma(z)$$
then the numerator must vanish too. Therefore,
 with $z(\gamma)$ the inverse of $\gamma(z)$ 
(which exists by the implicit function theorem in a neighborhood
of the origin) we deduce 
$$C(0,\gamma,\alpha,\beta)=\frac{\alpha z(\gamma)}{\gamma}C(z(\gamma), 0,\alpha,\beta).$$
Since if we choose $\alpha=\sqrt{8t_2},\beta=\sqrt{8t_1}$,
we have $C(0,\gamma,\alpha,\beta)=C(\gamma,A,B)$ and
$C(z,0,\alpha,\beta)=M(z)$, we have proved the proposition.
\end{proof} 
\subsection{Solving the auxiliary matrix model}\label{subsec:auxmm}
In this section we study the law $P^{M,K}_t$
and prove \eqref{convsm}.
\subsubsection{Large deviation estimates and limit points}

Using standard large deviation theory (see \cite[Section 2.6]{AGZ}), and putting
$$\Sigma(\mu,\nu):=\iint \log|x-y|d\mu(x)d\nu(y),\quad
\Sigma(\mu):=\Sigma(\mu,\mu)\,,$$
 we deduce that

\begin{proposition}\label{propldp} Set $\alpha(v)=\sqrt{8t_2}$ (resp. $=\sqrt{8t_1}$) if  $v\in V_-$ (resp. $v\in V_+$).

 and $\beta=\sqrt{8t_1}$.
The law of the spectral measures $$\left(\frac{1}{M_v}\sum_{i=1}^{M_v}\delta_{\lambda_i^v}, v\in V\right)$$ of the matrices $G_v, v\in V$ under $P^{M,K}_t$
 satisfies
a large deviation principle in the scale $M^2$ with good
rate function 
\begin{multline*}
I_t(\nu_v, v\in V):=\frac{1}{2} \sum_{v\in V} \mu(v)^2 \int x^2 d\nu_v(x)
-
 \sum_{v\in V}\mu(v)^2 \Sigma(\nu_v)
\\
 + \sum_{e\in E_+}
\mu(t(e))\mu(s(e)) \int \log(1+\alpha(s(e)) x +\alpha(t(e)) y)d\nu_{s(e)}
(x)d\nu_{t(e)}(y)
.\end{multline*}
Take $\Gamma=\mathbb{A}_n, n\ge 1$.
Then, $I_t$ achieves its minimal value at a unique
point so that $\nu_v=\nu_\pm$ if $v\in V_\pm$ with $(\nu_+,\nu_-)$
the unique minimizer of 
$$S(\nu_+,\nu_-)=
\sum_{\e=\pm}\left(\frac{1}{2}\int x^2d\nu_\e(x)-\Sigma(\nu_\e)\right)
+\delta \iint \log|1+\alpha x+\beta
y|d\nu_+(x)d\nu_-(y).$$ 
In particular, for any $p\ge0$, we have
$$\lim_{M\ra\infty}E\left[\frac{1}{M_v}\sum_{i=1}^{M_v}({\lambda_i^v})^p\right]
=\int x^p d\nu_\pm(x) \quad \textrm{ if } v\in V_\pm$$
\end{proposition}
\begin{remark} \label{rem1}By large deviation
techniques as in \cite{BAG97}, it is easy to
check that $(\nu_+,\nu_-)$
is also the limit of the spectral measures
of the two Hermitian $M\times M$ random matrices $G_+,G_-$
with joint law given by
$$ \frac{1_{\|G_\pm\|\le K'}}{Z^{M,\delta}}\exp\left[-\delta  \Tr\otimes\Tr
\log(I +\alpha I\otimes G_+ +\beta G_-\otimes I)\right] \exp\left[-\frac{M}{2}
\Tr(G_+^2+G_-^2)\right]dG_+ dG_-$$
when $K'$ is large enough.
This last formula is easily obtained for integer values of $\delta$
starting from the graphical rules of the double cup matrix model, 
and usually in the physics literature
such expressions are then analytically continued to $\delta$ non integer.

\end{remark}
{\bf Proof of Proposition \ref{propldp}.}
The first point is to show the uniqueness of the minimizers
of $I_t$. We  let  $\tilde\nu_v$ be the law of 
$-\alpha x$ (resp. $1+\beta y$)
under $\nu_v$ for $v\in V_+$ (resp. $v\in V_-$).
Then, we have to minimize
$$ I_t( \nu_v,v\in V):= H(\tilde \nu_v,v\in V)+
L(\tilde \nu_v,v\in V)$$
with, for probability measures $p_v, v\in V$ on the real line, 
 $$H(p_v, v\in V):=
 \sum_{e\in
 E_+}
\mu(t(e))\mu(s(e)) \Sigma( p_{s(e)}, p_{t(e)})
-\sum_{v\in V} \mu(v)^2\Sigma(
 p_v)
$$
and 
$$L(p_v,v\in V):= \frac{1}{2}\sum_{v\in E_-} ( \frac{\mu(v)}{\alpha})^2 \int x^2 dp_v(x)+\sum_{v\in E_+} 
 (\frac{\mu(v)}{\beta})^2 \int \left(1- x\right)^2  dp_v(x).$$
Since  $I_t$ is a good rate
function, it  has compact level sets (see the case $V=\{0\}$
in \cite{BAG97}) and therefore $I_t$ achieves its maximal value.
We next prove that its maximizer is unique.
Note that $L$ is linear in the measures. We shall
prove that $H$  is strictly convex. Indeed, put
$$d(v)= \sharp\{e\in E_+:s(e)=v\}+
 \sharp\{e\in E_+:t(e)=v\}$$
and observe that when $\Gamma\subset \mathbb{A}_\infty$
the degree $d(v)$ of each vertex is bounded by  one (for the boundary points)
or by two. Therefore, the quadratic form
$$Q(x)=\sum_{v\in V} x_v^2-\sum_{ e\in E_+}
 x_{s(e)}x_{t(e)}=\frac{1}{2}\sum_{ e\in E_+}
 (x_{s(e)}-x_{t(e)})^2
+\sum_{v\in V} x_v^2 (1-\frac{d(v)}{2})
$$
is positive definite. We let $(\gamma_i, v_i)$ be the eigenvalues
and eigenvectors of the corresponding matrix  in $\mathbb R^{+*}\times \mathbb R^{|V|}$ and write

$$H(p_v,v\in V)=-\sum_{i}\gamma_i \Sigma(\sum_{u\in V}
v_i(u)\mu(u) p_u)
.$$
Finally $ \Sigma $
is strictly log-concave on the set of real valued  measures with fixed
mass, as was proved  in \cite{BAG97}.  
Applying it to the
measure  $p_i=\sum_{u\in V} \mu(u) v_i(u)p_u$ with given mass $
\sum_{u\in V} v_i(u)\mu(u)$ and using that
the $\gamma_i$'s are positive,  we deduce 
that $H$ is strictly convex.
Therefore, $I_t$ (and by the
same argument  $S$) achieves its minimal value 
at a unique point $(\nu_v,v\in V)$ (resp. $(\nu_+,\nu_-)$).
We next show that it has to be $\nu_v=\nu_\pm$ if $v\in V_\pm$. 
In fact,  the infimum of $I_t$ is characterized by 
the fact that  for all $v\in V ^+$, the function $f_v(x)$ given by
$$2 \mu(v) \int \log|x-y| d\nu_v(y)- \sum_{e: s(e)=v}
\mu(t(e)) \int \log(1+\alpha x +\beta y)d\nu_{t(e)}(y)
- \frac{\mu(v) }{2} x^2 $$
is constant on the support of $\nu_v$ and non positive outside.
We have the same equation for $v\in E_-$ with $\alpha$ and
$\beta$ exchanged. The same characterization holds for 
$\nu_+$ which is  such that 
the function
$f_+$ given  by

$$-\delta
\int \log(1+\alpha x +\beta y)d\nu_{-}( y)
+2  \int \log|x-y| d\nu_{+}(y)
-  \frac{ x^2 }{2} $$
is constant on the support of $\nu_+$ and non positive outside
(and again a similar equation for $\nu_-$ with $\alpha$ and $\beta$ exchanged).
Putting  $\nu_v=\nu_\pm$ for $v\in V_\pm$, we find that
$f_v=\mu(v) f_{\pm}$
as  $\sum_{e:s(e)=v}
 \mu(t)=\delta\mu(v)$, and therefore
 is  indeed constant on the support of $\nu_v$ and non positive outside.
Thus $\nu_v=\nu_\pm$ for $v\in V_\pm$ is a solution, and by the previous argument
the unique solution.\qed

\subsubsection{Properties of the 
minimizers of the rate function}
We can  give  the following 
characterization of the  minimizer $(\nu_+,\nu_-)$ of $S$. 
\begin{lemma} \label{propmin} Let $\alpha=\sqrt{8t A}$ and $\beta=\sqrt{8t B}
$.  There exists $t_0>0$ so that
for $|t|\le t_0$,
\begin{itemize}
\item $\nu_\pm$ has a connected support $S_\pm$
included in $[-2-A(t),2+A(t)]$ with $A(t)$ going
to zero as $t$ goes to zero.
\item There exist functions $(f_\pm, g_\pm)$
which are analytic in a neighborhood
of $S_\pm$ and so that for $z\in \mathbb C$
\begin{equation}\label{devan}
G_\pm(z)=\int\frac{1}{z-x} d\nu_\pm(x)=f_\pm(z)-\sqrt{g_\pm(z)}
\end{equation}
with the branch of the square root chosen on $\mathbb R^-$.
Moreover $g_\pm$ is real on the real line and $S_\pm=\{ x: g_\pm(x)<0\}$.
We  define $G_\pm(z+i0)$ and $G_\pm(z-i0)$ as the limit of $G_\pm$ when $z$ goes to an element 
of $ S_\pm$ from above or from below.
\item For all $x\in S_\pm$ we have with $\alpha_+=\alpha,\alpha_-=\beta$,
\begin{equation}\label{SDl}
\delta\frac{\alpha_\pm}{\alpha_\mp} G_\mp\left(\frac{1-\alpha_\pm x}{\alpha_{\mp}}\right)
+x=G_\pm(x+i0) +G_\pm(x-i0).\end{equation}
\item There exists at most one  solution to \eqref{SDl} so that
$G_{\pm}$ are the Cauchy--Stieltjes transform of probability measures $\nu_\pm$
supported by $[-3,3]$.
\end{itemize}
\end{lemma}
\begin{proof}
The fact that the support of $\nu_-$ and $\nu_+$ is
connected is a consequence of Remark \ref{rem1} and \cite[Theorem 4.4 and Theorem 4.2]{guionnet-shlyakht:convexPotentials} which asserts that the
limiting measures $\nu_\pm$ have connected supports 
 when  the potential 
is strictly locally convex, which is the case when $t$ is small
enough (note here that
the potential is not a polynomial, however
it expands in absolutely converging power series when $t$ is small
enough and $x,y$ are bounded by $K'$
so that we can apply the techniques  from \cite{guionnet-shlyakht:convexPotentials}
to represent the measures as the law of an element of
the $C^*$ algebra generated by the free Brownian motion).
Another way to see this is to notice that the function $f_+$
is strictly concave outside of the support
and continuous on the support where it takes only one constant value;
hence the support can not be disconnected.  The fact that
the support is bounded   is  a  direct consequence
of Brascamp-Lieb inequalities.

To deduce the second point, we obtain an equation on $(\nu_+,\nu_-)$
by writing $$S(\nu_+^\zeta, \nu_-^\zeta)\ge S(\nu_+,\nu_-)$$
with  $\nu^\zeta_{\pm}$ the law of $x+\zeta h_\pm(x)$  for  bounded continuous functions $h_\pm$ on $\mathbb R$. 
Writing 
 that the linear term in $\zeta$ must vanish results 
with the equation
\begin{equation}\label{sadpoint}
\sum_{\e=\pm}\left(\int  x h_\e(x) d\nu_\e(x)- \iint\frac{h_\e(x)-h_\e(y)}{x-y}d\nu_\e(x)d\nu_\e(y)
\right)
\end{equation}
$$\qquad =-\delta \int \frac{\alpha h_+(x)+\beta h_-(y)}{1+\alpha x+\beta y}
d\nu_+(x)d\nu_-(y).$$
Taking $h_+(x)=-\frac{\beta}{\alpha} (z+\frac{1+\alpha x}{\beta}
)^{-1}$ and $h_-(x)=(z-x)^{-1}$
we get that, with $m(z)=G_+(-(1+\beta z)/\alpha)$,

$$\frac{\beta^2}{\alpha^2}m(z)^2+G_-(z)^2+\delta \frac{\beta}{\alpha}
m(z) G_-(z) -zG_-(z)+1+\frac{\beta^2}{\alpha^2}(1-(1+\frac{\beta z}{\alpha})m(z))=0$$
which gives
$$G_-(z)=\frac{1}{2}( b(z)-\sqrt{b(z)^2-4a(z)})$$
with the cut of the square root on $\mathbb R^-$ and 
$$b(z)=z-\delta \frac{\beta}{\alpha}
m(z),\qquad a(z)=\frac{\beta^2}{\alpha^2}m(z)^2+
1+\frac{\beta^2}{\alpha^2}(1-(1+\frac{\beta z}{\alpha})m(z))
.$$
But, for small $t$,
when $z$ is in the neighborhood of $S_-$,
$-(1+\beta z)/\alpha$ is in the neighborhood of $-1/\alpha$ 
which is far from the support $S_+$.
Hence, $a(z)$ and $b(z)$ are analytic in the
neighborhood of $S_-$ and also take
real values. This completes the proof of the
second point.

For the third point, it is enough to take $h_+=0$ or $h_-=0$
in \eqref{sadpoint}
with the remark that
$$\int \int\frac{h(x)-h(y)}{x-y} d\nu_\pm(x)d\nu_\pm(y)
=\int h(x) [ G_\pm(x+i0)+G_\pm(x-i0)] d\nu_\pm(x)$$
and use the continuity of $G_\pm$ above and below the
cut due to the previous point to obtain the desired equations
almost surely and then everywhere.

For the last point, since \eqref{SDl} is equivalent to
\eqref{sadpoint}, we show the uniqueness
of the solution by  taking $h_+(x)=(z-x)^{-1}$ to deduce
that
$$-1+zG_\pm(z)-G_{\pm}(z)^2=\delta\alpha_\pm \int \frac{1}{z-x}
\alpha_\mp G_\mp(-\frac{1+\alpha_\pm x}{\alpha_\mp} ) d\nu_\pm(x)=:\e_\pm(z)$$
so that for $z$ sufficiently large
$$G_{\pm}(z)=\frac{1}{2}(z-\sqrt{z^2 -4\e_\pm(z)})$$
where we have taken the usual determination of
the square root because $\e_\pm(z)$ is small, since $\alpha_\mp G_\mp(-\frac{1+\alpha_\pm x}{\alpha_\mp} )$ is uniformly close to one for $x\in S_\pm$ and $t$
small.  Therefore, if we have
two solutions $G$ and $\tilde G$,
 we  find that there exists $M(t)$ finite
so that for  $t$ small,
$$\sup_{\substack{z\in\mathbb R\\ |z|\ge M(t)}}
|G_{\pm}(z)-\tilde G_{\pm}(z)|\le \frac{1}{2}\sup_{\substack{z\in\mathbb R\\ |z|\ge M(t)}}|G_{\mp}(z)-\tilde G_{\mp}(z)|$$
which results with $G_\mp(z)=\tilde G_\mp(z)$ for $z$ large enough
and real, and then for all $z$ in the complement of $S_\pm$ by analyticity.
\end{proof}

\subsubsection{Characterization of the 
minimizers of the rate function}

In this section we completely characterize the
measures $(\nu_+,\nu_-)$ by their Cauchy--Stieltjes transform.
To simplify the notations, we let $\tilde\nu_+$ and
$\tilde\nu_-$ be the
law of $-\alpha  x$ and $1+\beta x$
under $\nu_+$ and $\nu_-$ respectively.
By 
 Lemma \ref{propmin}, for $t$ small enough,
$\tilde\nu_+$ and $\tilde\nu_-$ have
disjoint connected supports $[a_1,a_2]$ and $[b_1,b_2]$ 
around the origin and the unity
respectively. Our study will be restricted 
to this case, which therefore include 
small $t$'s but eventually a larger
class of parameters. We will first
proceed by a reparametrization
of the  Cauchy--Stieltjes transforms of $\tilde\nu_+$ and $\tilde\nu_-$ ,
which will allow to obtain very simple equations,
and then solve these equations.

$\bullet$ {\bf Parameterization}.
\renewcommand\Re{\mathop{\rm Re}\nolimits}
\renewcommand\Im{\mathop{\rm Im}\nolimits}
Consider
the standard parameterization of
the elliptic curve $y^2=(z-a_1)(z-a_2)(z-b_1)(z-b_2)$:
\begin{equation}\label{param}
u(z)=\frac{i}{2}\sqrt{(b_1-a_1)(b_2-a_2)} 
\int_{b_2}^z \frac{dz'}{\sqrt{(z'-a_1)(z'-a_2)(z'-b_1)(z'-b_2)}}
\end{equation}
where the path of integration avoids the segment
$[a_1,b_2]$ (for $z\in[a_2,b_1]$, we choose to come from the upper half plane, though
this choice is irrelevant, see the comment below on $2K$-periodicity), 
and the square root in the denominator is defined as having cuts
$[a_1,a_2]$ and $[b_1,b_2]$ and such that at infinity it behaves like $z^2$.

The image of $z\in\mathbb{C}\cup\{\infty\}-([a_1,a_2]\cup[b_1,b_2])\mapsto u(z)=(Re\ u(z),\Im u(z))$ 
is a rectangle of the form $[-K,K[\times ]0,iK'[$, where $K$ and $K'$ are usually called
quarter-periods (though they will be half-periods in what follows). Now,
note that crossing the %zigzag 
line $[a_2,b_1]$ corresponds to $u\to u+2K$. Since all the functions of $z$
we shall consider are smooth when one crosses this line, they can be made into periodic functions
on the strip $0<\Im u<K'$.
The map $z\mapsto u(z)$ is then an analytic isomorphism from
the Riemann sphere $\mathbb{C}\cup\{\infty\}$ minus the two segments $[a_1,a_2]$ and $[b_1,b_2]$
to $\mathbb{R}/2K\mathbb{Z}\times ]0,iK'[$. 
If one extends the map to the cuts $[a_1,a_2]$ and $[b_1,b_2]$, then the result
depends on whether one approaches them from the upper or lower half-plane, and they get sent to
$\Im u=K'$ et $\Im u=0$ as described on the figure.
More precisely, if $a\in]a_1,a_2[$, then $u(a+i0)+u(a-i0)=2iK'$ (the two images of $a$ are symmetric
w.r.t.\ $iK'$ on the line $\Im u=K'$), and similarly if $b\in ]b_1,b_2[$, then $u(b+i0)+u(b-i0)=0$.

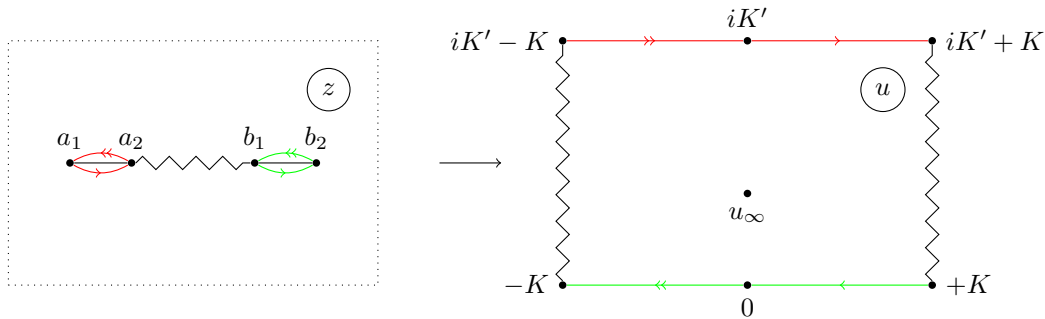
\begin{center}
\begin{tikzpicture}[scale=0.9,x=0.9cm,y=0.9cm]
\draw[dotted] (0,-2) rectangle (6,2);
%\node[rectangle,draw,dotted,anchor=north east] at (6,2) {$z$};
\node[circle,draw] at (5.2,1.2) {$z$};
\coordinate[circle,fill,inner sep=1pt,label=above:{$a_1$}] (a1) at (1,0);
\coordinate[circle,fill,inner sep=1pt,label=above:{$a_2$}] (a2) at (2,0);
\coordinate[circle,fill,inner sep=1pt,label=above:{$b_1$}] (b1) at (4,0);
\coordinate[circle,fill,inner sep=1pt,label=above:{$b_2$}] (b2) at (5,0);
\draw (a1) -- (a2);
\draw (b1) -- (b2);
\draw[decorate,decoration=zigzag] (a2) to (b1);
\draw[red,decoration={markings,mark = at position 0.5 with { \arrow{>} }},postaction={decorate},bend right] (a1) to (a2);
\draw[red,decoration={markings,mark = at position 0.5 with { \arrow{>>} }},postaction={decorate},bend right] (a2) to (a1);
\draw[green,decoration={markings,mark = at position 0.5 with { \arrow{>} }},postaction={decorate},bend right] (b1) to (b2);
\draw[green,decoration={markings,mark = at position 0.5 with { \arrow{>>} }},postaction={decorate},bend right] (b2) to (b1);
\draw[->] (7,0) -- (8,0);
%
%\node[rectangle,draw,dotted,anchor=north east] at (14,1) {$u$};
\node[circle,draw] at (14.2,1.2) {$u$};
\coordinate[circle,fill,inner sep=1pt,label=left:{$-K$}] (u1-b1) at (9,-2);
\coordinate[circle,fill,inner sep=1pt,label=right:{$+K$}] (u2-b1) at (15,-2);
\coordinate[circle,fill,inner sep=1pt,label=left:{$iK'-K$}] (u1-a2) at (9,2);
\coordinate[circle,fill,inner sep=1pt,label=right:{$iK'+K$}] (u2-a2) at (15,2);
\coordinate[circle,fill,inner sep=1pt,label=above:{$iK'$}] (u-a1) at (12,2);
\coordinate[circle,fill,inner sep=1pt,label=below:{$0$}]  (u-b2) at (12,-2);
\coordinate[circle,fill,inner sep=1pt,label=below:{$u_\infty$}]  (u-inf) at (12,-0.5);
%\draw  (u1-a1) rectangle (u2-a2);
\draw[decorate,decoration=zigzag] (u1-b1) to (u1-a2);
\draw[decorate,decoration=zigzag] (u2-b1) to (u2-a2);
\draw[red,decoration={markings,mark = at position 0.5 with { \arrow{>} }},postaction={decorate}] (u-a1) to (u2-a2);
\draw[red,decoration={markings,mark = at position 0.5 with { \arrow{>>} }},postaction={decorate}] (u1-a2) to (u-a1);
\draw[green,decoration={markings,mark = at position 0.5 with { \arrow{>>} }},postaction={decorate}] (u-b2) to (u1-b1);
\draw[green,decoration={markings,mark = at position 0.5 with { \arrow{>} }},postaction={decorate}] (u2-b1) to (u-b2);
\end{tikzpicture}
\end{center}

The inverse map $z(u)$ can be expressed
in terms of Jacobi's elliptic function $\sn$, and can be deduced from the following identity:
\[
\sn^2(u,k^2)=\frac{a_1-b_1}{b_2-b_1} \frac{b_2-z}{a_1-z}
\]
with $k^2=\frac{(b_2-b_1)(a_2-a_1)}{(b_2-a_2)(b_1-a_1)}$.
Note finally that
$$\lim_{z\ra\infty} u(z)=\frac{i}{2}
\int_{b_2}^{+\infty} \frac{\sqrt{(b_1-a_1)(b_2-a_2)}  dx}{\sqrt{(x-a_1)(x-a_2)(x-b_1)(x-b_2)}}=:
u_\infty\in i\mathbb R$$
is a pole of $z(u)$. It is a simple pole as one easily sees that for $z$
large $u(z)\sim u_\infty +\frac{z_{-1}}{z}+o(1/z)$ with
a non vanishing constant $z_{-1}$. Moreover,
 $z(u)$ is analytic everywhere else in $[-K;K]\times[0,iK']$
by the implicit function theorem, and,
once analytically continued to the whole complex plane, is even and
{\em elliptic}\/ (doubly periodic), with poles at 
$\pm u_\infty\pmod {2K, 2iK'}$.

$\bullet$ {\bf Resolvents}.
We let $\omega_\pm$ be the reparametrization of
the Cauchy--Stieltjes transform of $\tilde\nu_+$ and
$\tilde \nu_-$ respectively;
$$
\omega_\pm(u)=\int \frac{1}{z(u)-x} d\tilde\nu_\pm(x)
$$
The functions $\omega_\pm(u)$
are analytic in the strip $0<\Im u<iK'$, and according
to the second point \eqref{devan} of Lemma \ref{propmin}, 
they can in fact be analytically continued to some neighborhood of the closed strip $0\le\Im u\le iK'$.
Indeed, if $0<|\Re u(z)|<K$ the mapping $z\mapsto u(z)$
is invertible around $z$ and the extension
in the $z$ variable translates directly into the $u$ variable.
If $z=a_1,a_2$ (resp.\ $b_1,b_2$ for $G_-$), 
one has according to \eqref{devan} (with
a more detailed study of $g_\pm$ which shows that
for $t$ small enough $g_\pm'$ does not vanish in a neighborhood of $S_\pm$),
 $G_\pm(z)\propto  
(z-z')^{1/2}$ as $z'\to z$; but this matches the behavior of $u(z)$
in \eqref{param},
so once again $\omega_\pm(u)$ is a well-defined analytic function 
in the neighborhood of $0,iK',\pm K,iK'\pm K$.

By \eqref{SDl},  we have 
\begin{align}
\omega_+(u)+\omega_+(2iK'-u)-\delta\omega_-(u)&=P_+(u)&&\Im u=K' \label{firsteq}\\
\omega_-(u)+\omega_-(-u)- \delta\omega_+(u)&=P_-(u)&&\Im u=0
\end{align}
where $P_+(u)=z(u)/\alpha$ and  $P_-(u)=(z(u)-1)/\beta$. $P_\pm$ 
are even elliptic functions (with periods $2K$, $2iK'$).

We also have the following additional conditions:
\begin{align}
\omega_+(u)&=\omega_+(-u)&&\Im u=0\\
\omega_-(u)&=\omega_-(2iK'-u)&&\Im u=K'
\end{align}
expressing the fact that the Cauchy--Stieltjes transform of
$\tilde\nu_+$ is analytic in a neighborhood
of  $[b_1,b_2]$,
so its values at $z\pm i0$, $z\in [b_1,b_2]$ should be equal; and similarly for $\tilde\nu_-$. 

Now these equations can be repeatedly used to extend $\omega_\pm$
to the whole complex plane: for example $u\to 2iK'-u$ maps the strip $0\le \Im u\le K'$ to
the strip $K'\le \Im u\le 2K'$, so we can use Eq.~\eqref{firsteq} as a definition of $\omega_+$ in this new strip
and equation \eqref{firsteq} precisely ensure 
that the two definitions coincide at their common
boundary $\Im u=K'$; and so on. This way we obtain meromorphic functions $\omega_\pm(u)$
defined on the whole complex plane,
and by uniqueness of analytic functions that coincide on a set
with accumulation points, we deduce that the above equations are true
for all $u$:
\begin{eqnarray}
\omega_+(u+2K)&=&\omega_+(u)\label{firstpty}\\
\omega_-(u+2K)&=&\omega_-(u)\label{per2}\\
\omega_+(u)&=&\omega_+(-u)\label{pio1}\\
\omega_-(u)&=&\omega_-(2iK'-u)\\
\omega_+(u)+\omega_+(2iK'-u)-\delta\omega_-(u)&=&P_+(u)\label{cot1}\\
\omega_-(u)+\omega_-(-u)-\delta \omega_+(u)&=&P_-(u)
\label{lastpty}
\end{eqnarray}
Furthermore, these functions must possess a zero at 
$u_\infty:=u(z=\infty)$ and a prescribed
derivative at $u_\infty$.

$\bullet$ {\bf General solution of the saddle point equations}.
Now we have a well-posed analytic problem, which can be solved explicitly.
Set $\delta=q+q^{-1}$. If $|\delta|<2$, $q$ is not real and has modulus one,
e.g.~if $\delta=2\cos(\pi/n)$, $q=e^{i\pi /n}$.
If $\delta>2$, $q$ is real and can be chosen in $]0,1[$. 
Set
\[
\phi_\pm(u) = q^{\pm 1}\omega_+(u) - \omega_-(u)-R_\pm(u)
\]
where $R_\pm(u)=\frac{1}{1-q^{\pm 2}}(q^{\pm 1} P_+(u)+q^{\pm 2}P_-(u))$.
Then the equations can be recombined into:
\begin{align}
\phi_\pm(u+2K)&=\phi_\pm(u)\label{periodic}\\
\phi_\pm(u+2iK')&=q^{\pm 2} \phi_\pm(u)\label{twist}
\end{align}
where the first point is a direct consequence of \eqref{firstpty} and \eqref{per2},
whereas the second is obtained by multiplying \eqref{cot1}
by $q^\pm$ and \eqref{lastpty}  by $q^{\pm 2}$ and adding the two corresponding
equations. Moreover,  \eqref{pio1} and \eqref{lastpty} 
implies that 
 \begin{equation}\label{sym}
\phi_\pm(-u)=-\phi_\mp(u).\end{equation}
Thus we may consider $\phi_+$ only.
Furthermore we know that the only poles of $\phi_\pm$ in the fundamental domain
$[-K;K]\times [-iK';iK']$ are at $\pm u_\infty$ i.e.\ $z\to\infty$;
they appear because of the inhomogeneous
terms $R_\pm$, which have such poles. We can therefore express
$\phi_+(u)$ in terms
of $\theta$ functions.
Define $\Theta$ to be the rescaled $\theta_1$ function, or explicitly
\begin{equation}\label{theta}
\Theta(u)=2\sum_{n=0}^\infty e^{-(n+1/2)^2\pi K'/K} \sin (2n+1)\frac{\pi u}{2K}
\end{equation}
which satisfies
\begin{align*}
\Theta(u+2K)&=-\Theta(u)\\
\Theta(u+2iK')&=-e^{\pi (K'-i u)/K} \Theta(u)
\end{align*}
and with a unique simple zero at $u=0\pmod{2K,2iK'}$.

Then we have %by \cite[p.80]{cohn80}
\begin{proposition}\label{ratiotheta}
Write $q=e^{i\pi\nu}$ with $\nu$ real if $\delta<2$ and
purely
 imaginary if $\delta>2$. Then,
\begin{equation}\label{generalsol}
\phi_+(u)=c_+ \frac{\Theta(u-u_\infty-2\nu K)}{\Theta(u-u_\infty)}
+c_- \frac{\Theta(u+u_\infty-2\nu K)}{\Theta(u+u_\infty)}
\end{equation}
where
\[
c_\pm =\mp z_{-1}\frac{\Theta'(0)}{\Theta(2\nu K)}
\frac{1}{q-1/q}(\alpha^{-1}+q^{\pm 1}\beta^{-1})
\]
if $z(u)=z_{-1}/(u-u_\infty)+O(1)$ as $u$ goes to $u_\infty$. 
\end{proposition}cannot
\begin{proof} This can be viewed as a consequence of the Riemann--Roch
theorem, but we give here an elementary proof.

If $\delta>2$, let us first rule out the possibility
that $q^2=e^{-2\pi n K'/K}$ for some integer number $n$. Indeed then 
$e^{-ni\pi u/K} \phi_+(u)$ is also elliptic and therefore the sum
of its residues vanishes. But the latter is given by
$e^{-ni\pi u_\infty/K}(\alpha^{-1}+q\beta^{-1}) +e^{ni\pi u_\infty/K}(\alpha^{-1}+q^{-1}\beta^{-1})$ which cannot vanish. Thus, $q^2\neq e^{-2\pi n K'/K}$
and therefore $\Theta(2\nu K)\neq 0$ (which allows in particular to
define $c_\pm$ above). 

Next, let $\tilde \phi_+$ denote the right hand side of
\eqref{generalsol}  and observe that by the properties of the
functions $\Theta$, $\tilde\phi_+$ satisfies \eqref{twist} and
\eqref{periodic}. 

The formula for $c_\pm$ is obtained by requiring that
$\tilde \phi_+$ have the same
residues at $u\sim \pm u_\infty$ as our solution $\phi_+$. 
Indeed, $\Theta(u)\sim \Theta'(0) u$ as $u$ goes
to zero so that \eqref{generalsol} shows that
$$\tilde\phi_+(u)\sim_{u\ra u_\infty} c_+\frac{\Theta(2\nu K)}{ \Theta'(0) (u-u_\infty)} +O(1)$$
whereas both $\omega_+$ and $\omega_-$ go to zero and
$$R_+(u)\sim \frac{1}{1-q^2}(\alpha^{-1}q+\beta^{-1}q^2) z(u) \approx \frac{1}{z_{-1}(1-q^2)(u-u_\infty)}.$$
The formula $\phi_+=q\omega_+-\omega_--R_+$ allows to conclude. The
same reasoning works as $u\to -u_\infty$ by using $\phi_+(-u)=-\phi_-(u)$.

Let us  finally show
that  $f:=\phi_+-\tilde\phi_+$
must vanish. Indeed  $f$ is holomorphic
and therefore $g:=f'/f$ is holomorphic
except where $f$ vanishes, where
it has  only simple poles, with non-negative residues.
But since $f$ satisfies \eqref{twist} and \eqref{periodic}, $g$ is elliptic
and therefore the sum of its residues vanishes. Hence, the residues
of $g$ vanish, and therefore by Liouville's Theorem,  $g$ is constant, resulting with $f(u)=e^{\gamma u}$ for some constant $\gamma$. But then \eqref{twist} implies that
$\gamma=i\pi n/K$ and $q^2=e^{-2\pi n K'/K}$, which we excluded earlier.

\end{proof}

Finally, we  need to fix the parameters $a_1$, $a_2$, $b_1$, $b_2$.

The first way is to notice that $G_\pm(z)$ is an analytic function
in $\alpha,\beta,z$ in a neighborhood of the origin as, by Remark \ref{rem1},
it is the Stieltjes function of the limiting spectral measure of a matrix
model with strictly log-concave density which, even though
not polynomial, expands as a power series, see \cite{guionnet-shlyakht:convexPotentials}. The coefficients of these series can be computed recursively 
by the Schwinger--Dyson equation. Finally, by \eqref{devan}, the boundary of
the support $[a_1,a_2]$ are determined
by $g_+(a_i)=0, i=1,2$  which shows that there is a
polynomial $P$ so that $P(a_i,G_+(a_i),\alpha,\beta)=0$.
The implicit function theorem then implies that $a_i$ is
an analytic function of $\alpha,\beta$ for $i=1,2$ 
whose expansion can be deduced from the expansion
of $G_+$. The same applies for $b_1,b_2$. 

The second way to determine these boundary points uses
the explicit formula in terms of $\theta$ functions
and the reparametrization of the problem in terms 
of  $p:=\exp(-\pi K'/K)$ and of 
\[
\kappa:=p e^{-2 i \pi u_\infty/K}
\]
Note that because $(a_1,a_2,b_1,b_2)$ expand analytically in $\alpha,\beta$,
so do $(K,K',u_\infty)$ with
\begin{gather*}u_\infty=\sum_{n+m\ge -1}u_{n,m} \beta^n\alpha^m,\\
K=u(b_1)=\sum_{n+m\ge -1}K_{n,m} \beta^n\alpha^m,\textrm{ and }K'=-iu(a_1)=
\sum_{n+m\ge 0}K'_{n,m} \beta^n\alpha^m.\end{gather*}
As a consequence, $(p,\kappa)$ also expand in terms 
of $(\alpha,\beta)$, with $\kappa\sim \sqrt{\alpha/\beta}$ and $p\sim \sqrt{\alpha\beta}$ when $\alpha,\beta$ are small but $\alpha/\beta$ of
order one.  Again by the implicit function theorem, we
can invert this expansion and obtain $\alpha,\beta$ as 
a power series in $(p,\kappa)$, and therefore also $(a_1,a_2,b_1,b_2)$,
 $z_{-1}$ and $z(u)$.  We can then identify the expansion of 
$(q-q^{-1})\omega_+$ and $\phi_++R_+-\phi_-R_-$ around $u_\infty$
to compute $\omega_+$ recursively. In appendix A, we provide the first few orders
of the power series expansion of some quantities as a function of
$p$ and $\kappa$, and their diagrammatic meaning.

\begin{proposition}\label{prop:alg}
If $\delta=2\cos(\pi/n)$, $n\ge 3$, $G_\pm$ satisfy an algebraic equation.
\end{proposition}
\begin{proof} Observe that since $q= e^{i\pi/n}$,
by equations (\ref{periodic},\ref{twist}), the functions 
$\phi_\pm$ are elliptic with periods $(2K, 2n iK')$, and therefore so are
$\omega_\pm$.
The function $z(u)$ is also elliptic with these periods.
But a fundamental theorem of elliptic functions \cite[section 20.54]{WW96} 
states that two elliptic functions with the same periods are
related by an algebraic equation:  there exist two
polynomials $P_\pm$ so that
$$P_\pm(\omega_\pm(u), z(u))=0\,.$$
Composing with $u(z)$ shows the existence of
an algebraic relation. 

We can in fact determine the degree of $P_\pm$ by a slightly more explicit
construction of these polynomials;
we find it is at most $2n-2$ in $z$ and $2n$ in $\omega$.
\end{proof}

This result is to be compared with Theorem 15 of \cite{BBM09}. Indeed, their
generating series $M(q,\nu,t,w,z;x,y)$ is closely related to our 
$G_\pm(z;\alpha,\beta;\delta)$. The correspondence of
parameters goes as follows:
$q=\delta^2$, $\nu=1+\delta\, \alpha/\beta$; among the three
parameters $t,w,z$, one is redundant and if $z$ is set to $1$
one has $t=\beta$, but $w$ is fixed to be $w=1/\delta$ ($w$ and $z$ are
parameters weighing in our language white and black regions between tangles;
note that introducing an extra parameter in our matrix model to let
$w$ vary is possible and would make no difference in the exact solution
presented so far, so that Prop.~\ref{prop:alg} would still hold).
$x,y$ are ``boundary'' parameters similar to our parameter $z$, in the sense
that they give a weight to a particular edge (or vertex, or face) of the
planar map. However they are not exactly the same, and therefore
direct identification of our generating series is not possible; only
specific identities can be written, such as
\[
\frac{1}{\alpha}\int x\, d\tilde \nu_+(x)=M(\delta^2,1+\delta\,\alpha/\beta,
\beta,1/\delta,1;1,1)
\]
where the left hand side is the $1/z^2$ term in the $z\to\infty$ expansion
of $G_+(z)$, and the right hand side is the generating series of
\cite{BBM09} with certain specializations of its parameters.

\appendix
\section{Analyticity of \texorpdfstring{$\mathscr{T}$}{tr} as a function of fugacity}\label{app:analytic}
 Let $S_1,\ldots,S_k$ and $S$ be 
elements of the Temperley--Lieb planar algebra  with fugacity $\delta$ and
set   for  complex parameters $t_1,\ldots,t_k$,
\begin{equation}\label{trbeta2}
\mathscr{T}_{t,\delta} (S)=\sum_{n_1,\ldots,n_k=0}^\infty
\prod_{i=1}^{n_k}\frac{t_i^{n_i}}{n_i!} \sum_{P\in
P(n_1,\ldots,n_k,S)}
\delta^{\#\textrm{ loops in } P}
\end{equation}
where we sum over all admissible planar maps 
built on $S_1,\ldots,S_k,S$. Then we state that if $b_i$ is the number of boundary points of $S_i$, then
\begin{lemma}\label{lemana} There exists a positive constant $B$ so that
for all $t_1,\ldots,t_k,\delta\in\mathbb C^{k+1}$ so
that $\max_{1\le i\le k} |\delta|^{\frac{b_i}{2}} | t_i|<B$, $\mathscr{T}_{t,\delta} (S)$
is a well defined absolutely converging series,
and therefore
$t,\delta\ra \mathscr{T}_{t,\delta} (S)$ is analytic on this set.
\end{lemma}
Indeed, the number of loops
is bounded by one half of the total number of boundary points, i.e., $\frac{1}{2}(b + \sum n_i b_i)$, whereas before $b_i$ is the number of boundary points of $S_i$ and $b$ is the number of boundary points of $S$.
Therefore, the coefficients of the series are simply bounded 
by $$C(n_1,\ldots,n_k):=|\delta|^{ \frac{1}{2} b}
\prod_{i=1}^{k}\frac{(|t_i||\delta|^{\frac{b_i}{2}}) ^{n_i }}{n_i!}$$
and the sum can be enlarged to all planar maps
that can be built over $n_i$ (resp. one)
 vertices with degree given by the number
of boundary points of $S_i$, $1\le i\le k$ (resp. $S$).
It is well known, see e.g. \cite[p. 255]{guionnet-edouard:combRM},
that the number of such maps grows as 
$\prod_{i=1}^k n_i! A^{ n_1+\cdots+n_k}$ for 
some finite constant $A$.  Hence, $ \mathscr{T}_{t,\delta} (S)$
is an absolutely converging series on $\max_{1\le i\le
k}|t_i||\delta|^{\frac{b_i}{2}}A<1$, domain on which it is analytic.

Note that it is expected that as one increases the 
$t_i$, one should eventually
reach a hypersurface of singularities which signals the boundary 
of the analyticity region in the variables $t=(t_1,\ldots,t_k)$.
This singularity is usually present in matrix models and is explained by
the proliferation of planar maps: typically the number of planar maps grows
exponentially with its number of vertices and this produces a finite radius
of convergence of the corresponding generating series.
This is certainly what happens in the cases studied in section \ref{few}.
The model of \ref{cupmodsec} is closely related to the so-called 
one-matrix-model, whose possible ``critical behaviors'' (i.e., types of
singularities) are well-known \cite{DFGZJ}.
No exact solution is known for the model of \ref{onmsec}, but 
a conjecture on its critical exponent is proposed in \cite{dif}.
Finally, it is expected that the model of \ref{doublecupsec} has a critical
behavior of the type of the $O(n)$ model on the line $A=B$, and that of pure
gravity on other lines of constant ratio $A/B\ne 1$.
It would be interesting to find an interpretation of these critical
behaviors in the present context, i.e.\ in terms of properties of $\ourtrt$.

\newpage
\section{First few diagrams of the double cup matrix model}
The parameters $a_1,a_2,b_1,b_2$ of the auxiliary model of 
sect.~\ref{subsec:auxmm} can, as already mentioned,
be fixed by appropriate expansion of $\omega_+(u)$ around $\pm u_\infty$.
In practice, the resulting equations reduce to first degree equations on the 
condition that one solves them parametrically in terms of
the elliptic nome $p:=\exp(-\pi K'/K)$ and of $u_\infty$.
All other quantities can then be obtained in the same parameterization.

As a check, we shall here write the first few orders of the expansion
for small $\alpha$, $\beta$.
The correct scaling as $\alpha,\beta\to 0$
at fixed ratio is to keep the quantity
\[
\kappa:=p e^{-2 i \pi u_\infty/K}
\]
fixed while sending $p$ to zero, so that $\alpha/\beta\sim\kappa^2$ and $\alpha\beta\sim p^2$. 
We find the following expansions:
\begin{align*}
\alpha&=\kappa\Big[
p- (\kappa(3 \delta +2)+ \kappa^{-1}( 2 \delta 
 +6))p^2\\
&+
\left((8 \delta^2 +5\delta+3) \kappa^2
   +(8 \delta^2 +45\delta+24) +(5 \delta^2 +12\delta+17)\kappa^{-2}\right)p^3
+O(p^4)
\Big]
\\
\beta&=\kappa^{-1}\Big[
p- (\kappa^{-1}(3 \delta +2)+ \kappa( 2 \delta 
 +6))p^2\\
&+
\left((8 \delta^2 +5\delta+3) \kappa^{-2}
   +(8 \delta^2 +45\delta+24) +(5 \delta^2 +12\delta+17)\kappa^2\right)p^3
+O(p^4)
\Big]
\\
a_1&=\kappa\left[-2p^{1/2}+\delta p+2((3+\delta)\kappa^{-2}+(1+\delta)\kappa^2)p^{3/2}+O(p^2)\right]
\\
a_2&=\kappa\left[2p^{1/2}+\delta p-2((3+\delta)\kappa^{-2}+(1+\delta)\kappa^2)p^{3/2}+O(p^2)\right]
\\
b_1&=1-\kappa^{-1}\left[2p^{1/2}+\delta p-2((1+\delta)\kappa^{-2}+(3+\delta)\kappa^2)p^{3/2}+O(p^2)\right]
\\
b_2&=1-\kappa^{-1}\left[-2p^{1/2}+\delta p+2((1+\delta)\kappa^{-2}+(3+\delta)\kappa^2)p^{3/2}+O(p^2)\right]
\\
\int x\,& d\tilde\nu_+(x)
=\kappa\Big[
\delta p- \delta  (\kappa(2 \delta+1) +\kappa^{-1}(\delta +5))p^2\\
&+
\delta  \left(
\kappa^2(4\delta^2+1)
+3\delta^2+24\delta+1
+\kappa^{-2}(2\delta^2+4\delta+11)
\right)
p^3
+O(p^4)\Big]
\\
\int x^2\,& d\tilde\nu_+(x)
=\kappa\Big[
p+(\kappa(\delta^2-2\delta-2)-2\kappa^{-1}(3+\delta))p^2\\
&+
(
\kappa^2(-4 \delta ^3+3 \delta ^2+3 \delta+3)
+(-2\delta ^3-4 \delta ^2+36 \delta +24)\\
&+\kappa^{-2}(5 \delta ^2+12 \delta +17)
)p^3
+O(p^4)\Big]
\end{align*}
(the expansions of $a_1,a_2,b_1,b_2$ are only given up to order $p^{3/2}$
because of issues of space).

Inverting the first two expansions and inserting the result in the last two yields
\begin{multline*}
\frac{1}{\alpha}\int x\, d\tilde\nu_+(x)
=\delta+\delta(1+\delta)(\alpha+\beta) \\
+\delta((2+5\delta+2\delta^2)\alpha^2+(6+8\delta+4\delta^2)\alpha\beta+(2+5\delta+2\delta^2)\beta^2)+\cdots
\end{multline*}

Similarly,
$$
\frac{1}{\alpha^2}\Big(\int x^2\,d\tilde\nu_+(x)-\alpha\Big)
=\delta(1+\delta)+\delta(2+5\delta+2\delta^2)\alpha +\delta(3+4\delta+2\delta^2)\beta+\cdots
$$
According to Prop.~\ref{prop:conv},
$\int x\, d\tilde\nu_+(x)/\alpha$ corresponds diagrammatically to:
\begin{align*}
&\frac{1}{\alpha}\int x\,d\tilde\nu_+(x)=\delta
\begin{tikzpicture}[baseline=0]
\cupb{cup}(0,0)
\fixbb
\path[connect] (cup.north west) .. controls ++(0,0.5) and ++(0,0.5) .. (cup.north east);
\end{tikzpicture}
+
\alpha\delta
\begin{tikzpicture}[baseline=0]
\cupb{cup}(0,0)
\vertexvw{vert}(0,1.2)
\fixbb
\path[connect] (cup.north west) .. controls ++(-0.2,0.1) and ++(-1.1,-0.3) .. (vert.north west) .. controls (vert) .. (vert.south west) .. controls ++(-0.3,-0.3) and ++(0.3,-0.3) .. (vert.south east) .. controls (vert) .. (vert.north east) .. controls ++(1.1,-0.3) and ++(0.2,0.1) .. (cup.north east);
\end{tikzpicture}
+
\beta\delta
\begin{tikzpicture}[baseline=0]
\cupb{cup}(0,0)
\vertexvb{vert}(0,1.2)
\fixbb
\path[connect,bend left] (cup.north west) to (vert.south west) .. controls (vert) .. (vert.north west) .. controls ++(-0.3,0.3) and ++(0.3,0.3) .. (vert.north east) .. controls (vert) .. (vert.south east) to (cup.north east);
\end{tikzpicture}
\\ & +
\alpha\delta^2
\begin{tikzpicture}[baseline=0]
\cupb{cup}(0,0)
\vertexhw{vert}(0,1.2)
\fixbb
\path[connect,bend left] (cup.north west) to (vert.south west) .. controls (vert) .. (vert.south east) to (cup.north east);
\path[connect] (vert.north west) .. controls ++(-0.3,0.3) and ++(0.3,0.3) .. (vert.north east) .. controls (vert) .. (vert.north west);
\end{tikzpicture}
+
\beta\delta^2
\begin{tikzpicture}[baseline=0]
\cupb{cup}(0,0)
\vertexhb{vert}(0,1.2)
\fixbb
\path[connect] (cup.north west) .. controls ++(-0.2,0.1) and ++(-1.1,-0.3) .. (vert.north west) .. controls (vert) .. (vert.north east) .. controls ++(1.1,-0.3) and ++(0.2,0.1) .. (cup.north east) (vert.south west) .. controls ++(-0.3,-0.3) and ++(0.3,-0.3) .. (vert.south east) .. controls (vert) .. (vert.south west);
\end{tikzpicture}
\\
&+\alpha^2\delta^2\left(
\begin{tikzpicture}[baseline=1.1cm]
\cupb{cup}(0,0)
\vertexhw{a}(-0.6,1.2)
\vertexhw{b}(0.6,1.2)
\fixbb
\path[connect] (cup.north west) .. controls ++(0,0.3) and ++(-0.3,-0.3) .. (a.south west) .. controls (a) .. (a.south east) to [bend right] (b.south west) .. controls (b) .. (b.south east) .. controls ++(0.3,-0.3) and ++(0,0.3) .. (cup.north east);
\path[connect] (a.north west) .. controls ++(-0.3,0.5) and ++(0.3,0.5) .. (b.north east) .. controls (b) .. (b.north west) to[bend right] (a.north east) .. controls (a) .. (a.north west);
\end{tikzpicture}
+
\begin{tikzpicture}[baseline=1.1cm]
\cupb{cup}(0,0)
\vertexhw{a}(-0.6,1.4)
\vertexvw{b}(0.6,1.4)
\fixbb
\path[connect] (cup.north west) .. controls ++(0,0.3) and ++(-0.3,-0.3) .. (a.south west) .. controls (a) .. (a.south east) .. controls ++(0.3,-0.3) and ++(-0.3,0.3) .. (b.north west) .. controls (b) .. (b.south west) .. controls ++(-0.3,-0.3) and ++(0.3,-0.3) .. (b.south east) .. controls (b) .. (b.north east) .. controls ++(0.4,0.2) and ++(1.5,0) .. (cup.north east);
\path[connect] (a.north west) .. controls ++(-0.3,0.3) and ++(0.3,0.3) .. (a.north east) .. controls (a) .. (a.north west);
\end{tikzpicture}
+
\begin{tikzpicture}[baseline=1.1cm]
\cupb{cup}(0,0)
\vertexvw{a}(-0.6,1.4)
\vertexhw{b}(0.6,1.4)
\fixbb
\path[connect] (cup.north east) .. controls ++(0,0.3) and ++(0.3,-0.3) .. (b.south east) .. controls (b) .. (b.south west) .. controls ++(-0.3,-0.3) and ++(0.3,0.3) .. (a.north east) .. controls (a) .. (a.south east) .. controls ++(0.3,-0.3) and ++(-0.3,-0.3) .. (a.south west) .. controls (a) .. (a.north west) .. controls ++(-0.4,0.2) and ++(-1.5,0) .. (cup.north west);
\path[connect] (b.north east) .. controls ++(0.3,0.3) and ++(-0.3,0.3) .. (b.north west) .. controls (b) .. (b.north east);
\end{tikzpicture}\right. \\
&\left. +
\begin{tikzpicture}[baseline=1.1cm]
\cupb{cup}(0,0)
\vertexvw{a}(-0.6,1.4)
\vertexvw{b}(0.6,1.4)
\fixbb
\path[connect] (cup.north west) .. controls ++(-1.5,0.1) and ++(-0.4,0.2) .. (a.north west) .. controls (a) .. (a.south west) .. controls ++(-0.3,-0.3) and ++(0.3,-0.3) .. (b.south east) .. controls (b) .. (b.north east) .. controls ++(0.4,0.2) and ++(1.5,0.1) .. (cup.north east);
\path[connect,bend left] (b.south west) to (a.south east) .. controls (a) .. (a.north east) to (b.north west) .. controls (b) .. (b.south west);
\end{tikzpicture}
+
\begin{tikzpicture}[baseline=1.1cm]
\cupb{cup}(0,0)
\vertexhw{a}(0,1.2)
\vertexvw{b}(0,2.4)
\fixbb
\path[connect,bend left] (cup.north west) to (a.south west) .. controls (a) .. (a.south east) to (cup.north east);
\path[connect] (a.north west) .. controls ++(-0.2,0.1) and ++(-1.1,-0.3) .. (b.north west) .. controls (b) .. (b.south west) .. controls ++(-0.3,-0.3) and ++(0.3,-0.3) .. (b.south east) .. controls (b) .. (b.north east) .. controls ++(1.1,-0.3) and ++(0.2,0.1) .. (a.north east) .. controls (a) .. (a.north west);
\end{tikzpicture}
\right)
+\alpha^2\delta\left(
\begin{tikzpicture}[baseline=1.1cm]
\cupb{cup}(0,0)
\vertexvw{a}(-0.6,1.4)
\vertexvw{b}(0.6,1.4)
\fixbb
\path[connect] (cup.north west) .. controls ++(-1.5,0.1) and ++(-0.4,0.2) .. (a.north west) .. controls (a) .. (a.south west) .. controls ++(-0.3,-0.3) and ++(0.3,-0.3) .. (a.south east) .. controls (a) .. (a.north east) to[bend left] (b.north west) .. controls (b) .. (b.south west) .. controls ++(-0.3,-0.3) and ++(0.3,-0.3) .. (b.south east) .. controls (b) .. (b.north east) .. controls ++(0.4,0.2) and ++(1.5,0.1) .. (cup.north east);
\end{tikzpicture}
+
\begin{tikzpicture}[baseline=1.1cm]
\cupb{cup}(0,0)
\vertexvw{a}(0,1.2)
\vertexvw{b}(0,2.4)
\fixbb
\path[connect] (cup.north west) .. controls ++(-0.2,0.1) and ++(-1.1,-0.3) .. (b.north west) .. controls (b) .. (b.south west) to[bend right] (a.north west) .. controls (a) .. (a.south west) .. controls ++(-0.3,-0.3) and ++(0.3,-0.3) .. (a.south east) .. controls (a) .. (a.north east) to[bend right] (b.south east) .. controls (b) .. (b.north east) .. controls ++(1.1,-0.3) and ++(0.2,0.1) .. (cup.north east);
\end{tikzpicture}
\right)\\&
+\alpha^2\delta^3\left(
\begin{tikzpicture}[baseline=1.1cm]
\cupb{cup}(0,0)
\vertexhw{a}(-0.6,1.2)
\vertexhw{b}(0.6,1.2)
\fixbb
\path[connect] (cup.north west) .. controls ++(0,0.3) and ++(-0.3,-0.3) .. (a.south west) .. controls (a) .. (a.south east) to[bend right] (b.south west) .. controls (b) .. (b.south east) .. controls ++(0.3,-0.3) and ++(0,0.3) .. (cup.north east);
\path[connect] (a.north west) .. controls ++(-0.3,0.3) and ++(0.3,0.3) .. (a.north east) .. controls (a) .. (a.north west);
\path[connect] (b.north west) .. controls ++(-0.3,0.3) and ++(0.3,0.3) .. (b.north east) .. controls (b) .. (b.north west);
\end{tikzpicture}
+
\begin{tikzpicture}[baseline=1.1cm]
\cupb{cup}(0,0)
\vertexhw{a}(0,1.2)
\vertexhw{b}(0,2.4)
\fixbb
\path[connect,bend left] (cup.north west) to (a.south west) .. controls (a) .. (a.south east) to (cup.north east);
\path[connect,bend left] (a.north west) to (b.south west) .. controls (b) .. (b.south east) to (a.north east) .. controls (a) .. (a.north west);
\path[connect] (b.north west) .. controls ++(-0.3,0.3) and ++(0.3,0.3) .. (b.north east) .. controls (b) .. (b.north west);
\end{tikzpicture}
\right)
+\cdots
\end{align*}
%%%%%%%%%%%%%%%%%%%%%5
\vfill\eject
Similarly, we have:
%%%%%%%%%%%%%%%%%%%%%%
\begin{align*}
&\frac{1}{\alpha^2}\Big(\int x^2 d\tilde\nu_+(x)-\alpha\Big)
=\delta
\begin{tikzpicture}[baseline=0]
\cupcupb{cup}(0,0)
\path[connect] (cupa.north west) .. controls ++(0,0.5) and ++(0,0.5) .. (cupa.north east);
\path[connect] (cupb.north west) .. controls ++(0,0.5) and ++(0,0.5) .. (cupb.north east);
\fixbb
\end{tikzpicture}
+\delta^2
\begin{tikzpicture}[baseline=0]
\cupcupb{cup}(0,0)
\fixbb
\path[connect] (cupa.north west) .. controls ++(0,1) and ++(0,1) .. (cupb.north east) -- (cupb.north west) .. controls ++(0,0.4) and ++(0,0.4) .. (cupa.north east);
\end{tikzpicture}
\\& +\alpha\delta\Bigg(
\begin{tikzpicture}[baseline=0]
\cupcupb{cup}(0,0)
\vertexvw{vert}(0,1.4)
\fixbb
\path[connect] (cupa.north west) .. controls ++(0,1) and ++(-0.3,0.3) .. (vert.north west) .. controls (vert) .. (vert.south west) .. controls ++(-0.3,-0.3) and ++(0.3,-0.3) .. (vert.south east) .. controls (vert) .. (vert.north east) .. controls ++(0.3,0.3) and ++(0,1) .. (cupb.north east) -- (cupb.north west) .. controls ++(0,0.2) and ++(0,0.2) .. (cupa.north east);
\end{tikzpicture}
+
\begin{tikzpicture}[baseline=0]
\cupcupb{cup}(0,0)
\vertexvw{vert}(0,1.2)
\fixbb
\path[connect] (cupa.north west) .. controls ++(0,2.5) and ++(0,2.5) .. (cupb.north east) -- (cupb.north west) .. controls ++(0,0.1) and ++(0.1,-0.1) .. (vert.south east) .. controls (vert) .. (vert.north east) .. controls ++(0.3,0.3) and ++ (-0.3,0.3) .. (vert.north west) .. controls (vert) .. (vert.south west) .. controls ++(-0.1,-0.1) and ++(0,0.1) .. (cupa.north east);
\end{tikzpicture}
\Bigg)
\\
&+\alpha\delta^2\Bigg(
\begin{tikzpicture}[baseline=0]
\cupcupb{cup}(0,0)
\vertexvw{vert}(0,1.2)
\fixbb
\path[connect] (cupa.north west) .. controls ++(0,0.5) and ++(-0.3,0.3) .. (vert.north west) .. controls (vert) .. (vert.south west) .. controls ++(-0.1,-0.1) and ++(0,0.1) .. (cupa.north east);
\path[connect] (cupb.north east) .. controls ++(0,0.5) and ++(0.3,0.3) .. (vert.north east) .. controls (vert) .. (vert.south east) .. controls ++(0.1,-0.1) and ++(0,0.1) .. (cupb.north west);
\end{tikzpicture}
+
\begin{tikzpicture}[baseline=0]
\cupcupb{cup}(0,0)
\vertexhw{vert}(0,1.2)
\fixbb
\path[connect] (cupa.north west) .. controls ++(0,0.2) and ++(-0.3,-0.3) .. (vert.south west) .. controls (vert) .. (vert.south east) .. controls ++(0.3,-0.3) and ++(0,0.2) .. (cupb.north east) -- (cupb.north west) .. controls ++(0,0.2) and ++(0,0.2) .. (cupa.north east);
\path[connect] (vert.north west) .. controls ++(-0.3,0.3) and ++(0.3,0.3) .. (vert.north east) .. controls (vert) .. (vert.north west);
\end{tikzpicture}
+
\begin{tikzpicture}[baseline=0]
\cupcupb{cup}(0,0)
\vertexhw{vert}(0,1.4)
\fixbb
\path[connect] (cupa.north east) .. controls ++(-1,1) and ++(-0.2,0) .. (vert.north west) .. controls (vert) .. (vert.north east) .. controls ++(0.2,0) and ++(1,1) .. (cupb.north west) -- (cupb.north east) .. controls ++(0.6,2.7) and ++(-0.6,2.7) .. (cupa.north west);
\path[connect] (vert.south west) .. controls ++(-0.3,-0.3) and ++(0.3,-0.3) .. (vert.south east) .. controls (vert) .. (vert.south west);
\end{tikzpicture}
 \\ &+
\begin{tikzpicture}[baseline=0]
\cupcupb{cup}(0,0)
\vertexvw{vert}(0.6,1.2)
\fixbb
\path[connect] (cupa.north west) .. controls ++(0,0.5) and ++(0,0.5) .. (cupa.north east);
\path[connect] (cupb.north west) .. controls ++(-0.2,0.1) and ++(-1.1,-0.3) .. (vert.north west) .. controls (vert) .. (vert.south west) .. controls ++(-0.3,-0.3) and ++(0.3,-0.3) .. (vert.south east) .. controls (vert) .. (vert.north east) .. controls ++(1.1,-0.3) and ++(0.2,0.1) .. (cupb.north east);
\end{tikzpicture}
+
\begin{tikzpicture}[baseline=0]
\cupcupb{cup}(0,0)
\vertexvw{vert}(-0.6,1.2)
\fixbb
\path[connect] (cupb.north west) .. controls ++(0,0.5) and ++(0,0.5) .. (cupb.north east);
\path[connect] (cupa.north west) .. controls ++(-0.2,0.1) and ++(-1.1,-0.3) .. (vert.north west) .. controls (vert) .. (vert.south west) .. controls ++(-0.3,-0.3) and ++(0.3,-0.3) .. (vert.south east) .. controls (vert) .. (vert.north east) .. controls ++(1.1,-0.3) and ++(0.2,0.1) .. (cupa.north east);
\end{tikzpicture}
\Bigg) \\ &
+\alpha\delta^3\Bigg(
\begin{tikzpicture}[baseline=0]
\cupcupb{cup}(0,0)
\vertexhw{vert}(0.6,1.2)
\fixbb
\path[connect] (cupa.north west) .. controls ++(0,0.5) and ++(0,0.5) .. (cupa.north east);
\path[connect,bend left] (cupb.north west) to (vert.south west) .. controls (vert) .. (vert.south east) to (cupb.north east);
\path[connect] (vert.north west) .. controls ++(-0.3,0.3) and ++(0.3,0.3) .. (vert.north east) .. controls (vert) .. (vert.north west);
\end{tikzpicture}
+
\begin{tikzpicture}[baseline=0]
\cupcupb{cup}(0,0)
\vertexhw{vert}(-0.6,1.2)
\fixbb
\path[connect] (cupb.north west) .. controls ++(0,0.5) and ++(0,0.5) .. (cupb.north east);
\path[connect,bend left] (cupa.north west) to (vert.south west) .. controls (vert) .. (vert.south east) to (cupa.north east);
\path[connect] (vert.north west) .. controls ++(-0.3,0.3) and ++(0.3,0.3) .. (vert.north east) .. controls (vert) .. (vert.north west);
\end{tikzpicture}
\Bigg)
\\
&+\beta\delta\Bigg(
\begin{tikzpicture}[baseline=0]
\cupcupb{cup}(0,0)
\vertexhb{vert}(0,1.2)
\fixbb
\path[connect] (cupa.north west) .. controls ++(0,0.5) and ++(-0.3,0.3) .. (vert.north west) .. controls (vert) .. (vert.north east) .. controls ++(0.3,0.3) and ++(0,0.5) .. (cupb.north east) -- (cupb.north west) .. controls ++(0,0.1) and ++(0.1,-0.1) .. (vert.south east) .. controls (vert) .. (vert.south west) .. controls ++(-0.1,-0.1) and ++(0,0.1) .. (cupa.north east);
\end{tikzpicture}
+
\begin{tikzpicture}[baseline=0]
\cupcupb{cup}(0,0)
\vertexvb{vert}(0,1.2)
\fixbb
\path[connect] (cupa.north west) .. controls ++(0,0.2) and ++(-0.3,-0.3) .. (vert.south west) .. controls (vert) .. (vert.north west) .. controls ++(-0.3,0.3) and ++(0.3,0.3) .. (vert.north east) .. controls (vert) .. (vert.south east) .. controls ++(0.3,-0.3) and ++(0,0.2) .. (cupb.north east) -- (cupb.north west) .. controls ++(0,0.2) and ++(0,0.2) .. (cupa.north east);
\end{tikzpicture}
+
\begin{tikzpicture}[baseline=0]
\cupcupb{cup}(0,0)
\vertexvb{vert}(0,1.4)
\fixbb
\path[connect] (cupa.north east) .. controls ++(-1,1) and ++(-0.2,0) .. (vert.north west) .. controls (vert) .. (vert.south west) .. controls ++(-0.3,-0.3) and ++(0.3,-0.3) .. (vert.south east) .. controls (vert) .. (vert.north east) .. controls ++(0.2,0) and ++(1,1) .. (cupb.north west) -- (cupb.north east) .. controls ++(0.6,2.7) and ++(-0.6,2.7) .. (cupa.north west);
\end{tikzpicture}
\Bigg) \\
&+\beta\delta^2\Bigg(
\begin{tikzpicture}[baseline=0]
\cupcupb{cup}(0,0)
\vertexhb{vert}(0,1.4)
\fixbb
\path[connect] (cupa.north west) .. controls ++(0,1) and ++(-0.3,0.3) .. (vert.north west) .. controls (vert) .. (vert.north east) .. controls ++(0.3,0.3) and ++(0,1) .. (cupb.north east) -- (cupb.north west) .. controls ++(0,0.2) and ++(0,0.2) .. (cupa.north east) (vert.south west) .. controls ++(-0.3,-0.3) and ++(0.3,-0.3) .. (vert.south east) .. controls (vert) .. (vert.south west);
\end{tikzpicture}
+
\begin{tikzpicture}[baseline=0]
\cupcupb{cup}(0,0)
\vertexhb{vert}(0,1.2)
\fixbb
\path[connect] (cupa.north west) .. controls ++(0,2.5) and ++(0,2.5) .. (cupb.north east) -- (cupb.north west) .. controls ++(0,0.1) and ++(0.1,-0.1) .. (vert.south east) .. controls (vert) .. (vert.south west) .. controls ++(-0.1,-0.1) and ++(0,0.1) .. (cupa.north east) (vert.north east) .. controls ++(0.3,0.3) and ++ (-0.3,0.3) .. (vert.north west) .. controls (vert) .. (vert.north east);
\end{tikzpicture}
 +
\begin{tikzpicture}[baseline=0]
\cupcupb{cup}(0,0)
\vertexvb{vert}(0.6,1.2)
\fixbb
\path[connect] (cupa.north west) .. controls ++(0,0.5) and ++(0,0.5) .. (cupa.north east);
\path[connect,bend left] (cupb.north west) to (vert.south west) .. controls (vert) .. (vert.north west) .. controls ++(-0.3,0.3) and ++(0.3,0.3) .. (vert.north east) .. controls (vert) .. (vert.south east) to (cupb.north east);
\end{tikzpicture}
\\& +
\begin{tikzpicture}[baseline=0]
\cupcupb{cup}(0,0)
\vertexvb{vert}(-0.6,1.2)
\fixbb
\path[connect] (cupb.north west) .. controls ++(0,0.5) and ++(0,0.5) .. (cupb.north east);
\path[connect,bend left] (cupa.north west) to (vert.south west) .. controls (vert) .. (vert.north west) .. controls ++(-0.3,0.3) and ++(0.3,0.3) .. (vert.north east) .. controls (vert) .. (vert.south east) to (cupa.north east);
\end{tikzpicture}
\Bigg)
+\beta\delta^3\Bigg(
\begin{tikzpicture}[baseline=0]
\cupcupb{cup}(0,0)
\vertexhb{vert}(0.6,1.2)
\fixbb
\path[connect] (cupa.north west) .. controls ++(0,0.5) and ++(0,0.5) .. (cupa.north east);
\path[connect] (cupb.north west) .. controls ++(-0.2,0.1) and ++(-1.1,-0.3) .. (vert.north west) .. controls (vert) .. (vert.north east) .. controls ++(1.1,-0.3) and ++(0.2,0.1) .. (cupb.north east) (vert.south west) .. controls ++(-0.3,-0.3) and ++(0.3,-0.3) .. (vert.south east) .. controls (vert) .. (vert.south west);
\end{tikzpicture}
+
\begin{tikzpicture}[baseline=0]
\cupcupb{cup}(0,0)
\vertexhb{vert}(-0.6,1.2)
\fixbb
\path[connect] (cupb.north west) .. controls ++(0,0.5) and ++(0,0.5) .. (cupb.north east);
\path[connect] (cupa.north west) .. controls ++(-0.2,0.1) and ++(-1.1,-0.3) .. (vert.north west) .. controls (vert) .. (vert.north east) .. controls ++(1.1,-0.3) and ++(0.2,0.1) .. (cupa.north east) (vert.south west) .. controls ++(-0.3,-0.3) and ++(0.3,-0.3) .. (vert.south east) .. controls (vert) .. (vert.south west);
\end{tikzpicture}
\Bigg)+\cdots
\end{align*}
%%%%%%%%%%%%%
\vfill\eject
%%%%%%%%%%%%%
%

\section{The stitching of two planar algebras.}\label{sec:couplings}
% ${\mathcal{P}}_1 \copyrightsign {\mathcal{P}}_2$.}
Let ${\mathcal P}=P_n^{\pm}$ and ${\mathcal Q}=Q_n^{\pm}$ be two subfactor planar algebras
(assumed spherical for simplicity). We define a new subfactor planar algebra

$${\PP}\copyrightsign \Q=(P \copyrightsign Q)_n^{\pm}$$ which will not be irreducible even if both $\PP$ and $\Q$ are.

The vector spaces of $\PP\copyrightsign\Q$ are defined as follows. Fix the even number $2n$ and
consider the possible partititions $\pi$ of $\{1,2,3,...,2n\}$ into two subsets of even sizes $2(p_\pi)$ and 
$2(q_\pi)$, whose elements we call of type   $\PP$ and  $\Q$ respectively. Then
\begin{gather*}
(P \copyrightsign Q)_n^{\pm}=\bigoplus_{\pi}(P \copyrightsign Q)_\pi^{\pm} 
\textrm{    \qquad       with }(P \copyrightsign Q)_\pi^{\pm}= P_{p_\pi}^{\pm}\otimes Q_{q_\pi}^{\pm}.
\end{gather*}

To define the action of planar tangles on $\PP\copyrightsign\Q$ we use the notion
of a {\emph string-coloring}.
Given a planar tangle  $T$ with discs $D_i$ as in section 2.1, a string-colouring
$\sigma$ of $T$ is an assignment
of $\PP$ or $\Q$ to every string of $T$ so that if one removes all the strings of
either color, one gets
two planar tangles $T_{\PP}^\sigma$ and $T_{\Q}^\sigma$ when one takes as initial
segments, with shading, the
intervals containing the initial segments of $T$. In particular there must be an
even number of strings of 
each color incident to every disc of $T$, but this is not a sufficient condition for
the shadings to be coherent.

A string-coloring $\sigma$ of $T$ defines:

\renewcommand{\theenumi}{\alph{enumi}}
\begin{enumerate}
\item a partition $\pi_\sigma$ of the boundary points for each disc (numbered from $1$ to $2b_j$ starting at the first boundary point after the
initial segment in clockwise order)into  $\PP$ and $\Q$ points, and 
\item
two planar tangles $T_{\PP}^\sigma$ and $T_{\Q}^\sigma$ by removing all the
strings of the other color and taking as initial segments  for discs the ones containing the initial segments of $T$.
The shadings of $T_{\PP}^\sigma$ and $T_{\Q}^\sigma$ are determined by that of the initial segment of the outside boundary of $T$.
\end{enumerate}

By multilinearity it suffices to define the action $Z_T$ of a planar tangle $T$, with $k$ internal discs, on a $k-$tuple $(x_1,x_2,...,x_k)$
 of  elements of $\PP \copyrightsign \Q$ where  $x_i=\sum_\pi v_i^\pi \otimes w_i^\pi$ with $v_i\in P_{p_\pi}^{\pm}$ and $w_i Q_{q_\pi}^{\pm} $.
 
Suppose such an element $x_i$ of  $(P \copyrightsign  Q)_{b_i}^{\pm}$ is assigned to each $D_i$, then we define
$$M_T(x_1,x_2,...,x_k)= \sum_{\sigma}M_{T_\PP^\sigma}(v_1^{\pi_\sigma}, v_2^{\pi_\sigma}, ...,v_k^{\pi_\sigma})\otimes
M_{T_\Q^\sigma}(w_i^{\pi_\sigma}, w_2^{\pi_\sigma}, ...,w_k^{\pi_\sigma})$$
 where $\sigma$ runs over all the string-colorings of $T$. 
 
 It is clear that this action is compatible with the gluings.
 
 The $*$-structure on ${\PP} \copyrightsign \Q$ is derived in
 the obvious way from those of $\PP$ and $\Q$.
 
\subsection*{Notes}  (\romannumeral 1) Another way to describe the action is as follows: suppose elements  $v_i\otimes w_i \in (P \copyrightsign Q)_{\pi_i}^{\pm}$
 are assigned to the internal discs $D_i$ of $T$. Then the value of $M_T$ is zero unless the colouring of the boundary points
 implied by the $\pi_i$ extends to a string colouring  of $T$ and then this value is the sum over all such extensions $\sigma$ of
  $M_{T_\PP^\sigma}(v_1,v_2,...v_k)\otimes M_{T_\Q^\sigma}(w_1,w_2,...,w_k)$. If there are strings connecting the outside boundary
  of $T$ to itself this element will lie in more than one direct summand of $(P \copyrightsign Q)_n^{\pm}$.
  
  (\romannumeral 2) It is clear from (\romannumeral 1) that the loop parameter of ${\PP} \copyrightsign \Q$ is the \emph{sum} of the loop parameters
  of $\PP$ and $\Q$ respectively.
  
  (\romannumeral 3) Positive definiteness of the inner product is also clear from (\romannumeral 1)-the various $(P \copyrightsign Q)_\pi$ 
  are orthogonal and in each one the inner product is just the tensor product inner product for $P_{p_\pi}^{\pm}\otimes Q_{q_\pi}^{\pm}$.
  
  (\romannumeral 4) It is the exponential generating functions for $\PP$ and $\Q$ that behaves well under this operation.
  
  (\romannumeral 5) For the inductive limit algebra structure of ${\PP} \copyrightsign \Q$ a complete set of centrally orthogonal 
  minimal projections is given by the tensor products of such sets of projections for $\PP$ and $\Q$. Thus the vertices of the principal
  graph of ${\PP} \copyrightsign \Q$ is the Cartesian product of the vertices of the principal graphs of $\PP$ and $\Q$ and 
  $(p,q)$ is adjacent to $(p',q')$ iff $p$ is adjacent to $p'$ or $q$ is adjacent to $q'$ but not both.
  For instance if $\PP$ has principal graph $\mathbb{A}_3$ and $\Q$ has principal graph $\mathbb{A}_4$, then that of ${\PP} \copyrightsign \Q$ is:
  $$
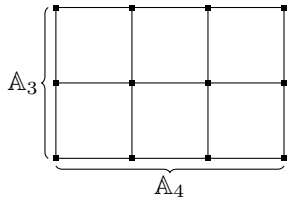
\begin{tikzpicture}
		\coordinate (A1) at (0,0);
		\coordinate (A2) at (1,0);
		\coordinate (A3) at (2,0);
		\coordinate (A4) at (3,0);
		\coordinate (B1) at (0,1);
		\coordinate (B2) at (1,1);
		\coordinate (B3) at (2,1);
		\coordinate (B4) at (3,1);
		\coordinate (C1) at (0,2);
		\coordinate (C2) at (1,2);
		\coordinate (C3) at (2,2);
		\coordinate (C4) at (3,2);
		\draw (A1) --(A2) -- (A3) -- (A4);
		\draw (B1) -- (B2) -- (B3) -- (B4);
		\draw (C1) -- (C2) -- (C3) -- (C4);
		\draw (A1) -- (B1) -- (C1);
		\draw (A2) -- (B2) -- (C2);
		\draw (A3) -- (B3) -- (C3);
		\draw (A4) -- (B4) -- (C4);
		\node [fill=black,inner sep=1pt] at (A1){};
		\node [fill=black,inner sep=1pt] at (A2){};
		\node [fill=black,inner sep=1pt] at (A3){};
		\node [fill=black,inner sep=1pt] at (A4){};
		\node [fill=black,inner sep=1pt] at (B1){};
		\node [fill=black,inner sep=1pt] at (B2){};
		\node [fill=black,inner sep=1pt] at (B3){};
		\node [fill=black,inner sep=1pt] at (B4){};
		\node [fill=black,inner sep=1pt] at (C1){};
		\node [fill=black,inner sep=1pt] at (C2){};
		\node [fill=black,inner sep=1pt] at (C3){};
		\node [fill=black,inner sep=1pt] at (C4){};
		\draw[decorate,decoration=brace] (3,-0.1)--(0,-0.1) node [below, midway]{$\mathbb{A}_4$};
		\draw[decorate,decoration=brace] (-0.1,0)--(-0.1,2) node [left, midway]{$\mathbb{A}_3$};

	\end{tikzpicture}
$$
  
  (\romannumeral 6) We see that in the loop basis description off ${\PP} \copyrightsign \Q$, 
 from each vertex of a loop one may choose to travel  on the principal graph of $\PP$ or that of  $\Q$.
    
  (\romannumeral 7) If we were dealing with unshaded planar algebras we would simply remove the
  restrictions on the parity of the numbers of boundary points and the partitions $\pi$.
  
  (\romannumeral 8) A TL basis diagram in   $\PP \copyrightsign\Q$ consists of a sum over all \emph{planar} partitions of the
  boundary points into $\PP$ points and $\Q$ points with that basis diagram regarded as a tensor product of 
  its $\PP$ part and its $\Q$ part.

%
%\renewcommand\MR[1]{\relax\ifhmode\unskip\spacefactor3000 \space\fi
%  \MRhref{#1}{{\sc mr}}}
%\gdef\MRSIMPLIFY#1 #2MREND{#1}%
%\renewcommand{\MRhref}[2]{%
% \href{http://www.ams.org/mathscinet-getitem?mr=\MRSIMPLIFY#1 MREND}{#2}}
%\bibliography{matmod}
%\bibliographystyle{amsplainhyper}

\bibliographystyle{amsalpha}
\providecommand{\bysame}{\leavevmode\hbox to3em{\hrulefill}\thinspace}
\providecommand{\MR}{\relax\ifhmode\unskip\space\fi MR }
\providecommand{\MRhref}[2]{%
  \href{http://www.ams.org/mathscinet-getitem?mr=#1}{#2}
}
\providecommand{\href}[2]{#2}

\end{document}